\documentclass[11pt,reqno]{amsart}
\usepackage[margin=1in]{geometry}

\usepackage{amssymb}
\usepackage{graphicx} 
\usepackage{mathrsfs}
\usepackage{enumerate}
\usepackage{xspace}
\usepackage{color}
\usepackage{enumerate}
\usepackage{enumitem}
\usepackage{amsthm}
\usepackage[T1]{fontenc}
\usepackage{lmodern}
\usepackage{dsfont}
\usepackage{bbm}
\usepackage{commath}
\usepackage[colorlinks=true,linkcolor=blue]{hyperref}
\usepackage[numbers]{natbib}
\usepackage[backgroundcolor=white,bordercolor=red]{todonotes}
\DeclareMathAlphabet{\mathpzc}{OT1}{pzc}{m}{it}
\usepackage[nameinlink]{cleveref}
\numberwithin{equation}{section}
\begin{document}

\expandafter\let\expandafter\oldproof\csname\string\proof\endcsname
\let\oldendproof\endproof
\renewenvironment{proof}[1][\proofname]{%
	\oldproof[\normalfont\scshape\hspace{1em}#1]%
}{\oldendproof}

\newtheoremstyle{mystyle_thm}
{6pt}
{6pt}
{\itshape}
{1em}
{\scshape}
{.}
{.5em}
{}%

\newtheoremstyle{mystyle_def}
{6pt}
{6pt}
{}
{1em}
{\scshape}
{.}
{.5em}
{}%

\theoremstyle{mystyle_thm}

\newtheorem{theorem}{Theorem}[section]
\newtheorem{lemma}[theorem]{Lemma}
\newtheorem{proposition}[theorem]{Proposition}
\newtheorem{corollary}[theorem]{Corollary}
\newtheorem{definition}[theorem]{Definition}
\newtheorem{Ass}[theorem]{Assumption}
\newtheorem{condition}[theorem]{Condition}

\theoremstyle{mystyle_def}

\newtheorem{example}[theorem]{Example}
\newtheorem{remark}[theorem]{Remark}
\newtheorem{SA}[theorem]{Standing Assumption}
\newtheorem{discussion}[theorem]{Discussion}
\newtheorem{remarks}[theorem]{Remark}
\newtheorem*{notation}{Remark on Notation}
\newtheorem{application}[theorem]{Application}

\newcommand{\of}{[\hspace{-0.06cm}[}
\newcommand{\gs}{]\hspace{-0.06cm}]}

\newcommand{\llambda}{{\mathchoice
		{\lambda\mkern-4.5mu{\raisebox{.4ex}{\scriptsize$\setminus$}}}
		{\lambda\mkern-4.83mu{\raisebox{.4ex}{\scriptsize$\setminus$}}}
		{\lambda\mkern-4.5mu{\raisebox{.2ex}{\footnotesize$\scriptscriptstyle\setminus$}}}
		{\lambda\mkern-5.0mu{\raisebox{.2ex}{\tiny$\scriptscriptstyle\setminus$}}}}}

\newcommand{\1}{\mathds{1}}

\newcommand{\F}{\mathbf{F}}
\newcommand{\G}{\mathbf{G}}

\newcommand{\B}{\mathbf{B}}

\newcommand{\M}{\mathcal{M}}

\newcommand{\la}{\langle}
\newcommand{\ra}{\rangle}

\newcommand{\lle}{\langle\hspace{-0.085cm}\langle}
\newcommand{\rre}{\rangle\hspace{-0.085cm}\rangle}
\newcommand{\blle}{\Big\langle\hspace{-0.155cm}\Big\langle}
\newcommand{\brre}{\Big\rangle\hspace{-0.155cm}\Big\rangle}

\newcommand{\X}{\mathsf{X}}

\newcommand{\tr}{\operatorname{tr}}
\newcommand{\N}{{\mathbb{N}}}
\newcommand{\cadlag}{c\`adl\`ag }
\newcommand{\on}{\operatorname}
\newcommand{\oP}{\overline{P}}
\newcommand{\oO}{\mathcal{O}}
\newcommand{\D}{\mathsf{D}} 
\newcommand{\bx}{\mathsf{x}}
\newcommand{\bb}{\hat{b}}
\newcommand{\bs}{\hat{\sigma}}
\newcommand{\bv}{\hat{v}}
\newcommand{\ob}{\bar{b}}
\newcommand{\oa}{\bar{a}}
\newcommand{\os}{\widehat{\sigma}}
\renewcommand{\j}{\varkappa}
\newcommand{\scl}{\ell}
\newcommand{\Y}{\mathscr{Y}}
\newcommand{\Z}{\mathscr{Z}}
\newcommand{\T}{\mathcal{T}}
\newcommand{\con}{\mathsf{c}}
\newcommand{\nk}{\hspace{-0.25cm}{{\phantom A}_k^n}}
\newcommand{\nl}{\hspace{-0.25cm}{{\phantom A}_1^n}}
\newcommand{\nm}{\hspace{-0.25cm}{{\phantom A}_2^n}}
\newcommand{\n}{\hspace{-0.35cm}{\phantom {Y_s}}^n}
\newcommand{\nme}{\hspace{-0.35cm}{\phantom {Y_s}}^{n - 1}}
\renewcommand{\o}{\hspace{-0.35cm}\phantom {Y_s}^0}
\newcommand{\e}{\hspace{-0.4cm}\phantom {U_s}^1}
\newcommand{\z}{\hspace{-0.4cm}\phantom {U_s}^2}
\newcommand{\iii}{|\hspace{-0.05cm}|\hspace{-0.05cm}|}
\newcommand{\co}{\overline{\on{co}}}
\newcommand{\ovb}{\overline{b}}
\newcommand{\ova}{\overline{a}}
\newcommand{\s}{\mathfrak{s}}
\newcommand{\opsi}{\overline{\Psi}}
\newcommand{\ol}{\mathcal{L}}
\newcommand{\cW}{\mathscr{W}}
\newcommand{\cU}{\mathcal{U}}
\newcommand{\oD}{\overline{D}}
\newcommand{\ua}{\underline{a}}
\newcommand{\ou}{\overline{b}}
\newcommand{\uu}{\underline{b}}

\renewcommand{\epsilon}{\varepsilon}

\newcommand{\fPs}{\fP_{\textup{sem}}}
\newcommand{\fPas}{\mathfrak{S}_{\textup{ac}}}
\newcommand{\rrarrow}{\twoheadrightarrow}
\newcommand{\cA}{\mathcal{A}}
\newcommand{\ocA}{\mathcal{U}}
\newcommand{\cR}{\mathcal{K}}
\newcommand{\cK}{\mathcal{K}}
\newcommand{\cQ}{\mathcal{Q}}
\newcommand{\cF}{\mathcal{F}}
\newcommand{\cE}{\mathcal{E}}
\newcommand{\cC}{\mathcal{C}}
\newcommand{\cD}{\mathcal{D}}
\newcommand{\bC}{\mathbb{C}}
\newcommand{\cH}{\mathcal{H}}
\newcommand{\bth}{\overset{\leftarrow}\theta}
\renewcommand{\th}{\theta}
\newcommand{\cG}{\mathcal{G}}
\newcommand{\fPasn}{\mathfrak{S}^{\textup{ac}, n}_{\textup{sem}}}
\newcommand{\CLM}{\mathfrak{M}^\textup{ac}_\textup{loc}}
\newcommand{\Sd}{\mathcal{S}^\textup{sp}_{\textup{d}}}
\newcommand{\Sc}{\mathcal{S}}
\newcommand{\Sac}{\mathcal{S}_\textup{ac}}
\newcommand{\A}{\mathsf{A}}
\newcommand{\Td}{\mathsf{T}^\textup{d}}
\renewcommand{\t}{\mathfrak{t}}

\newcommand{\bR}{\mathbb{R}}
\newcommand{\nnabla}{\nabla}
\newcommand{\f}{\mathfrak{f}}
\newcommand{\g}{\mathfrak{g}}
\newcommand{\oconv}{\overline{\on{co}}\hspace{0.075cm}}
\renewcommand{\a}{\mathfrak{a}}
\renewcommand{\b}{\mathfrak{b}}
\renewcommand{\d}{\mathsf{d}}
\newcommand{\bS}{\mathbb{S}^d_+}
\newcommand{\p}{\mathsf{p}}
\newcommand{\dr}{r} 
\newcommand{\m}{\mathbb{M}}
\newcommand{\Q}{Q}
\newcommand{\usc}{\textit{USC}}
\newcommand{\lsc}{\textit{LSC}}
\newcommand{\q}{\mathfrak{q}}
\renewcommand{\X}{\mathscr{X}}
\newcommand{\W}{\mathscr{W}}
\newcommand{\fP}{\mathcal{P}}
\newcommand{\w}{\mathsf{w}}
\newcommand{\oM}{\mathsf{M}}
\newcommand{\oZ}{\mathsf{Z}}
\newcommand{\oK}{\mathsf{K}}
\renewcommand{\Re}{\operatorname{Re}}
\newcommand{\cCk}{\mathsf{c}_k}
\newcommand{\C}{\mathsf{C}}
\newcommand{\cP}{\mathcal{P}}
\newcommand{\oPi}{\overline{\Pi}}
\newcommand{\cI}{\mathcal{I}}

\renewcommand{\emptyset}{\varnothing}

\allowdisplaybreaks

\makeatletter
\@namedef{subjclassname@2020}{%
	\textup{2020} Mathematics Subject Classification}
\makeatother

 \title[Stochastic control with measurable coefficients]{Stochastic optimal control problems with measurable coefficients and unbounded drift} 
\author[D. Criens]{David Criens}
\address{University of Freiburg, Ernst-Zermelo-Str. 1, 79104 Freiburg, Germany}
\email{david.criens@stochastik.uni-freiburg.de}



\thanks{}
\date{\today}

\maketitle

\begin{abstract}
We consider controlled stochastic differential equations (SDEs) with measurable coefficients, a uniformly elliptic diffusion coefficient and an \(L_d\)-drift. No space-regularity will be assumed for the coefficients. 
In this framework we investigate the relationship among value functions, partial differential equations (PDEs), and operator semigroups.
First, for a cost with an infinite time horizon on a bounded domain, we identify the value function as an \(L_{d_0}\)-viscosity solution to a Hamilton--Jacobi--Bellman equation and we establish quantitative regularity estimates. The constant \(d_0 \in (d/2, d)\) only depends on the space dimension \(d\), the ellipticity constants of the diffusion coefficient and the \(L_d\)-bound of the drift. To illustrate applications of these results, we provide a uniqueness theorem under an additional assumption on the diffusion coefficient, showing a stochastic representation, and we discuss stability of value functions. Second, we consider a cost with a finite time horizon, terminal and running terms, assuming an \(L_{d + 1}\)-drift. We show that the value function indexed over the terminal cost is a nonlinear semigroup on \(C_b (\bR^d)\) and we establish a regularization by noise effect, which shows that the semigroup regularizes lower semicontinuity to local H\"older continuity. We also establish quantitative regularity estimates for the semigroup that are local in time and global in space. Lastly, under an additional \(L_{d + 2, \on{loc}}\)-assumption on the drift, we relate the semigroup to a parabolic PDE, showing that the associated value function is an \(L_{d + 1}\)-viscosity solution. Our proofs of the regularity of the value functions, the \(C_b\)-Feller property of the semigroup and its regularization by noise effects are based on a strong Markov selection principle and analytic estimates for linear diffusions that were recently established by N.~V.~Krylov in a series of papers. We highlight that our method covers frameworks without uniqueness of the controlled SDEs, as well as the associated PDEs. 
\end{abstract}

\section{Introduction} 
In the theory of stochastic control, it is of fundamental interest to establish a precise connection between the value function, the associated Hamilton--Jacobi--Bellman (HJB) partial differential equation (PDE), and the nonlinear semigroup generated by the control problem. Such connections are well understood in the context of controlled stochastic differential equations (SDEs) with sufficiently regular coefficients of suitable growth, see, e.g., \cite{FS_06, krylov_80, Lions85_2, NL_82, N_15}. The main objective of this paper is to establish these relationships in a framework allowing for an unbounded drift and without assuming regularity in the space variable of the coefficients, opening the door for applications of the general theory of viscosity solutions to a broad class of stochastic control problems. 
To make our setting precise, we consider controlled \(\bR^d\)-valued SDEs of the type
\begin{align*}
	d X_t = b (\lambda_t, X_t) \, dt + \sqrt{a (\lambda_t, X_t)} \, d W_t,
\end{align*} 
with controls \((\lambda_t)_{t \geq 0}\) taking values in a compact Polish space \(\Lambda\), formalized in the relaxed control framework, under the following assumptions on the coefficients:
\begin{enumerate}
			\item[(a)] \(a, b\) are Borel functions and, for every compact set \(K \subset \bR^d\), there exists a modulus of continuity \(w_{a, b, K} \colon \bR_+ \to \bR_+\) and a non-negative Borel function \(\mathfrak{g}_K \in L_d (K)\) such that 
\[
\| a (\lambda, x) - a (\lambda', x) \| + \| b (\lambda, x) - b (\lambda', x) \| \leq \mathfrak{g}_K (x) \, w_{a, b, K} \big(d_\Lambda (\lambda, \lambda') \big) 
\] 
for all \(\lambda, \lambda' \in \Lambda\) and \(x \in K\). Here, \(d_\Lambda\) denotes a metric on \(\Lambda\) that induces its topology. 
	\item[(b)] There exists a \(\delta \in (0, 1)\) such that \(\delta \|\xi\|^2 \leq \langle \xi, a (\lambda, x) \xi \rangle \leq \|\xi\|^2 / \delta\) for all \(\lambda \in \Lambda\) and \(x, \xi \in \bR^d\).
	\item[(c)] There exists a \(\b \in L_d (\bR^d)\) such that \(\| b (\lambda, x) \| \leq \b (x)\) for all \(\lambda \in \Lambda\) and \(x \in \bR^d\).
\end{enumerate}
We impose no regularity assumption on \(x \mapsto (b (\lambda, x), a (\lambda, x))\), allow for arbitrary dimension \(d \geq 2\), and consider a drift of \(L_d\)-type, which is a novelty in this context.
Within this control framework, we study value functions with infinite and finite time horizons. 
Let us explain our contributions in more detail. 

We start by considering an infinite time horizon with a bounded domain \(D\), satisfying the exterior ball condition, and the value function 
\[
v (x) := \inf_\lambda E \Big[ f (X_{\tau_D}) + \int_0^{\tau_D} g (\lambda_s, X_s) \, ds \mid X_0 = x \Big], \quad x \in \oD = D \cup \partial D, 
\] 
with \(\tau_D := \inf \{t \geq 0 \colon X_t \not \in D\}\) and bounded lower semicontinuous input functions \(f, g\), where \(g\) is also continuous in the control variable. 
In classical frameworks with sufficiently regular coefficients, it is well-known (see, e.g., \cite{FS_06, Lions85_2}) that \(v\) is a classical viscosity solution to the Hamilton--Jacobi--Bellman Dirichlet problem
\begin{align} \label{eq: intro PDE}
	- \inf_\lambda \big\{ \tfrac{1}{2} \on{tr} \big[ a (\lambda, x) \nabla^2 u (x) \big] + \langle b (\lambda, x), \nabla u (x) \rangle + g (\lambda, x) \big\} = 0, \quad u = f \text{ on } \partial D. 
\end{align}
We are interested in a comparable connection for our framework, using the concept of \(L_p\)-viscosity solutions from \cite{CCKS_96, CKS_00}, which is tailored to settings with measurable coefficients. In contrast to classical viscosity solutions, \(L_p\)-viscosity solutions are defined using the broader class of Sobolev test functions. 
Our first main result shows that there exists a constant \(d_0 \in (d/2, d)\), depending only on \(d, \delta\) and \(\| \b\|_{L_d (\bR^d)}\), such that for continuous~\(f\) and all \(p \geq d_0\), the value function \(v\) is an \(L_p\)-viscosity solution to \eqref{eq: intro PDE}.
We further show that, even when \(f\) is only lower semicontinuous, the value function \(v\) is locally Hölder continuous in the open set~\(D\). This phenomenon can be interpreted as a form of regularization by noise. Using recent results by S.~Koike and A.~{\'S}wi{\k{e}}ch~\cite{KS_22} and N.~V.~Krylov \cite{krylov_CPDE_20}, under the uniform exterior ball condition on~\(D\) and an additional assumption on \(a\), we also provide a uniqueness theorem for \(L_p\)-viscosity solutions that identifies the value function as the unique \(L_p\)-viscosity solution. We interpret this result as a stochastic representation theorem. Finally, again under the uniform exterior ball condition, we establish an explicit modulus of continuity for the value function~\(v\), expressed in terms of the input data of our framework and the modulus of continuity of \(f\). In particular, if \(f\) is H\"older continuous, then \(v\) is H\"older continuous on~\(\oD\), with an exponent depending only on \(d, \delta, \|\b\|_{L_d (\bR^d)}\) and the H\"older exponent of~\(f\). 
As an application of these quantitative regularity estimates, we also establish a stability result for stochastic optimal control problems.

Next, we study value functions over a finite time horizon.
As demonstrated in the seminal work of M.~Nisio \cite{N_76a,N_76b}, control problems (with sufficiently regular coefficients) exhibit a fundamental connection to nonlinear semigroups. We are interested in the semigroup structure for our control framework with measurable coefficients and integrable, but possibly unbounded, drift. Let \((S_t)_{t \geq 0}\) be the nonlinear operators defined by
\[
S_t (f) (x) := \inf_\lambda E \Big[ f (X_t) + \int_0^t g (\lambda_s, X_s) \, ds \mid X_0 = x \Big], \quad f \text{ bounded, measurable}, 
\] 
where \(g\) is a given bounded, lower semicontinuous function that is continuous in the control variable.
Assuming \(\b \in L_{d + 1} (\bR^d)\), we show that \((S_t)_{t \geq 0}\) is a nonlinear semigroup on the space \(C_b (\bR^d)\), i.e., it satisfies the \(C_b\)-Feller property \(S_t (C_b (\bR^d)) \subset C_b (\bR^d)\) and the semigroup property \(S_{t + s} = S_t \circ S_s\) on \(C_b (\bR^d)\). We also establish a variant of the strong Feller property: for \(t > 0\), the operator \(S_t\) maps bounded lower semicontinuous functions to locally Hölder continuous functions, exhibiting a regularization by noise effect. 
Further, under the additional assumption \(\b \in L_{d + 2, \on{loc}} (\bR^d)\), we take a step towards an infinitesimal description and show, for any finite time horizon \(T > 0\), that  \(u (t, x) := S_{T - t} (f) (x)\) is an \(L_{d + 1}\)-viscosity solution to the backward PDE
\begin{align} \label{eq: parabolic PDE intro}
	\begin{cases}- \partial_t u -  \inf_\lambda \big\{ \tfrac{1}{2} \on{tr} [ a (\lambda, \cdot \, ) \nabla^2 u ] + \langle b (\lambda, \cdot \,), \nabla u \rangle + g (\lambda, \cdot \,) \big\} = 0, & \text{on } [0, T) \times \bR^d, \\ u (T, \cdot \,) = f, & \text{on } \bR^d, \end{cases}
\end{align} 
which can be viewed as the generator equation for the semigroup \((S_t)_{t \geq 0}\); cf. \cite{CN22b,hol16,NL_82}. As in the infinite time horizon case, we prove that uniform continuity is propagated from the terminal function \(f\) to the value function \(u \colon [0, T] \times \bR^d \to \bR\) and we describe its modulus of continuity in terms of the variables in our framework and the modulus of continuity of \(f\). In particular, if \(f\) is H\"older continuous, then \(u\) is H\"older continuous on \([0, T] \times \bR^d\) with a H\"older exponent that only depends on \(d, \delta\) and the H\"older exponent of \(f\), see Theorem~\ref{theo: joint regu semigroup}~(c).

Summarizing, our main results establish that value functions are \(L_p\)-viscosity solutions of the associated HJB equations with unbounded first order terms, without imposing regularity assumptions on the coefficients. In doing so, we resolve the existence problem for this class of equations, see~\cite{CKLS99} for a general existence result covering Isaacs' equations with bounded, measurable coefficients. We also provide quantitative regularity estimates, a uniqueness theorem, and investigate stability of value functions. General \(L_p\)-viscosity theory for equations with unbounded coefficients is currently undergoing active development and important results, such as weak Harnack and maximum principles, have recently been established, see, e.g., \cite{KS_09, KS_12, KS_22, KST_19, krylov_CPDE_20, N_09, N_19}. Our main results pave the way for applying this general theory to stochastic control problems.

Let us now comment on the methods of proof.
The core of the analysis concerns the regularity of the value function and the Feller properties of the associated semigroup. To address this, we linearize the control problem via a strong Markov selection and study regularity within this linearized setting, relying on analytic estimates recently developed by N.~V.~Krylov in a series of impressive papers \cite{krylov_PTRF_21, krylov_AOP_21, krylov_21, krylov_PA_23}.
The idea of obtaining regularity through a strong Markov selection was proposed in the recent work \cite{CN22b} within a one-dimensional framework of model uncertainty, where it was combined with results from \cite{SV} on one-dimensional martingale problems that, in particular, always satisfy uniqueness. In the present work, we develop this approach for our multidimensional relaxed control framework and strengthen it with analytic methods that remain applicable even in the absence of uniqueness. Moreover, we consider cost functions that depend on both the state and the control, an aspect that could not be addressed within the framework of \cite{CN22b}. As a byproduct of our analysis, we establish a strong Markov selection principle, which yields the existence of strong Markov optimal controls. We believe this result is new for our framework and of independent interest. 
Its proof builds upon pioneering work of N.~V.~Krylov \cite{krylov_selection}, see also \cite{CN22b, nicole1987compactification,Haus86,HausLep90} for approaches in control settings.
The main novel challenges we encounter concern compactness and continuity properties for the relaxed control framework with \(L_d\)-drift, the treatment of the stopping time \(\tau_D\), as well as the existence of a time-homogeneous selection.

Let us take the opportunity to briefly discuss two other strategies to relate value functions and PDEs. The first relies on establishing the existence of a Sobolev solution, which is then identified as the value function via a verification argument. The second is based on the theory of discontinuous viscosity solutions and a comparison principle: one studies the upper and lower envelopes of the value function and shows that they are (classical) viscosity sub- and supersolutions. A comparison principle for such discontinuous solutions then yields that the value function itself is a viscosity solution, with continuity as a consequence. A key distinction between these strategies and our method lies in their reliance on uniqueness properties. In the verification approach, it is known from~\cite{CCKS_96,KS_22} that the mere existence of a Sobolev solution already guarantees uniqueness of \(L_p\)-viscosity solutions, whereas in the comparison approach, uniqueness follows as a direct consequence of the comparison principle. In contrast, our method does not entail uniqueness and remains applicable even in situations where uniqueness may fail, as illustrated by N.~Nadirashvili’s \cite{zbMATH01140123} famous example (here, we recall that ``good solutions'' as considered in \cite{zbMATH01140123} are always \(L_d\)-viscosity solutions, see \cite{CCKS_96, JKS_02, J_96}). 
At this point, let us mention that our analysis provides some new insights for the linear situation without uniqueness. As we explain in Discussion~\ref{diss: 1 MR}~(b) below, optimization over solutions to SDEs provides the extremal solutions of linear Dirichlet problems that lack uniqueness. 
It is also worth noting that many uniqueness results in the literature rely on certain regularity assumptions for the diffusion coefficient or on bounds for its oscillation, see, e.g., \cite{CCKS_96, CKS_00, Jen_Sw_05, krylov_18, krylov_CPDE_20, W_09}. We impose no such conditions, and our results therefore remain applicable even in settings with highly irregular coefficients.

Lastly, let us discuss related literature. To the best of our knowledge, controlled diffusions with measurable coefficients and integrable drift have not been studied before. Controlled SDEs with measurable coefficients have been investigated, for instance, in \cite{F_25, nicole1987compactification, MPT_23}. Under conditions comparable to those used in this paper, but with bounded drift coefficient \(b\), the seminal paper~\cite[Section~7]{nicole1987compactification} establishes a strong Markov selection principle and lower semicontinuity of the value function \((t, x) \mapsto S_t (f) (x)\).
The method for the lower semicontinuity is based on Berge's maximum theorem and upper hemicontinuity of a set-valued map associated to relaxed control rules. Our strategy reveals more regularity of \((t, x) \mapsto S_t (f) (x)\). 
The paper~\cite{MPT_23} studies the stochastic maximum principle in a setting with additive noise and bounded measurable drift. 
Finally, the recent paper \cite{F_25} investigates \(L_p\)-viscosity and Sobolev solutions for the HJB equation related to the value function 
\begin{align}\label{eq: infinite time horitin value}
	x \mapsto \inf_\lambda E \Big[ \int_0^\infty e^{- \rho s} g (\lambda_s, X_s) \, ds \mid X_0 = x \Big],
\end{align} 
using the verification approach outlined above. 
In Discussion~\ref{diss: 1 MR}~(d) below, we explain that our method also applies to this value function, providing some new conditions for the relation of \eqref{eq: infinite time horitin value} to an HJB equation.
The semigroup connection for controlled SDEs was first revealed by M.~Nisio \cite{N_76a,N_76b}. More recently, this connection was further investigated in \cite{BDLM_25, C_24_JMAA, CN22b, CN_25_EJP, FN_26, GNR_24, hol16}. 
A version of the strong Feller property of nonlinear semigroups has been introduced in \cite{CN22b}. 
Our results seem to be the first for controlled SDEs with merely measurable coefficients and unbounded first order terms.

The paper is organized as follows. In Section~\ref{sec: control} we introduce our control framework. The main results for the elliptic setting are given in Section~\ref{sec: elliptic} and the semigroup connection is studied in Section~\ref{sec: parabolic}. The remaining Sections \ref{sec: pf SMS}--\ref{sec: pf parabolic viscosity} are devoted to the proofs of our main results. We highlight the strong Markov selection principle that is given in Section~\ref{sec: SMSP}. Finally, the paper concludes with Appendix~\ref{app: recap}, which provides a short recap of the concept of \(L_p\)-viscosity solutions, summarizing a variety of results from \cite{KS_22,KS_26} that are used in the main body of the paper.

\section{Value functions, HJB equations and nonlinear semigroups} \label{sec: main1}

\subsection{The control framework} \label{sec: control}

Let \(\Lambda\) be a compact Polish space and take a spatial dimension \(d \geq 2\).
The set of all symmetric positive semidefinite \(d \times d\) matrices with real entries is denoted by \(\bS\). The coefficients of the control problem under investigation are the two functions 
\begin{align*}
	a &\colon\Lambda \times \bR^d \to \bS, \\
	b &\colon \Lambda \times \bR^d \to \bR^d.
\end{align*}

Throughout this paper, we impose the following assumptions on the coefficients. 

\begin{SA} \label{SA: main1}
	\begin{enumerate}
			\item[(a)] \(a, b\) are Borel functions and, for every compact set \(K \subset \bR^d\), there exists a modulus of continuity \(w_{a, b, K} \colon \bR_+ \to \bR_+\) and a non-negative Borel function \(\mathfrak{g}_K \in L_d (K)\) such that 
		\[
		\| a (\lambda, x) - a (\lambda', x) \| + \| b (\lambda, x) - b (\lambda', x) \| \leq \mathfrak{g}_K (x) \, w_{a, b, K} \big(d_\Lambda (\lambda, \lambda') \big) 
		\] 
		for all \(\lambda, \lambda' \in \Lambda\) and \(x \in K\). Here, \(d_\Lambda\) denotes a metric on \(\Lambda\) that induces its topology. 
		\item[(b)] \(a\) is bounded and uniformly elliptic in the sense that there exists a \(\delta \in (0, 1)\) such that
		\[
		a (\lambda, x) \in \mathbb{S}_\delta := \Big\{ A \in \mathbb{S}^d_+ \colon \delta \, \|\xi\|^2 \leq \langle \xi, A \xi \rangle \leq \frac{\|\xi\|^2}{\delta}, \ \ \forall \, \xi \in \bR^d\Big\}
		\] 
		for all \((\lambda, x) \in \Lambda \times \bR^d\).
		\item[(c)] There exists a \(\b \in L_d (\bR^d)\) such that \(\|b (\lambda, x)\| \leq \b (x)\) for all \((\lambda, x) \in \Lambda \times \bR^d\). 
	\end{enumerate}
\end{SA}

We now introduce our control setup, using the so-called relaxed control framework (\cite{EKNJ88,nicole1987compactification, ElKa15, HausLep90}). For a comparison with the {\em weak} and {\em strong} control frameworks, we refer to the articles \cite{nicole1987compactification,ElKa15}.

Let \(\Omega\) be the space of continuous functions \(\omega \colon \bR_+ \to \bR^d\) and endow it with the local uniform topology. The coordinate process on \(\Omega\) is denoted by \(X\), i.e., \(X_t (\omega) = \omega (t)\) for \(\omega \in \Omega\) and \(t \in \bR_+\). Further, \(\cF\) and \((\cF_t)_{t \geq 0}\) are the canonical \(\sigma\)-field and filtration generated by \(X\), i.e., \(\cF := \sigma (X_t \colon t \geq 0)\) and \(\cF_t := \sigma (X_s \colon s \leq t)\).

Let \(\m\) be the set of all Radon measures on \(\bR_+ \times \Lambda\) whose projections to \(\bR_+\) coincide with the Lebesgue measure. 
We endow \(\m\) with the vague topology, which turns it into a compact Polish space (\cite[Theorem 2.2]{EKNJ88}). The identity map on \(\m\) is denoted by \(M\), i.e.,
\[
M (ds, d\lambda) (m) := m (ds, d \lambda), \quad m \in \m.
\]
Define the \(\sigma\)-field \[\M := \sigma (M_t (\phi) \colon t \in \bR_+, \phi \in C_{c} (\bR_+ \times \Lambda)),\] where
\[
M_t (\phi) := \int_0^t \int \phi (s, \lambda) \, M(ds, d\lambda).
\]
At this point, we emphasize that \((M_t (\phi))_{t \geq 0}\) is a stochastic process on the space \(\m\).
In the following, we work with the product space \((\Omega \times \m, \cF \otimes \M)\). Extending the notation above, we denote the coordinate process on this space by \((X, M)\). 

For \(\varphi \in C^2_b (\bR^d)\), define
\begin{align}\label{eq: operator L}
	L (\lambda, x, \nabla \varphi (x), \nabla^2 \varphi (x)) &:= \tfrac{1}{2} \on{tr} \big[ a (\lambda, x) \nabla^2 \varphi (x) \big] + \langle b (\lambda, x), \nabla \varphi (x) \rangle, \quad x \in \bR^d,
\end{align}
and 
\begin{align} \label{eq: test process}
	C_t (\varphi) := \varphi (X_t) - \int_{0}^{t} \int  L (\lambda, X_s, \nabla \varphi (X_s), \nabla^2 \varphi (X_s)) \, M (ds , d \lambda), \quad t \in \bR_+.
\end{align} 
A \emph{relaxed control rule} with initial value \((t, x) \in \bR_+ \times \bR^d\) is a probability measure \(P\) on the product space \((\Omega \times \m, \mathcal{F} \otimes \M)\) with \(P (X = x \text{ on } [0, t]) = 1\) such that the processes \((C_s (\varphi))_{s \geq t}\), with \(\varphi \in C^2_b (\bR^d)\), are \(P\)-martingales for the filtration
\[
\cG_s := \sigma (X_r, M_r (\phi) \colon r \leq s, \phi \in C_c (\bR_+ \times \Lambda)), \quad s \in \bR_+.
\]
Finally, for \((t, x) \in \bR_+ \times \bR^d\), we define 
\begin{align*}
	\cK (t, x)&:= \big\{ \text{all relaxed control rules with initial value } (t, x) \big \}, \quad
	\cR (x) := \cK (0, x).
\end{align*}

\begin{remark} \label{rem: Ito representation}
	Let \((t, x) \in \bR_+ \times \bR^d\) and \(P \in \cK (t, x)\). Then, \((X_s)_{s \geq t}\) is a \(P\)-\((\mathcal{G}_s)_{s \geq t}\)-semimartingale with drift \(\int_t^{\cdot} \int b (\lambda, X_r) \, M (dr, d\lambda)\) and quadratic variation \(\int_t^{\cdot} \int a (\lambda, X_r) \, M (dr, d\lambda)\), see, for instance, \cite[Theorem~II.2.42]{JS} or \cite[Proposition~5.4.6]{KaraShre}. Here, \(M (dr, d\lambda) = \mathsf{m}_r (d \lambda) \, dr\) for a \((\mathcal{G}_r)_{r \geq 0}\)-predictable probability kernel \(\mathsf{m}\), whose existence is due to \cite[Lemma~3.2]{LakSPA15}. As \(\mathbb{S}_\delta\) is convex, \(\int a (\lambda, X_r) \, \mathsf{m}_r (d \lambda) \in \mathbb{S}_\delta\). Hence, by a standard representation result for continuous local martingales with absolutely continuous quadratic variation \cite[Theorem~3.4.2, Remark~3.4.3]{KaraShre}, there exists a Brownian motion \((W_s)_{s \geq t}\) with \(W_t = 0\), relative to the \(P\)-augmentation of \((\mathcal{G}_s)_{s \geq t}\), such that \(P\)-a.s.
	\[
	X_s = x + \int_t^s \int b (\lambda, X_r) \, M (dr, d\lambda) + \int_t^s \sqrt{ \int a (\lambda, X_r) \, \mathsf{m}_r (d \lambda) } \, d W_r, \quad s \geq t,
	\]
	see also \cite[Theorem~2.5]{nicole1987compactification} for related statements.
	Since \(\|\int b (\lambda, X_r) \, \mathsf{m}_r (d \lambda) \| \leq \b (X_r)\), the process \((X_{t + s} - x)_{s \geq 0}\), driven by the Brownian motion \((W_{t + s})_{s \geq 0}\), satisfies \cite[Assumption~1.1]{krylov_AOP_21} and \cite[Assumption~1.1]{krylov_21} with \(\b\) replaced by \(\b (x + \cdot\,)\) and \(2 \b (x + \cdot\,) / \delta\), respectively. Consequently, the results from \cite{krylov_AOP_21, krylov_21} apply under every \(P \in \cK (t, x)\).
\end{remark}

\begin{remark} 
	Thanks to \cite[Theorem~1.1]{krylov_AOP_21}, we have \(\cK (x) \neq \emptyset\), see Lemma~\ref{lem: upper hemi} below for more properties of the set-valued map \((t, x) \mapsto \cK (t, x)\). The \(L_d\)-condition on the drift is the sharpest general \(L_p\)-assumption ensuring \(\cK(x) \neq \emptyset\). Let us explain this with an example. Adapting \cite[Example~1.1]{krylov_AOP_21}, we notice that the~SDE 
	\begin{align} \label{eq: SDE no solution} 
		d Y_t = - \frac{d}{2} \, \frac{Y_t}{\| Y_t \|^2} \1_{\{\|Y_t\| \leq 1\}} \, dt + d W_t, \quad Y_0 = 0, 
	\end{align} 
	has no solution. To see this, assume for contradiction that \(Y\) is a solution. By It\^o's formula and the fact that a.s. \(\int_0^\infty \1_{\{Y_t = 0\}} \, dt = 0\), by the occupation times formula applied to the first coordinate of \(Y\), it follows that
	\[
	d \| Y_{t} \|^2 = d (1 - \1_{\{\|Y_t\| \leq 1\}}) \, dt + 2 \langle Y_t, d W_t \rangle.
	\] 
	Hence, for \(T := \inf \{t \geq 0 \colon \|Y_t\| \geq 1\} > 0\), the stopped process \(\|Y_{\cdot \wedge T}\|^2\) is a non-negative local martingale, hence a supermartingale. As supermartingales cannot resurrect from zero, we must have a.s. \(Y = 0\) on \([0, T]\). However, this means that \(W_{\cdot \wedge T}\) is of finite variation, which is absurd. Consequently, the SDE \eqref{eq: SDE no solution} cannot have a solution. 
	We now examine the integrability properties of the drift coefficient \(b (x) = - \frac{d}{2} \, \frac{x}{\|x\|^2} \1_{\{\|x\| \leq 1\}}\), \(x \neq 0\), with \(b (0) := 0\).
	The computation 
	\[
	\int_{\bR^d} \| b (x) \|^{d - \varepsilon} \, dx = C_{d, \varepsilon} \int_0^1 r^{\varepsilon - d} \, r^{d - 1} \, dr = C_{d, \varepsilon} \int_0^1 r^{- 1 + \varepsilon} \, dr
	\] 
	shows that \(b \in L_{d - \varepsilon} (\bR^d)\) for every \(\varepsilon \in (0, d)\), but \(b \not \in L_d(\bR^d)\).  
\end{remark} 

By Remark~\ref{rem: Ito representation} and \cite[Corollary~2.7]{krylov_21}, there exist two constants \(N, \nu > 0\), depending only on \(d, \delta\) and \(\|\b\|_{L_d (\bR^d)}\), such that, for every bounded open domain \(D \subset \bR^d\), 
\begin{equation} \label{eq: bound stopping time}
	\begin{split}
		\sup_{x \in \bR^d} \sup_{P \in \cR (x)} E^P \Big[ e^{ \nu \tau_D / \on{diam} \, (D)^2 } \Big] \leq N, 
	\end{split}
\end{equation} 
where \(\tau_D := \inf \{t \geq 0 \colon X_t \not \in D\}\) denotes the first exit time from \(D\). This integrability is useful to define value functions in elliptic settings, which is the topic of the next subsection.

\subsection{HJB Dirichlet problems} \label{sec: elliptic}
Let \(D \subset \bR^d\) be an open bounded domain. As usual, the boundary of \(D\) is denoted by \(\partial D := \overline{D} \setminus D\), where \(\overline{D}\) denotes the closure in \(\bR^d\). Recall that \(D\) is said to satisfy the exterior ball condition\footnote{The exterior ball condition is sometimes also called the exterior sphere condition (see, e.g., \cite[p.~27] {GT_01}). We follow here the terminology from \cite[p.~92]{krylov_18}, where, however, the term refers to the uniform exterior ball condition defined below.} if for every \(x \in \partial D\) there exists a radius \(\rho (x) > 0\) and a center \(z \not \in \oD\) such that \(\overline{B}_{\rho (x)} (z) \cap \oD= \{x\}\), where \(\overline{B}_{\rho (x)} (z) := \{ y \in \bR^d \colon \|y - z\| \leq \rho (x)\}\). Further, \(D\) is said to satisfy the uniform exterior ball condition if \(\rho(x)\) can be taken uniformly for all \(x \in \partial D\). 
In addition to the coefficients \(a\) and \(b\) from Section~\ref{sec: control}, we consider
\begin{align*}
	g &\colon \Lambda \times \overline{D} \to \bR, \\
	f &\colon \partial D\to \bR, \phantom {\overline{D}}
\end{align*}
and impose the following:

\begin{SA} \label{SA: main2}
	\begin{enumerate}
		\item[(a)]  \(D\) satisfies the exterior ball condition.
		\item[(b)] \(g\) and \(f\) are bounded lower semicontinuous functions, and \(g\) is also continuous in the \(\Lambda\)-variable.
	\end{enumerate}
\end{SA}
We will indicate in each result below which precise standing assumption is in force.
Recalling \eqref{eq: bound stopping time}, we define
\begin{align} \label{eq: value function D}
	v (x) := \inf_{P \in \cR(x)} E^P & \Big[ f ( X_{\tau_D}) + \int_0^{\tau_D} \int g (\lambda, X_t) \, M (dt, d \lambda)\Big], \quad x \in \oD.
\end{align}
This function is called the {\em value function}. We also notice that \(v\) is bounded by the boundedness of \(f, g\) and \eqref{eq: bound stopping time}. 

Our first main result shows that the value function \(v\) is connected to the Hamilton--Jacobi--Bellman (HJB) equation
\begin{align} \label{eq: PDE}
	- H (x, \nabla u (x), \nabla^2 u (x)) = 0, \quad x \in D, 
\end{align} 
where, for \(x \in D\) and \(\varphi \in C^2 (D)\), 
\begin{align*}
	H (x, \nabla \varphi (x), \nabla^2 \varphi (x)) &:=  \inf_{\lambda \in \Lambda} \big\{ L (\lambda, x, \nabla \varphi (x), \nabla^2 \varphi (x)) + g (\lambda, x)  \big\},
\end{align*}
with \(L\) given by \eqref{eq: operator L}. As the coefficients of \eqref{eq: PDE} are merely measurable and not necessarily continuous, we use the concept of \(L_p\)-viscosity solutions as introduced in the seminal work~\cite{CCKS_96}. For a recap of the definition and important properties, we refer to Appendix~\ref{app: recap elliptic}.
The proofs of Theorems~\ref{theo: main continuity} and \ref{theo: main Holder continuity} are given in Section~\ref{sec: pf SMS} below.

\begin{theorem}\label{theo: main continuity} Assume that the Standing~Assumptions~\ref{SA: main1} and \ref{SA: main2} are in force and
	let \(f\) be continuous. Then, there exists a constant \(d_0 =\) \(d_0 (d, \delta, \|\b\|_{L_d (\bR^d)})\) \(\in (d /2, d)\) such that, for every \(p \geq d_0\), the value function \(v\) is an \(L_p\)-viscosity solution to the equation \eqref{eq: PDE}.
\end{theorem}

We can say more about the regularity of \(v\).

\begin{theorem} \label{theo: main Holder continuity}
	Assume that the Standing~Assumptions~\ref{SA: main1} and \ref{SA: main2} are in force.
	\begin{enumerate} 
		\item[\textup{(a)}]
		There exists an \(\alpha = \alpha (d, \delta, \|\b\|_{L_d (\bR^d)}) \in (0, 1)\) such that \(v\) is locally \(\alpha\)-H\"older continuous in the open set~\(D\).
		\item[\textup{(b)}] 
		If \(D\) satisfies the uniform exterior ball condition and \(f\) is continuous, then \(v\) is continuous on \(\oD\) with modulus of continuity \(w\) given by 
		\[
		w ( \rho ) = C \big( \rho^{\alpha / 2} + w_f (C \rho^{\alpha / 4}) \big), \quad \rho \geq 0, 
		\] 
		where \(w_f\) is a concave, increasing modulus of continuity for \(f\),\footnote{Such a modulus of continuity always exists, see \cite[Lemma~12~(ii)]{JKZ_24}.} and the constant \(C > 0\) depends only on \(d, \delta, \|\b\|_{L_d(\bR^d)}, \on{diam} (D), \|g\|_\infty, \sup_{z \in \oD} | v (z) |\), and the radius from the uniform exterior ball condition. 
		\item[\textup{(c)}]
		If \(D\) satisfies the uniform exterior ball condition and \(f\) is \(\beta\)-H\"older continuous, then \(v\) is \(\alpha \beta / 4\)-H\"older continuous on \(\oD\). More precisely, for all \(x, y \in \oD\), 
		\[
		| v (x) - v (y) | \leq C \| x - y \|^{\alpha \beta / 4},
		\] 
		where the constant \(C > 0\) depends only on \(d, \delta, \|\b\|_{L_d(\bR^d)}, \on{diam}(D), \|g\|_\infty,\) \(\sup_{z \in \oD} |v (z)|\), the H\"older constant of \(f\), and the radius from the uniform exterior ball condition. 
	\end{enumerate}
\end{theorem}

The next result is a uniqueness theorem, which follows from combining recent results from~\cite{KS_22,krylov_CPDE_20} with Theorem~\ref{theo: main continuity}. For the following results, notice that, without loss of generality, \(d_0\) from Theorem~\ref{theo: main continuity} can be taken larger than the constant \(\bar{p}\) from Lemma~\ref{lem: strict maximum}, with the ellipticity constants \(\delta/2\) and \(1/(2\delta)\), and the constant \(d_0\) from \cite[Theorem~1.2]{krylov_CPDE_20}, with the ellipticity constant \(\delta/2\). We impose this assumption for the rest of this section.

\begin{proposition} \label{theo: uniqueness} 
	Assume that the Standing~Assumptions~\ref{SA: main1} and \ref{SA: main2} are in force.
	Take \(p \in (d_0, d)\) and let \(\theta = \theta (d, p, \delta, \|\b\|_{L_d (\bR^d)}) > 0\) be the constant from \cite[Theorem~1.2]{krylov_CPDE_20} for the ellipticity constant \(\delta / 2\) and twice the Frobenius norm. Assume that \(D\) satisfies the uniform exterior ball condition and that there exists an \(R_0 > 0\) such that, for every \(z \in D, r \in (0, R_0]\), there exists a continuous map \(\Lambda \ni \lambda \mapsto \overline{a}_{z, r} (\lambda) \in \mathbb{S}_\delta\) such that 
	\begin{align} \label{eq: uniqueness condition} 
		\frac{1}{\mu_L (B_r (z))} \int_{B_r (z)} \sup_{\lambda \in \Lambda} \| a (\lambda, x) - \overline{a}_{z, r}(\lambda) \| \, dx \leq \theta, 
	\end{align} 
	where \(\mu_L\) denotes the Lebesgue measure.
	Then, whenever \(f \in C (\partial D)\), the value function \(v\) is the unique \(L_p\)-viscosity solution to 
	\begin{align*}
		\begin{cases} 
			- H (x, \nabla u (x), \nabla^2 u (x)) = 0, & \text{on } D, \\ u = f, & \text{on } \partial D, 
		\end{cases} 
	\end{align*} 
	and \(v \in W^{2}_{p, \textup{loc}} (D) \cap C (\oD)\).
\end{proposition}

\begin{proof}
	Recall from Lemma~\ref{lem: uniqueness from Sobolev solution} in Appendix~\ref{app: recap elliptic} that uniqueness among \(L_p\)-viscosity solutions follows from the existence of a Sobolev solution, which then coincides with the unique \(L_p\)-viscosity solution. As the existence of a Sobolev solution is entailed by our assumptions thanks to \cite[Theorem~1.2]{krylov_CPDE_20}, see \cite[Corollary~4.1.4]{krylov_18} for the hypothesis, the value function \(v\) is the only \(L_p\)-viscosity solution by Theorem~\ref{theo: main continuity}. These observations entail all claims. 
\end{proof}

\begin{remark}
	Condition \eqref{eq: uniqueness condition} is satisfied, for example, if \(\{ x \mapsto a (\lambda, x) \colon \lambda \in \Lambda\}\) is uniformly equicontinuous. In this case, \(R_0 > 0\) can be taken small enough such that, for all \(\lambda \in \Lambda\),
	\[
	\| a (\lambda, x) - a (\lambda, y) \| \leq \theta, \quad \|x - y\| \leq R_0,
	\]
	and \(\overline{a}_{z, r} (\lambda) := a (\lambda, z)\). 
	
\end{remark} 

The quantitative regularity estimates from Theorem~\ref{theo: main Holder continuity} are useful in the context of stability, allowing us to apply the Arzel\`a--Ascoli theorem. 

\begin{proposition} Assume that the Standing~Assumptions~\ref{SA: main1} and \ref{SA: main2} are in force and that \(D\) satisfies the uniform exterior ball condition. Take a sequence of coefficients \((b^m, a^m, f^m, g^m)_{m = 1}^\infty\) satisfying the same structural assumptions as \((b, a, f, g)\) above (with the same \(\delta\) and \(\b\)), and set 
	\begin{align*}
		H^m (x, \nabla \varphi (x)&, \nabla^2 \varphi (x)) 
		\\&:= \inf_{\lambda \in \Lambda} \Big\{ \tfrac{1}{2} \on{tr} \big[ a^m (\lambda, x) \nabla^2 \varphi (x) \big] + \langle b^m (\lambda, x), \nabla \varphi (x) \rangle + g^m (\lambda, x) \Big\}.
	\end{align*}
	Take \(p \in (d_0, d)\) and assume the following:
	\begin{enumerate}
		\item[\textup{(a)}] \(\sup_{m \geq 1} \| g^m \|_\infty < \infty\), \(\sup_{m \geq 1} \| f^m \|_\infty < \infty\), and there exists a modulus of continuity \(w\) such that 
		\[
		| f^m (x) - f^m (y) | \leq w (\|x - y\|), \quad x, y \in \partial D, \ m \geq 1;
		\] 
		\item[\textup{(b)}] for all \(B_r (x_0) \subset \subset D\) and \(\varphi \in W^2_p (B_r (x_0))\), 
		\[
		\|h_\varphi - h^m_\varphi  \|_{L_p (B_r (x_0))} \to 0 \text{ as } m \to \infty, 
		\] 
		where
		\begin{align*} 
			h_\varphi (x) &:= H (x, \nabla \varphi (x), \nabla^2 \varphi (x)), \\
			h^m_\varphi (x) &:= H^m (x, \nabla \varphi (x), \nabla^2 \varphi (x)); 
		\end{align*} 
		\item[\textup{(c)}] \(f^m (x) \to f (x)\) for all \(x \in \partial D\). 
	\end{enumerate}
	Let \((u^m)_{m = 1}^\infty\) be the value functions associated to \((b^m, a^m, f^m, g^m)_{m = 1}^\infty\). Then, the sequence \((u^m)_{m = 1}^\infty\) is relatively compact in \(C (\oD)\) and any accumulation point \(u\) is an \(L_p\)-viscosity solution to \eqref{eq: PDE} that coincides with \(f\) on \(\partial D\). 
	Finally, under the hypotheses of Proposition~\ref{theo: uniqueness}, \(u\) is the unique \(L_p\)-viscosity solution to \eqref{eq: PDE} with \(u = f\) on \(\partial D\) and \(u = v\). 
\end{proposition} 
\begin{proof} 
	By Theorem~\ref{theo: main Holder continuity} and the Arzel\`a--Ascoli theorem (\cite[Theorem~A5.2]{Kallenberg}), \((u^m)_{m = 1}^\infty\) is relatively compact in \(C (\oD)\). Let \(u\) be an accumulation point. By (c), we have \(u = f\) on \(\partial D\). 
	It remains to prove that \(u\) is an \(L_p\)-viscosity solution. This follows from Lemma~\ref{lem: stability appendix} in Appendix~\ref{app: recap elliptic}.
	The uniqueness statement is due to Proposition~\ref{theo: uniqueness}.
\end{proof} 

\begin{discussion} \label{diss: 1 MR}
	(a) To the best of our knowledge, for genuinely nonlinear settings, value functions of controlled diffusions with \(L_d\)-drift have not been studied before. Even in settings with bounded coefficients (that are included in our setting by the boundedness of the domain \(D\); take \(b (\lambda, x) = 0\) for \(x \not \in \oD\)), our results seem to be new. 
	For a linear setting, the \(L_{d_0}\)-viscosity property of a stochastic solution has been established in \cite[Theorem~6.10]{krylov_PTRF_21}. At this point, it is worth mentioning that the linear result covers more general cost functions, allowing for more irregularity than semicontinuity. Our approach uses lower semicontinuity to prove the existence of a family of Markovian optimizers (see part (c) of this discussion for more comments). It is an interesting question whether our technique can be refined to cover more general cost functions. We leave this for future investigation.
	
	\smallskip
	(b) Our results also provide some understanding of extremal solutions for linear equations. To explain this, assume that \(\partial D \in C^{2, \alpha}\) and that \(f\) is continuous, take a Borel function \(\overline{a} \colon \bR^d \to \mathbb{S}_\delta\), and consider the linear Dirichlet problem
	\begin{align} \label{eq: linear equation}
		\begin{cases} \tfrac{1}{2}\on{tr} \big[ \overline{a} (x) \nabla^2 u (x) \big] = 0, & \text{on } D, \\ u = f, & \text{on } \partial D. \end{cases}
	\end{align} 
	We know from the counterexample in \cite{zbMATH01140123} that this equation may have more than one \(L_d\)-viscosity solution. It was proved in \cite{J_96} that any \(L_d\)-viscosity solution \(u\) to this equation can be approximated uniformly on \(\oD\) by Sobolev solutions \(u^n \in W^{2}_{d, \textup{loc}} (D) \cap C (\oD)\) to 
	\[
	\tfrac{1}{2}\on{tr} \big[ \overline{a}^n (x) \nabla^2 u^n (x) \big] = 0\ \text{ on } D, 
	\] 
	where \(\overline{a}^n \colon D \to \mathbb{S}_\delta\) is continuous and \(\overline{a}^n_{ij} \to \overline{a}_{ij}\) in \(L_1 (D)\). We stress that the sequence \((\overline{a}^n)_{n = 1}^\infty\) depends on the solution \(u\).
	Let \(P^n_x\) be the unique law of a solution process to
	\[
	d Y^n_{t \wedge \tau_D (Y^n)} = \sqrt{\overline{a}^n (Y^n_t)} \1_{\{t \leq \tau_D (Y^n)\}} \, d W^n_t, \quad Y^n_0 = x \in \oD.
	\] 
	By an application of Krylov's It\^o formula (see \cite[Theorem~2.10.1]{krylov_80} or \cite[Theorem~1.3]{krylov_AOP_21}), it follows that
	\[
	u^n (x) = E^{P^n_x} \big[ u^n (X_{\tau_D}) \big], \quad x \in \oD. 
	\]
	Using an argument similar to the proof of \cite[Theorem~1.1]{krylov_AOP_21}, we get that the sequence \((P^n_x)_{n = 1}^\infty\) is tight and any accumulation point \(P_x\) is the law of a solution to the SDE
	\begin{align} \label{eq: SDE good solution}
		d Y_{t \wedge \tau_D (Y)} = \sqrt{\overline{a} (Y_t)} \1_{\{t \leq \tau_D (Y)\}} \, d W_t, \quad Y_0 = x. 
	\end{align} 
	Consequently, using that \(u^n \to u\) uniformly on \(\oD\) and \(P^n_x \to P_x\) weakly, by arguments similar to those in Lemma~\ref{lem: st cont} below, we conclude that
	\[
	u (x) = E^{P_x} \big[ f (X_{\tau_D}) \big], \quad x \in \oD. 
	\]
	Denoting the set of all laws of solutions to \eqref{eq: SDE good solution} by \(\mathcal{S} (x)\), this yields that
	\[
	u (x) \geq \inf_{P \in \mathcal{S} (x)} E^P \big[ f (X_{\tau_D}) \big] =: \underline{u} (x), \quad x \in \oD, 
	\] 
	and with Theorem~\ref{theo: main continuity} we conclude that \(\underline{u}\) is the minimal \(L_d\)-viscosity solution to~\eqref{eq: linear equation}. 
	As in the proof of Theorem~\ref{theo: main continuity}, one can show that 
	\[
	\overline{u} (x) := \sup_{P \in \mathcal{S} (x)} E^P \big[ f (X_{\tau_D}) \big], \quad x \in \oD, 
	\] 
	is also an \(L_d\)-viscosity solution to \eqref{eq: linear equation} and, by the preceding argument, \(\overline{u}\) is the maximal one.
	This complements the abstract result from \cite{Jen_Sw_05} that general elliptic equations with bounded coefficients always have minimal and maximal \(L_p\)-viscosity solutions, identifying them for the linear equation \eqref{eq: linear equation} as value functions that arise through optimization over laws of SDEs (equivalently, solutions to martingale problems). 
	It is an interesting question whether this observation extends to the Dirichlet problem \eqref{eq: PDE}. We leave this open for future investigation. 
	
	\smallskip
	(c) 
	Let us elaborate on our strategy for the proof of the continuity of the value function. First, we establish the existence of a strong Markov selection, i.e., a family \((P_x)_{x \in \oD}\) such that \(P_x \in \cK (x)\), \((P_x \circ X^{-1}_{\cdot \wedge \tau_D})_{x \in \oD}\) is a strong Markov family and 
	\begin{equation*} \begin{split}
			v (x) = E^{P_x}  \Big[ f ( &X_{\tau_D}) + \int_0^{\tau_D} \int g (\lambda, X_t) \, M (dt, d \lambda)\Big], \quad x \in \oD.
	\end{split} \end{equation*} 
	Utilizing the Markov property of the selection \((P_x \circ X^{-1}_{\cdot \wedge \tau_D})_{x \in \oD}\), we show that there exists a bounded Borel map \(g^* \colon \oD \to \bR\) such that 
	\[
	v (x) = E^{P_x}  \Big[ f ( X_{\tau_D}) + \int_0^{\tau_D}g^* (X_t) \, dt \,\Big], \quad x \in \oD.
	\] 
	In summary, we have linearized the problem, reducing the continuity question for \(v\) to the function 
	\[
	x \mapsto E^{P_x}  \Big[ f ( X_{\tau_D}) + \int_0^{\tau_D} g^* (X_t) \, dt\,\Big]. 
	\]
	To establish its continuity, we relate the strong Markov selection to an SDE whose diffusion coefficient takes values in \(\mathbb{S}_\delta\) and whose drift has an \(L_d\)-bound. Building on this connection, we adapt a number of analytic results that were developed by N.~V.~Krylov \cite{krylov_PTRF_21, krylov_AOP_21, krylov_21}. 
	
	\smallskip
	(d) In the following, we briefly sketch how our methods can be applied to stochastic control problems of the type studied in the recent paper \cite{F_25}.
	Consider the value function
	\begin{align*}
		w (x) := \inf_{P \in \cK(x)} E^P \Big[ \int_0^\infty \int e^{- \rho s} g (\lambda, X_s) \, M (ds, d \lambda) \Big], \quad x \in \bR^d, 
	\end{align*}
	for a bounded lower semicontinuous Borel function \(g \colon \Lambda \times \bR^d \to \bR\) that is continuous in the \(\Lambda\)-variable and a parameter \(\rho > 0\). 
	Using the same methods as in the proof of Theorem~\ref{theo: SMS} below, one shows that there exists a selection \((P_x)_{x \in \bR^d}\) such that \(P_x \in \cK (x)\), \((P_x \circ X^{-1})_{x \in \bR^d}\) is a strong Markov family and 
	\[
	w (x) = E^{P_x} \Big[ \int_0^\infty \int e^{- \rho s} g (\lambda, X_s) \, M (ds, d \lambda) \Big], \quad x \in \bR^d. 
	\] 
	As in the proof of Theorem~\ref{theo: krylov interior} below, \((P_x \circ X^{-1})_{x \in \bR^d}\) is associated to an SDE with \(\mathbb{S}_\delta\)-valued diffusion coefficient and \(L_d\)-drift. Further, using arguments as in the proof of Theorem~\ref{theo: main Holder continuity}, there exists a bounded Borel function \(g^* \colon \bR^d \to \bR\) such that 
	\[
	E^{P_x} \Big[ \int_0^\infty \int e^{- \rho s} g (\lambda, X_s) \, M (ds, d \lambda) \Big] = E^{P_x} \Big[ \int_0^\infty e^{- \rho s} g^* (X_s) \, ds \Big], \quad x \in \bR^d. 
	\]
	Summing up, compared to its counterpart with a bounded domain \(D\), we get a global Markov selection and \(w\) turns out to be a resolvent.
	By \cite[Corollary~1.2]{krylov_21}, there exists a constant \(C = C(d, \delta, \|\b\|_{L_d (\bR^d)}) > 0\) such that, for every compact set \(K \subset \bR^d\), as \(R \to \infty\), 
	\begin{align*}
		\sup_{x \in K} E^{P_x} \Big[ \int_0^\infty e^{- \rho s} & | g^* (X_s) | \1_{\{ \|X_s\| > R \}} \, ds\Big] 
		\\&\leq \frac{C \|g^*\|_\infty}{R} \int_0^\infty e^{- \rho s} \Big(\sqrt{s} + \sup_{x \in K} \| x\| \Big) \, ds \to 0. 
	\end{align*}  
	Hence, as each 
	\[
	x \mapsto E^{P_x} \Big[ \int_0^\infty e^{- \rho s} g^* (X_s) \1_{\{ \|X_s\| \leq R\}} \, ds \Big], \quad R > 0, 
	\] 
	is continuous by \cite[Theorem~4.1]{krylov_PTRF_21}, the same is true for \(w\).
	Furthermore, again as in the proof of Theorem~\ref{theo: main continuity}, one may show that \(w\) is an \(L_{d_0}\)-viscosity solution to the HJB equation
	\[
	\rho w(x) - H (x, \nabla w (x), \nabla^2 w (x)) = 0, \quad x \in \bR^d. 
	\] 
	The same results hold when \(b\) is bounded, without necessarily being dominated by \(\b \in L_d (\bR^d)\). The only change in the above argument is that one can directly refer to \cite[Corollary~9, p.~78]{krylov_book_arxiv} and dominated convergence for the continuity of the resolvent
	\[
	x \mapsto E^{P_x} \Big[ \int_0^\infty e^{- \rho s} g^* (X_s) \, ds \Big],
	\] 
	which provides the continuity of the value function. 
	These observations complement some results from~\cite{F_25}, providing new conditions for the \(L_{d_0}\)-viscosity property of the value function~\(w\). 
\end{discussion}

\subsection{The semigroup connection} \label{sec: parabolic}
Besides the connection to viscosity theory, control problems are naturally related to nonlinear semigroups. This connection dates back to \cite{N_76a, N_76b} for equations with sufficiently regular coefficients, see also \cite{FS_06, N_15} for textbook treatments and \cite{DKN_20, NR_21} for more recent advances in this direction.
In the following, we investigate the relationship for our general setting with merely measurable coefficients.
Throughout this section, we take a function 
\[
g \colon \Lambda \times \bR^d \to \bR. 
\] 
In addition to our permanent Standing Assumption~\ref{SA: main1}, we impose the following:
\begin{SA}\label{SA: main 3}
	The running cost coefficient \(g\) is bounded, lower semicontinuous and continuous in the \(\Lambda\)-variable, and \(\b\) from Standing Assumption~\ref{SA: main1}~(c) satisfies \(\b \in L_{d + 1} (\bR^d)\).
\end{SA}
For a bounded Borel function \(f \colon \bR^d \to \bR\), we set 
\[
S_t (f) (x) := \inf_{P \in \cR (x)} E^P \Big[ f (X_t) + \int_0^t \int g (\lambda, X_s) \, M (ds, d\lambda)\Big], \quad x \in \bR^d, \ t \in \bR_+.
\] 
The next result clarifies the semigroup connection within our control framework. The proofs of Theorems~\ref{theo: semigroup}, \ref{theo: strong Feller} and \ref{theo: joint regu semigroup} are given in Section~\ref{sec: pf main semigroup} below.
\begin{theorem} \label{theo: semigroup}
	Assume that the Standing~Assumptions~\ref{SA: main1} and \ref{SA: main 3} are in force.
	The family \((S_t)_{t \geq 0}\) is a nonlinear semigroup on the space \(C_b (\bR^d)\) of bounded continuous functions from \(\bR^d\) into \(\bR\), i.e., 
	\[
	S_t (C_b (\bR^d)) \subset C_b (\bR^d), \quad S_t (S_s (f)) = S_{t + s} (f), \ \ \forall \, f \in C_b (\bR^d), \, t, s \in \bR_+.
	\] 
\end{theorem}

The operators \((S_t)_{t > 0}\) also have a smoothing property, which can be viewed as a regularization by noise effect.

\begin{theorem} \label{theo: strong Feller}
	Assume that the Standing~Assumptions~\ref{SA: main1} and \ref{SA: main 3} are in force.
	The family \((S_t)_{t > 0}\) has the strong \(\textit{LSC}_b\)-Feller property, i.e., 
	\begin{align} \label{eq: strong Feller}
		S_t (\textit{LSC}_b (\bR^d)) \subset C_b (\bR^d), \quad t > 0, 
	\end{align} 
	where \(\textit{LSC}_b (\bR^d)\) denotes the set of all bounded lower semicontinuous functions from \(\bR^d\) into \(\bR\). Moreover, there exists a constant \(\alpha = \alpha (d, \delta) \in (0, 1)\) such that, for every \(f \in \textit{LSC}_b (\bR^d)\) and \(t > 0\), the map \(x \mapsto S_t (f) (x)\) is locally \(\alpha\)-H\"older continuous. 
\end{theorem} 

We can also say something about the joint regularity of \((t, x) \mapsto S_t (f) (x)\).

\begin{theorem} \label{theo: joint regu semigroup}
	Assume that the Standing~Assumptions~\ref{SA: main1} and \ref{SA: main 3} are in force.
	\begin{enumerate}
		\item[\textup{(a)}] For all \(f \in C_b (\bR^d)\), \((t, x) \mapsto S_t (f) (x)\) is jointly continuous on~\(\bR_+ \times \bR^d\). 
		\item[\textup{(b)}] For every bounded uniformly continuous \(f \colon \bR^d \to \bR\) and \(T > 0\), the map \((t, x) \mapsto S_t (f) (x)\) is uniformly continuous on \([0, T] \times \bR^d\) with a modulus of continuity (associated to the metric \(|t - s| + \|x - y\|\)) given by 
		\[
		w (\rho) = C w_f (C \rho^{1/2}) + C \rho^{\alpha / 2}, \quad \rho \geq 0, 
		\] 
		where \(w_f\) is a concave, increasing modulus of continuity for \(f\), \(\alpha =\) \(\alpha (d, \delta) \in (0, 1)\) is the exponent from Theorem~\ref{theo: strong Feller}, and \(C > 0\) only depends on \(d\), \(\delta\), \(\| \b \|_{L_d (\bR^d)}\), \(\|\b\|_{L_{d + 1}(\bR^d)}\), \(\|f\|_\infty\), \(\|g\|_\infty\) and the time horizon \(T\). 
		\item[\textup{(c)}] For every bounded \(\beta\)-H\"older continuous function \(f \colon \bR^d \to \bR\) and \(T > 0\), the map \((t, x) \mapsto S_t (f) (x)\) is \((\alpha \wedge \beta) / 2\)-H\"older continuous on \([0, T] \times \bR^d\), i.e., for \((t, x), (s, y) \in [0, T] \times \bR^d\), 
		\[
		| S_t (f) (x) - S_s (f) (y) | \leq C \Big( |t - s| + \|x - y\| \Big)^{(\alpha \wedge \beta) / 2}, 
		\] 
		where \(C\) depends on \(d\), \(\delta\), \(\| \b \|_{L_d (\bR^d)}\), \(\|\b\|_{L_{d + 1}(\bR^d)}\), \(\|f\|_\infty\), \(\|g\|_\infty\), \(T\), and the H\"older constant of \(f\).
	\end{enumerate}
\end{theorem}

\begin{remark} 
	The smoothing property \eqref{eq: strong Feller} traces back to \cite{CN22b}, where it was established within a one-dimensional framework for model uncertainty. Our results demonstrate that this smoothing phenomenon is independent of both the dimension and the regularity of the coefficients. Furthermore, we strengthen this property by proving that lower semicontinuous functions are regularized to local Hölder continuity. 
\end{remark}

In the final part of this section, we investigate the infinitesimal description of the semigroup \((S_t)_{t \geq 0}\), relating it to the parabolic PDE
\begin{align} \label{eq: parabolic PDE}
	- \partial_t u - H (x, \nabla u (t, x), \nabla^2 u (t, x)) = 0 \text{ on } [0, T) \times \bR^d, 
\end{align} 
where \(T > 0\) is a finite deterministic time horizon and \(H\) is defined as in Section~\ref{sec: elliptic}, now with \(x \in \bR^d\) and \(g\) from Standing Assumption~\ref{SA: main 3}. As explained in \cite{hol16,NL_82}, this PDE can be viewed as the generator equation for the semigroup \((S_t)_{t \geq 0}\). As we only consider measurable coefficients and an \(L_{d + 1}\)-dominated drift, the results from these works do not apply to our setting and we leave a rigorous discussion of the connection open for future investigations.  

\smallskip 
As in the elliptic case of \eqref{eq: PDE}, we consider \(L_{d + 1}\)-viscosity solutions, whose definition is recalled in Appendix~\ref{app: recap parabolic}.

	\begin{SA}\label{SA: main 4}
		\(\b \in L_{d + 2, \on{loc}} (\bR^d)\).
\end{SA}
The following result is proved in Section~\ref{sec: pf parabolic viscosity} below.

\begin{theorem} \label{theo: parabolic PDE}
	Assume that the Standing~Assumptions~\ref{SA: main1}, \ref{SA: main 3} and \ref{SA: main 4} are in force.
	For every \(f \in C_b (\bR^d)\), the function \(u \colon [0, T] \times \bR^d \to \bR\) defined by \(u (t, x) := S_{T - t} (f) (x)\)
	is an \(L_{d + 1}\)-viscosity solution to \eqref{eq: parabolic PDE}. 
\end{theorem} 

For related results under suitable regularity conditions on the coefficients, partially also in a non-local setting, see \cite{CN22b, CN_25_JEEQ, CN_25_EJP}.
Theorem~\ref{theo: parabolic PDE} seems to be the first result that applies to a multidimensional setting with merely measurable coefficients and integrable drift. 

Theorems~\ref{theo: semigroup}, \ref{theo: joint regu semigroup}, and \ref{theo: parabolic PDE} together show that \(u\) is an \(L_{d + 1}\)-viscosity solution with a variety of regularity properties. For general parabolic nonlinear equations with bounded coefficients, interior H\"older estimates can be found in \cite[Proposition~5.3]{CKS_00}. For results on equations with unbounded ingredients, see also the recent paper \cite{KST_19}.  
It is worth mentioning that we have to establish a priori continuity of \(u\) before we can apply general results for \(L_p\)-viscosity solutions to the value function~\(u\).

Finally, let us comment on the \(L_{d+2, \on{loc}}\)-drift assumption that is imposed in Theorem~\ref{theo: parabolic PDE}. This type of assumption is quite common in the context of parabolic nonlinear PDEs, see the introduction of \cite{KS_26} for some comments. In our proof, it is used to establish a Krylov-type It\^o formula for parabolic Sobolev functions, see Theorem~\ref{theo: ito formula} below, and it ensures the structure condition~(b) from Appendix~\ref{app: recap parabolic} with \(\gamma = \b\).

\section{Proof of  Theorems~\ref{theo: main continuity} and \ref{theo: main Holder continuity}} \label{sec: pf SMS}
In this section we prove Theorems~\ref{theo: main continuity} and \ref{theo: main Holder continuity}, formalizing the strategy outlined in Discussion~\ref{diss: 1 MR}~(c). Throughout this section, the Standing Assumptions~\ref{SA: main1} and \ref{SA: main2} are assumed to hold.

\subsection{A strong Markov selection principle} \label{sec: SMSP}
A key tool for the proof of Theorems~\ref{theo: main continuity} and \ref{theo: main Holder continuity} is the following strong Markov selection principle, which is of independent interest. 

\begin{theorem} \label{theo: SMS}
	Assume that the Standing~Assumptions~\ref{SA: main1} and \ref{SA: main2} are in force.
	There exists a family \((P_y)_{y \in \overline{D}}\) of probability measures on \((\Omega \times \m, \cF \otimes \mathcal{M})\) such that \(y \mapsto P_y\) is measurable, \(P_y \in \cK (y)\), \((P_y \circ X^{-1}_{\cdot \wedge \tau_D})_{y \in \overline{D}}\) is strongly Markov and 
	\begin{equation*} \begin{split}
			v (x) = E^{P_x}  \Big[ f ( X_{\tau_D}) + \int_0^{\tau_D} \int g (\lambda, X_t) \, M (dt, d \lambda)\Big]
	\end{split} \end{equation*} 
	for all \(x \in \oD\). 
\end{theorem}
The core methods to prove this result were developed by N.~V.~Krylov in \cite{krylov_selection}, see also \cite[Theorems 6.2.3, 12.2.3]{SV}, \cite[Proposition~6.6]{nicole1987compactification}, \cite[Proposition~3.2]{Haus86} and \cite[Propositions~5.9, 5.14]{HausLep90}. The closest result in the literature appears to be \cite[Proposition~5.14]{HausLep90}. There are, however, some crucial differences between the existing results and our setting. First, we are working with an \(L_d\)-drift that may be unbounded (while the previous results were established under certain boundedness assumptions). This requires a careful investigation of some compactness and continuity properties of the set of control rules, as well as the treatment of the stopping time \(\tau_D\). Second, we aim for a time-homogeneous selection (as in \cite{Haus86}, but this paper considers a different control framework with a control-independent cost function), while the results from \cite{nicole1987compactification,HausLep90} give a time-inhomogeneous selection (which is natural as they consider coefficients with time-dependence). Again, this needs a careful investigation. We provide the details in the remainder of this subsection, mainly focusing on crucial differences to the existing literature. 

\subsubsection{Preparations}
By \cite[Lemma~3.2]{LakSPA15}, there exists a probability kernel \(\mathsf{m}\) from \(\bR_+ \times \m\) into \(\Lambda\), predictable with respect to \((\mathcal{G}_t)_{t \geq 0}\), such that 
\begin{align} \label{eq: kernel}
	\mathsf{m}_s (d \lambda) (m)  \, ds \equiv \mathsf{m} (s, m, d \lambda) \, ds = m (ds, d \lambda), \quad m \in \m.
\end{align} 
For \(t \in \bR_+\) and \(\overline{\omega}, \overline{\alpha} \in \Omega \times \m\) with \( X_t (\overline{\omega}) = X_t (\overline{\alpha})\), we define the concatenation \(\overline{\omega} \otimes_t \overline{\alpha} \in \Omega \times \m\) such that 
\[
(X_s, \mathsf{m}_s (d \lambda)) (\overline{\omega} \otimes_t \overline{\alpha}) := \begin{cases} (X_s, \mathsf{m}_s (d \lambda)) (\overline{\omega}), & s < t, \\ (X_s, \mathsf{m}_s (d \lambda)) (\overline{\alpha}), &  s \geq t. \end{cases}
\]
Let \(\cP(\Omega \times \m)\) be the set of probability measures on \((\Omega \times \m, \cF \otimes \mathcal{M})\).
For \(P \in \cP(\Omega \times \m)\), a finite \((\cG_t)_{t \geq 0}\)-stopping time \(\tau\) and a transition kernel \(\overline{\omega} \mapsto Q_{\overline{\omega}}\) from \((\Omega \times \m, \cF \otimes \M)\) into \((\Omega \times \m, \cF \otimes \M)\) such that \(Q_{\overline{\omega}} (X_{\tau (\overline{\omega})} = (X_\tau) (\overline{\omega})) = 1\) for \(P\)-a.e. \(\overline{\omega} \in \Omega \times \m\), we set 
\begin{align} \label{eq: pasting measure}
	(P \otimes_\tau Q) (G) := \iint \1_G (\overline{\omega} \otimes_{\tau (\overline{\omega})} \overline{\alpha} ) \, Q_{\overline{\omega}} (d \overline{\alpha}) \, P (d \overline{\omega}), \quad G \in \cF \otimes \M.
\end{align} 
Recall that \(\cG_\tau\) is countably generated for any \((\cG_t)_{t \geq 0}\)-stopping time \(\tau\), cf. \cite[Notations~6.1]{nicole1987compactification}. Consequently, for \(P \in \cP(\Omega \times \m)\), there exists a regular conditional probability w.r.t. \(\cG_\tau\) that we denote by \(P (\, \cdot \mid \cG_\tau)\).
The first exit time of \(X\) from a set \(U \subset \bR^d\) after time \(s \in \bR_+\) is denoted by 
\[
\tau_U^s := \inf \{t \geq s \colon X_t \not \in U\}.
\]
For \(t \in \bR_+\) and \(\overline{\omega} \in \Omega \times \m\), we set 
\begin{align} \label{eq: Pi kernel}
	\Pi^t_{\overline{\omega}} \, (A) := \delta_{( (s\, \mapsto\, X_{t} (\overline{\omega})), \delta^0)} (A), \quad A \in \cF \otimes \mathcal{M},
\end{align} 
where \(\delta^0 (ds, d\lambda) = \delta_{\lambda^0} (d \lambda) \, ds\) for some fixed but arbitrary element \(\lambda^0 \in \Lambda\).
\begin{definition} \label{def: stability}
	We say that a set-valued map  \(\cU \colon \bR_+ \times \oD \twoheadrightarrow \cP(\Omega \times \m)\) is
	\begin{enumerate}
		\item[\textup{(i)}]
		\emph{stable under conditioning} if for all \((t, x) \in \bR_+ \times \oD\), \(P \in \cU (t, x)\), any \(P\)-a.s. finite \((\cF_s)_{s \geq 0}\)-stopping time \(\tau \colon \Omega \to [0, \infty]\) with \(P\)-a.s. \(t \leq \tau \leq \tau^t_D\), 
		there exists a \(P\)-null set \(N \in \cG_\tau\) such that 
		\[\Pi_{\overline{\omega}}^{\tau (\overline{\omega})} \otimes_{\tau (\overline{\omega})} P (\,\cdot\mid \cG_\tau) (\overline{\omega}) \in \cU(\tau (\overline{\omega}), (X_\tau) (\overline{\omega}))\] for all \(\overline{\omega} \not \in N\);
		\item[\textup{(ii)}]
		\emph{stable under pasting} if for all \((t, x) \in \bR_+ \times \oD\), \(P \in \cU(t, x)\), any \(P\)-a.s. finite \((\cF_s)_{s \geq 0}\)-stopping time \(\tau \colon \Omega \to [0, \infty]\) with \(P\)-a.s. \(t \leq \tau \leq \tau^t_D\) and every \(\cG_\tau\)-measurable map \(\overline{\omega} \mapsto Q_{\overline{\omega}} \in \cP(\Omega \times \m)\) such that \(P\)-a.s. \(Q \in \cU (\tau, X_\tau)\), we have \(P \otimes_\tau Q \in \cU (t, x)\).
	\end{enumerate}
\end{definition}

Thinking of \(\cU\equiv\cK\), the concatenation \(\Pi^{\tau} \otimes_\tau\) in Definition~\ref{def: stability}~(i) is necessary, because \(P (\, \cdot \mid \cG_\tau)\) typically fails to satisfy the initial condition from \(\cK(\tau, X_\tau)\), i.e., under a measure from \(\cK (\tau, X_\tau)\) the first coordinate \(X\) remains equal to \(X_{\tau}\) until time \(\tau\).

\begin{lemma} \label{lem: stability}
	Assume that Standing Assumption~\ref{SA: main1} holds.
	The set-valued map \(\bR_+ \times \oD \ni (t, x) \mapsto \cK(t, x)\) is stable under conditioning and pasting. 
\end{lemma}

\begin{proof}
	The claim follows from standard properties of martingales; see \cite[Lemma 12.2.1]{SV} for the classical situation without controls, and \cite[Theorems~6.2, 6.3]{nicole1987compactification}, \cite[Lemmata~5.7, 5.8]{HausLep90} for control settings. We omit the details for brevity.
\end{proof}

For \(t \in \bR_+\), we define the shift operator \(\theta_t \colon \Omega \times \m \to \Omega \times\m\) by 
\begin{align} \label{eq: shift}
	(X, \mathsf{m}_r (d \lambda) \, dr) (\theta_t ( \overline{\omega})) := (X_{\,\cdot\, + t}, \mathsf{m}_{r + t} (d \lambda) \, dr) (\overline{\omega}).
\end{align}
Furthermore, we define the ``right inverse shift'' \(\gamma_t \colon \Omega \times \m \to \Omega \times \m\) by 
\begin{align} \label{eq: inverse shift}
	(X, \mathsf{m}_r (d \lambda) \, dr) (\gamma_t (\overline{\omega})) := (X_{ (\,\cdot\, - t) \, \vee \, 0}, \mathsf{m}_{(r - t) \, \vee \, 0} (d \lambda) \, dr) (\overline{\omega}).
\end{align} 
Notice that 
\[
\theta_t \circ \gamma_t = \on{id},
\]
which explains why \(\gamma_t\) can be seen as a right inverse of \(\theta_t\). 

\begin{lemma} \label{lem: shift} Assume that Standing Assumption~\ref{SA: main1} holds and take a point \((t, x) \in \bR_+ \times \bR^d\). 
	\begin{enumerate}
		\item[\textup{(i)}] If \(P \in \cK(t, x)\), then \(P \circ \theta_t^{-1} \in \cR(x)\).
		\item[\textup{(ii)}] \(P \in \cR (x)\) if and only if \(P \circ \gamma_t^{-1} \in \cK (t, x)\).
	\end{enumerate}
\end{lemma}

\begin{proof}
	The claims follow from \cite[Lemma~2.9]{jacod80}, the identity \(\theta_t \circ \gamma_t = \on{id}\) and the definition of the sets \(\cK(t, x)\) and \(\cR(x)\).
\end{proof}

\begin{lemma} \label{lem: st cont} 
	Assume that the Standing~Assumptions~\ref{SA: main1} and \ref{SA: main2} hold.
	\begin{enumerate}
		\item[\textup{(i)}] For every \(x \in \partial D\) and \(P \in \cK (x)\), \(P (\tau_{\oD} = 0) = 1\). 
		\item[\textup{(ii)}] For every \((t, x) \in \bR_+ \times \oD\) and \(P \in \cK (t, x)\), \(P (\tau^t_D = \tau^t_{\oD} < \infty) = 1\).
		\item[\textup{(iii)}] The map \((s, \omega') \mapsto \tau^s_D (\omega')\), from \(\bR_+ \times \Omega\) into \([0, \infty]\), is continuous at every point \((t, \omega)\) such that \(\omega \in \{ \tau^t_D = \tau^t_{\oD}\}\).
	\end{enumerate}
\end{lemma} 
\begin{proof}
(i) We adapt standard barrier arguments, see, e.g., \cite[Section~6.3]{GT_01}, \cite[Lemma~2.4]{LM80} or \cite[Theorem~A.12]{HausLep90}. Take \(x \in \partial D\) and \(P \in \cK (x)\). Let \(r > 0\) and \(z \not\in \oD\) be associated with \(x\) through the exterior ball condition, i.e., \(\{y \in \bR^d \colon \|y - z\| \leq r \} \cap \oD = \{x\}\). In particular, \(\|x - z\| = r\). For \(\varepsilon \in (0, r]\), set \(z_\varepsilon := x + \tfrac{\varepsilon}{r} (z - x)\). Then, \(\|x - z_\varepsilon\| = \varepsilon\) and \(\|z_\varepsilon - z\| = r - \varepsilon\), which implies \(\overline{B}_\varepsilon (z_\varepsilon) \subset \overline{B}_r (z)\) and consequently, \(\|y - z_\varepsilon\| \geq \varepsilon\) for all \(y \in \oD\). Set \(p := d / \delta^2\) and
\[
w_\varepsilon (y) := \varepsilon^{- p} - \|y - z_\varepsilon\|^{- p}, \quad y \in \bR^d \setminus \{z_\varepsilon\}.
\]
Then, \(w_\varepsilon \geq 0\) on \(\oD\) and \(w_\varepsilon (x) = 0\). With \(e := (y - z_\varepsilon) / \|y - z_\varepsilon\|\), as \(\on{tr} [ a (\lambda, y) ] \leq d / \delta\), \(\langle e, a (\lambda, y) e \rangle \geq \delta\) and \(p + 2 \geq d / \delta^2\), we obtain, for all \(y \in \oD\) and \(\lambda \in \Lambda\),
\begin{align*}
\tfrac{1}{2}\on{tr}\big[ a (\lambda, y) \nabla^2 w_\varepsilon (y) \big] &+ \langle b (\lambda, y), \nabla w_\varepsilon (y) \rangle
\\&= p \, \|y - z_\varepsilon\|^{- p - 2} \Big( \tfrac{1}{2} \on{tr} \big[ a (\lambda, y) \big] - \tfrac{p + 2}{2} \langle e, a (\lambda, y) e \rangle + \langle b (\lambda, y), y - z_\varepsilon \rangle \Big)
\\&\leq p \, \|y - z_\varepsilon\|^{- p - 1} \, \b (y) \leq p \, \varepsilon^{- p - 1} \, \b (y).
\end{align*}
Let \(\tau^\varepsilon := \inf \{ t \geq 0 \colon \|X_t - x\| \geq 3 \varepsilon \}\) and \(\sigma_\varepsilon := \tau_{\oD} \wedge \tau^\varepsilon\). Note that \(P\)-a.s. \(\sigma_\varepsilon < \infty\), by \eqref{eq: bound stopping time}, and that \(P\)-a.s. \(X_s \in \oD\) for \(s \leq \sigma_\varepsilon\). Let \(\varphi \in C^2_b (\bR^d)\) be such that \(\varphi = w_\varepsilon\) on \(\{ y \in \bR^d \colon \|y - z_\varepsilon\| \geq \varepsilon / 2\}\). Using the martingale property of \(C (\varphi)\) from \eqref{eq: test process} under \(P\), the optional stopping theorem and the Krylov estimate from \cite[Theorem~1.1]{krylov_21} or \cite[Theorem~2.8]{krylov_AOP_21}, we obtain that
\[
E^P \big[ w_\varepsilon (X_{\sigma_\varepsilon}) \big] \leq p \, \varepsilon^{- p - 1} \, E^P \Big[ \int_0^{\tau^\varepsilon} \b (X_s) \, ds \Big] \leq 3 p N \varepsilon^{- p} \, \|\b\|_{L_d (B_{3 \varepsilon} (x))},
\]
where \(N = N (d, \delta, \|\b\|_{L_d (\bR^d)})\). On \(\{ \tau^\varepsilon < \tau_{\oD}\}\), we have \(\|X_{\sigma_\varepsilon} - x\| = 3 \varepsilon\), hence \(\|X_{\sigma_\varepsilon} - z_\varepsilon\| \geq 2 \varepsilon\), and therefore \(w_\varepsilon (X_{\sigma_\varepsilon}) \geq \varepsilon^{-p} (1 - 2^{-p})\). As \(w_\varepsilon \geq 0\) on \(\oD\), we conclude that
\[
P (\tau^\varepsilon < \tau_{\oD}) \leq \frac{3 p N}{1 - 2^{-p}} \, \|\b\|_{L_d (B_{3 \varepsilon} (x))} \to 0, \quad \varepsilon \searrow 0.
\]
As the events \(\{\tau^\varepsilon < \tau_{\oD}\}\) increase when \(\varepsilon\) decreases, \(P (\tau^\varepsilon < \tau_{\oD}) = 0\) for all \(\varepsilon \in (0, r]\). Consequently, \(P\)-a.s. \(\tau_{\oD} \leq \lim_{\varepsilon \searrow 0} \tau^\varepsilon = \inf \{t \geq 0 \colon X_t \neq x\}\). Finally, the latter stopping time is \(P\)-a.s. zero, because \(d \langle X^{(1)}, X^{(1)} \rangle_t = \int a_{11} (\lambda, X_t) \, \mathsf{m}_t (d \lambda) \, dt \geq \delta \, dt\). Thus, \(P (\tau_{\oD} = 0) = 1\).

	\smallskip
	\noindent
	(ii) Take \(P \in \cK (t, x)\). By Lemma~\ref{lem: shift}, \(P \circ \theta_t^{-1} \in \cK (x)\) and \eqref{eq: bound stopping time} shows that
	\[
	P (\tau_D^t < \infty) = P \circ \theta_t^{-1} (\tau_D < \infty) = 1.
	\]
	Furthermore, by the stability under conditioning property of \(\cK\) and Lemma~\ref{lem: shift}, 
	\[
	P (\, \theta^{-1}_{\tau_D^t} (\, \cdot \,) \mid \cG_{\tau_D^t} ) \in \cK (X_{\tau_D^t}).
	\]
	Using part (i), we conclude that 
	\[
	P ( \tau^t_D = \tau^t_{\oD} ) = E^{P} \big[ P ( \theta_{\tau^t_D}^{-1} ( \{\tau_{\oD} = 0\}) \mid \cG_{\tau^t_D}) \big] = 1.
	\]
	
	\noindent
	(iii) Take a sequence \((t^n, \omega^n) \to (t, \omega)\) with \(\omega \in \{\tau^t_D = \tau^t_{\oD}\}\) and notice that
	\[
	\tau^{t^n}_D (\omega^n) = \tau_D (\theta_{t^n} ( \omega^n)) + t^n, \quad \tau_D (\theta_t (\omega)) = \tau_{\oD} (\theta_t (\omega)).
	\] 
	We have \(\theta_{t^n} ( \omega^n ) \to \theta_t ( \omega)\) by the Arzel\`a--Ascoli theorem.
	As \(\tau_D\) is lower semicontinuous and \(\tau_{\oD}\) is upper semicontinuous (\cite[Exercise~2.1, p.~75]{pinsky}), we obtain that 
	\begin{align*}
		\tau_D (\theta_t (\omega)) &\leq \liminf_{n \to \infty} \tau_D (\theta_{t^n} (\omega^n)) 
		\leq \limsup_{n \to \infty} \tau_D (\theta_{t^n} (\omega^n))
		\\&\leq \limsup_{n \to \infty} \tau_{\oD} (\theta_{t^n} (\omega^n))
		\leq \tau_{\oD} (\theta_t( \omega)) = \tau_D (\theta_t ( \omega)). \phantom {\lim_{n \to \infty}}
	\end{align*}
	In summary, the claim follows. 
\end{proof}

Following standard terminology (e.g., \cite[Lemma~17.4]{charalambos2013infinite}), a set-valued map \(\kappa \colon G \twoheadrightarrow H\) between topological spaces is called upper hemicontinuous if the lower inverse \(\kappa^\ell (C) = \{ g \in G \colon \kappa (g) \cap C \not = \emptyset\}\) of any closed set \(C \subset H\) is closed.

\begin{lemma} \label{lem: upper hemi}
	Assume that Standing~Assumption~\ref{SA: main1} holds.
	The set-valued map \(\bR_+ \times \bR^d \ni (t, x) \mapsto \cK(t, x)\) is upper hemicontinuous with non-empty compact values. 
\end{lemma}

\begin{proof}
	That \(\cK (x) \neq \emptyset\) follows from \cite[Theorem~1.1]{krylov_AOP_21}. Hence, we also have \(\cK (t, x) \neq \emptyset\) by Lemma~\ref{lem: shift}~(ii). 
	In view of \cite[Theorem~17.20]{charalambos2013infinite}, \((t, x) \mapsto \cK (t, x)\) is upper hemicontinuous with compact values if and only if, for every sequence \((t_n, x_n)_{n = 0}^\infty \subset \bR_+ \times \bR^d\) with \((t_n, x_n) \to (t_0, x_0)\) and \(P_n \in \cK (t_n, x_n)\) for \(n \in \mathbb{N}\), the sequence \((P_n)_{n = 1}^\infty\) has a limit point in \(\cK (t_0, x_0)\). 
	By \cite[Corollary~1.2]{krylov_21} and Kolmogorov's tightness criterion (\cite[Theorem~23.7]{Kallenberg}), the sequence \((P_n \circ X^{-1}_{\cdot + t_n})_{n = 1}^\infty\) is relatively compact. Hence, by the continuity of \((t, \omega) \mapsto \omega ( (\, \cdot - t) \vee 0)\), the same is true for \((P_n \circ X^{-1})_{n = 1}^\infty\), where we use that \(P_n ( X = x_n \text{ on } [0, t_n] ) = 1\). By the compactness of \(\m\), we conclude that \((P_n)_{n = 1}^\infty\) is relatively compact. Hence, after passing to a subsequence if necessary, we may assume that \(P_n \to P_0 \in \cP (\Omega \times \m)\).  It remains to show that \(P_0 \in \cK (t_0, x_0)\). 
	Using that the function \(\omega \mapsto \|\omega (t) - x_0\| \wedge 1\) is bounded and continuous, we get for every \(t < t_0\) that \(P_0\)-a.s. \(X_t = x_0\). Similarly, to cover the case \(t_0 = 0\), \(P_0\)-a.s. \(X_0 = x_0\), because \(P_n (X_0 = x_n) = 1\).  By the continuous paths, this yields \(P_0\)-a.s. \(X = x_0\) on \([0, t_0]\).
	We now use a Krylov-type smoothing method to show that, for every \(\varphi \in C^2_b (\bR^d)\), the process \((C_t (\varphi))_{t \geq t_0}\) from \eqref{eq: test process} is a local \(P_0\)-martingale. This suffices to conclude that \(P_0 \in \cK (t_0, x_0)\). To explain this in more detail, in case all \((C_t (\varphi))_{t \geq t_0}\) for \(\varphi \in C^2_b (\bR^d)\) are local \(P_0\)-martingales, \cite[Theorem~II.2.42]{JS} yields that the process \((X_t)_{t \geq t_0}\) is a continuous \(P_0\)-semimartingale with quadratic variation \((\int_{t_0}^t \int_\Lambda a (\lambda, X_s) \, M(ds, d \lambda))_{t \geq t_0}\). Hence, using the boundedness of the coefficient \(a\), It\^o's formula and the fact that local martingales with integrable quadratic variation are true martingales, all \((C_t (\varphi))_{t \geq t_0}\) for \(\varphi \in C_b^2 (\bR^d)\) are also true \(P_0\)-martingales. We now carry out this program.
	For \(R > 0\), set 
	\[
	\tau_R := \inf \{t \geq 0 \colon \|X_t\| \geq R\}.
	\] 
	By \cite[Lemma~11.1.2]{SV}, with the exception of countably many \(R > 0\), the stopping time \(\tau_R\) is \(P_0\)-a.s. continuous. 
	In particular, there exists a sequence \((R_n)_{n = 1}^\infty\) with \(R_n \nearrow \infty\) such that \(\tau_{R_n}\) is \(P_0\)-a.s. continuous. We now take such an \(R > 0\) that also satisfies \(\sup_{n \geq 0} \|x_n\| < R\).
	In the following, we show that, for every \(t_0 < s < t\), \(\varphi \in C^2_b (\bR^d)\) and any \(\cG_s\)-measurable bounded continuous function \(\psi \colon \Omega \times \m \to \bR\), 
	\begin{align} \label{eq: approx to show}
		E^{P_0} \big[ \psi \, ( C_{t \wedge \tau_R} (\varphi) - C_{s \wedge \tau_R} (\varphi) ) \big] = 0. 
	\end{align} 
	This implies the desired martingale property by a monotone class argument.
	Without loss of generality, we may assume that \(\sup_{n \geq 1} t_n < s\). Hence, by definition of \(\cK\), and because \(P_n \in \cK (t_n, x_n)\), we have 
	\begin{align} \label{eq: approx final 1}
		E^{P_n} \big[ \psi \, ( C_{t \wedge \tau_R} (\varphi) - C_{s \wedge \tau_R} (\varphi) ) \big]  = 0, \quad n \in \mathbb{N}. 
	\end{align} 
	We now transfer this property from the sequence \((P_n)_{n = 1}^\infty\) to its limit \(P_0\). 
	Let \(\rho \colon \bR^d \to \bR_+\) be a standard mollifier (\cite[Section~7.2]{GT_01}) and define, for \((\lambda, x) \in \Lambda \times \bR^d\),
	\[
	b_h (\lambda, x) := \int_{B_{1}} \rho( y ) b (\lambda, x - hy) \, dy, \quad B_R := \{y \in \bR^d \colon \|y\| < R\}.
	\] 
	Define \(a_h\) in the same way.
	We claim that, for every \(i, j = 1, \dots, d\), 
	\begin{equation} \label{eq: approx} 
		\begin{split} 
			\lim_{h \searrow 0} \, \int_{B_R} \sup_{\lambda \in \Lambda} \Big( | b_h^i (\lambda, x) - b^i (\lambda, x) |^d + | a_h^{ij} (\lambda, x) - a^{ij} (\lambda&, x) |^d \Big) \, dx = 0.
		\end{split}
	\end{equation} 
	
	We briefly discuss the first term, the argument for the second term is identical. Fix an \(\varepsilon > 0\). 
	By the compactness of \(\Lambda\) and Standing Assumption~\ref{SA: main1}~(a), there exist \(N \in \mathbb{N}\) points \(\lambda_1, \dots, \lambda_N \in \Lambda\), depending on \(\varepsilon\), and a non-negative Borel function \(\mathfrak{g} \in L_d (\bR^d)\), independent of \(\varepsilon\), such that, for all \(x \in B_R\), 
	\begin{align*}
		\sup_{\lambda \in \Lambda} | b^i_h (\lambda, x) - b^i (\lambda, x) | &\leq \max_{k \leq N} | b^i_h (\lambda_k, x) - b^i (\lambda_k, x) | + \varepsilon \, \Big( \int_{B_1} \rho (y) \mathfrak{g} (x - h y) \, dy + \mathfrak{g} (x) \Big).
	\end{align*} 
	Using this fact, we obtain that
	\begin{align*}
		\int_{B_R} \sup_{\lambda \in \Lambda} |& b_h^i (\lambda, x) - b^i (\lambda, x) |^d \, dx \\&\leq 2^{d-1} \sum_{k = 1}^N \int_{B_R} | b_h^i (\lambda_k, x) - b^i (\lambda_k, x)|^d \, dx + 2^{2 d-1} \varepsilon^d \| \mathfrak{g} \|_{L_d (B_{R + 1})}^d.
	\end{align*}
	Since the first term converges to zero as \(h \to 0\) (\cite[Lemma~7.2]{GT_01}), we conclude that \eqref{eq: approx} holds. 
	
	Define \(L_h\) analogously to \(L\) from \eqref{eq: operator L} with \(a_h\) and \(b_h\) instead of \(a\) and \(b\), respectively. 
	Combining \eqref{eq: approx} and \cite[Theorem~1.1]{krylov_21} yields that, with \(h \searrow 0\), 
	\begin{equation}\label{eq: approx final 2} \begin{split}
			&\sup_{n \geq 1} E^{P_n} \Big[ \Big| \int_{s \wedge \tau_R}^{t \wedge \tau_R} \int (L_h - L) (\lambda, X_r, \nabla \varphi (X_r), \nabla^2 \varphi (X_r)) \, M (dr, d \lambda) \Big| \Big]  
			\\& \quad \leq \sup_{n \geq 1} C \sum_{i, j = 1}^d \, E^{P_n \circ X_{\cdot + t_n}^{-1}} \Big[ \int_0^{\tau_R} \sup_{\lambda \in \Lambda} \Big( | b^i_h (\lambda, X_r) - b^i (\lambda, X_r)| \\&\hspace{6cm}+ | a^{ij}_h (\lambda, X_r) - a^{ij} (\lambda, X_r) | \Big) \, dr \Big] 
			\\&\quad \leq CR  \sum_{i, j = 1}^d  \Big( \int_{B_R} \sup_{\lambda \in \Lambda} \Big( | b_h^i (\lambda, x) - b^i (\lambda, x) |^d + | a_h^{ij} (\lambda, x) - a^{ij} (\lambda, x) |^d \Big) \, dx \Big)^{1/d} \to 0, 
		\end{split}
	\end{equation}
	where the constant \(C\) depends only on \(d, \delta, \|\b\|_{L_d(\bR^d)}, \|\nabla \varphi\|_\infty\) and \(\|\nabla^2 \varphi\|_\infty\). 
	We claim that further 
	\begin{equation} \label{eq: conv for P0}
		\begin{split}
			&E^{P_0} \Big[ \Big| \int_{s \wedge \tau_R}^{t \wedge \tau_R}\int (L_h - L) (\lambda, X_r, \nabla \varphi (X_r), \nabla^2 \varphi (X_r)) \, M (dr, d \lambda) \Big| \Big] \\& \, \leq CR  \sum_{i, j = 1}^d  \Big( \int_{B_R} \sup_{\lambda \in \Lambda} \Big( | b_h^i (\lambda, x) - b^i (\lambda, x) |^d + | a_h^{ij} (\lambda, x) - a^{ij} (\lambda, x) |^d \Big) \, dx \Big)^{1/d} \to 0, 
		\end{split}
	\end{equation}  
	where the constant \(C\) is as above. To get this inequality, we need an estimate as in \cite[Theorem~1.1]{krylov_21} for the measure \(P_0\). As \(P_0\) is not admissible at the current stage, we cannot directly refer to the reference \cite[Theorem~1.1]{krylov_21}. Nevertheless, as it applies to all \(P_n\), we can transport it to \(P_0\) by a weak-convergence and density argument: For every non-negative \(\ell \in C_b (\bR^d)\), the maps
	\[
	\omega \mapsto \int_0^{\tau_R (\omega) \wedge k} \ell (\omega(s)) \, ds, \quad k > 0, 
	\]
	are bounded and lower semicontinuous, following from the fact that \(\tau_R\) is lower semicontinuous. Hence, using Fatou's lemma, \cite[Theorem~15.5]{charalambos2013infinite}, and applying \cite[Theorem~1.1]{krylov_21} for every \(P_n\), we obtain that
	\begin{align*} 
		E^{P_0 \circ X_{\cdot + t_0}^{-1}} \Big[ \int_0^{\tau_R} \ell (X_s) \, ds \Big] &\leq \liminf_{k \to \infty} \liminf_{n \to \infty} E^{P_n \circ X^{-1}_{\cdot + t_n}} \Big[ \int_0^{\tau_R \wedge k} \ell (X_s) \, ds \Big] 
		\\&\leq \liminf_{n \to \infty} E^{P_n \circ X^{-1}_{\cdot + t_n}} \Big[ \int_0^{\tau_R} \ell (X_s) \, ds \Big] \leq C R \| \ell \|_{L_d (B_R)}.
	\end{align*} 
	A density argument, omitted here for brevity, shows that this inequality holds for all non-negative Borel \(\ell \colon \bR^d \to \bR\). Using this fact instead of \cite[Theorem~1.1]{krylov_21}, we conclude that \eqref{eq: conv for P0} holds. 
	Next, let \(C^h (\varphi)\) be defined as \(C (\varphi)\) in \eqref{eq: test process} with \(L_h\) instead of \(L\). Using \cite[Theorem~2.7]{krylov_AOP_21}, we obtain that 
	\begin{align*}
		E^{P_n} &\big[ (C^h_{t \wedge \tau_R} (\varphi) - C^h_{s \wedge \tau_R} (\varphi))^2 \big] 
		\\&\ \leq 8 \| \varphi \|_\infty^2  + 2 E^{P_n} \Big[ \Big( \int_{s \wedge \tau_R}^{t \wedge \tau_R} \int L_h (\lambda, X_r, \nabla \varphi (X_r), \nabla^2 \varphi (X_r)) \, M (dr, d \lambda) \Big)^2 \Big] 
		\\&\ \leq 8 \| \varphi\|_\infty^2 + C \sum_{i, j = 1}^{d} E^{P_n \circ X_{\cdot + t_n}^{-1}} \Big[ \Big( \int_0^{t} \sup_{\lambda \in \Lambda} \Big( | b^i_h (\lambda, X_r)| \\&\hspace{6.75cm}+ | a^{ij}_h (\lambda, X_r) | \Big) \1_{\{ X_r \in B_R\}} \, dr \Big)^2 \Big] 
		\\&\ \leq 8 \| \varphi\|_\infty^2 + C \sum_{i, j = 1}^d  \Big( \int_{B_R} \sup_{\lambda \in \Lambda} \Big( | b_h^i (\lambda, x) |^d + | a_h^{ij} (\lambda, x) |^d \Big) \, dx \Big)^{2/d}, 
	\end{align*}
	where the constant \(C\) depends only on \(d, \delta, \|\b\|_{L_d (\bR^d)}, \| \nabla \varphi\|_\infty, \|\nabla^2 \varphi\|_\infty\) and \(t\).
	By our choice of \(R > 0\), the random variable \(C^h_{t \wedge \tau_R} (\varphi) - C^h_{s \wedge \tau_R} (\varphi)\) is \(P_0\)-a.s. continuous by \cite[Theorem~8.10.61]{bogachev}. Thus, using the continuous mapping theorem and the uniform integrability that is entailed by the above second moment bound, we obtain that 
	\begin{equation} \label{eq: approx final 3}
		\begin{split} 
			\lim_{n \to \infty} E^{P_n} \big[ \psi \, ( &C^h_{t \wedge \tau_R} (\varphi) - C^h_{s \wedge \tau_R} (\varphi)) \big] 
			= E^{P_0}  \big[ \psi \, ( C_{t \wedge \tau_R}^h (\varphi) - C_{s \wedge \tau_R}^h (\varphi) ) \big].
		\end{split}
	\end{equation}
	Putting \eqref{eq: approx final 1}, \eqref{eq: approx final 2}, \eqref{eq: conv for P0}, and \eqref{eq: approx final 3} together yields \eqref{eq: approx to show} and hence, the proof is complete. 
\end{proof}

Define 
\[
\Gamma_t^f := f (X_{\tau_D^t}) + \int_t^{\tau_D^t} \int g (\lambda, X_s) \, M (ds, d \lambda), \quad t \in \bR_+,
\] 
where \(f\) and \(g\) are from Standing Assumption~\ref{SA: main2}~(b). 

\begin{lemma} \label{lem: lower expectation} 
	Assume that the Standing~Assumptions~\ref{SA: main1} and \ref{SA: main2} hold.
	Then, \((t, x, Q) \mapsto E^Q [ \Gamma^f_t ]\) is lower semicontinuous on the graph \(\on{gr}\, (\cK)\) of \(\bR_+ \times \oD \ni (t, x) \mapsto \cK (t, x)\).
\end{lemma} 
\begin{proof}
	Take \((t^n, x^n, Q^n)_{n = 0}^\infty \subset \on{gr} \, (\cK)\) with \((t^n, x^n, Q^n) \to (t^0, x^0, Q^0)\). 
	By Skorokhod's coupling theorem (\cite[Theorem~5.31]{Kallenberg}), there are \(\Omega \times \m\)-valued random variables \((X^n, M^n)_{n = 0}^\infty\) such that \((X^n, M^n)\) has law \(Q^n\) and a.s. \((X^n, M^n) \to (X^0, M^0)\). By Lemma~\ref{lem: st cont}, a.s. 
	\[
	\tau^{t^n}_D (X^n) \to \tau^{t^0}_D (X^0).
	\] 
	Now, as \(f\) and \(g\) are bounded and lower semicontinuous, by \cite[Theorem~8.10.61]{bogachev}, a.s. 
	\begin{align*}
		\liminf_{n \to \infty} \Big( f \big(X^n_{\tau^{t^n}_D (X^n)}\big ) &+ \int_{t^n}^{\tau^{t^n}_D (X^n)} \int g (\lambda, X^n_s) \, M^n (d s, d \lambda) \Big) 
		\\&\geq f\big (X^0_{\tau^{t^0}_D (X^0)} \big)+ \int_{t^0}^{\tau^{t^0}_D (X^0)} \int g (\lambda, X^0_s) \, M^0 (ds, d \lambda). 
	\end{align*} 
	By \eqref{eq: bound stopping time}, we have
	\[
	E^{Q^n} \big[ \big|\tau^{t^n}_D - t^n \big|^2\big] = E^{Q^n \, \circ\, \theta_{t^n}^{-1}} \big[ \big| \tau_D\big|^2 \big] \leq C (d, \delta, \|\b\|_{L_d (\bR^d)}, \on{diam} (D)), 
	\] 
	and consequently, 
	\[
	E \big[ \tau^{t^n}_D (X^n) - t^n \big] \to E \big[ \tau^{t^0}_D (X^0) - t^0\big]. 
	\] 
	Finally, Fatou's lemma applied to the non-negative random variable
	\[
	f \big(X^n_{\tau^{t^n}_D (X^n)}\big ) + \int_{t^n}^{\tau^{t^n}_D (X^n)} \int g (\lambda, X^n_s) \, M^n (d s, d \lambda) + \|f\|_\infty + \|g\|_\infty \, (\tau^{t^n}_D (X^n) - t^n)
	\] 
	yields the claimed lower semicontinuity. 
\end{proof}

We define
\[
\cK^* (t, x) := \Big\{ P \in \cK (t, x) \colon E^P \big[ \Gamma_t^f \big] = \inf_{Q \in \cK (t, x)} E^Q \big[ \Gamma_t^f \big] \Big\}, \quad (t, x) \in \bR_+ \times \oD.
\]
Following standard terminology (e.g., \cite[Definition~18.1]{charalambos2013infinite}), a set-valued map between topological spaces is called measurable if the lower inverse of any closed set is Borel. Trivially, upper hemicontinuous set-valued maps are measurable.  
If \(H\) is separable and metrizable and a set-valued map \(\kappa \colon G \twoheadrightarrow H\) has non-empty compact values, it can be viewed as a map \(\kappa^*\) from \(G\) into the compact subsets of \(H\), denoted \(\on{comp} (H)\), with the Hausdorff metric topology, and \(\kappa\) is measurable if and only if \(\kappa^*\) is Borel, see \cite[Theorem~18.10]{charalambos2013infinite}. 
We refer to \cite[Chapter~18]{charalambos2013infinite} and \cite{himmelberg} for detailed discussions of measurability of set-valued maps.

\begin{lemma} \label{lem: K initial}
	Assume that the Standing~Assumptions~\ref{SA: main1} and \ref{SA: main2} are in force.
	\begin{enumerate}
		\item[\textup{(i)}] \(\cK^*\) is measurable with non-empty compact values.
		\item[\textup{(ii)}] \(\cK^*\) is stable under conditioning and pasting.
		\item[\textup{(iii)}] For all \((t, x) \in \bR_+ \times \oD\), \(P \in \cK^* (t, x)\) implies that \(P \circ \theta^{-1}_t \in \cR^* (x)\), where \(\cR^* (x) := \cK^* (0, x)\).
		\item[\textup{(iv)}] For all \((t, x) \in \bR_+ \times \oD\), \(P \in \cR^* (x)\) implies that \(P \circ \gamma^{-1}_t \in \cK^* (t, x)\).
	\end{enumerate}
\end{lemma}
\begin{proof}
	(i) As \(\cK\) is upper hemicontinuous with compact values by Lemma~\ref{lem: upper hemi}, \(\on{gr} \, (\cK)\) is closed by \cite[Theorem~17.10]{charalambos2013infinite}. Consequently, it is a Polish space. 
	Recall that \((t, x, Q) \mapsto E^Q [ \Gamma^f_t ]\) is lower semicontinuous on \(\on{gr}\, (\cK)\) by Lemma~\ref{lem: lower expectation}.
	Now, \cite[Lemma~12.1.7]{SV} yields that the set-valued map 
	\[
	\on{comp} ( \on{gr} \, (\cK)) \ni K \mapsto \mathfrak{j} (K) := \Big\{ (t, x, P) \in K \colon E^P \big[ \Gamma^f_t \big] = \inf_{(s, y, Q) \in K} E^Q \big[ \Gamma^f_s \big] \Big\} 
	\] 
	is measurable with non-empty compact values. Using \cite[Theorems~18.5, 18.10]{charalambos2013infinite}, the set-valued map 
	\(
	(t, x) \mapsto \mathfrak{h} (t, x) := \{ (t, x) \} \times \cK (t, x) \in \on{comp} ( \on{gr} \, (\cK))
	\)
	is measurable. Consequently, also
	\(
	(t, x) \mapsto \mathfrak{j} (\mathfrak{h} (t, x)) = \{ (t, x) \} \times \cK^* (t, x)
	\)
	is measurable with non-empty compact values and, using once again \cite[Theorems~18.5, 18.10]{charalambos2013infinite}, the same is true for \(\cK^*\).

	\smallskip
	(ii) The proof is similar to \cite[Lemma~3.4]{Haus86}, \cite[Corollary~5.12]{HausLep90} and \cite[Corollary~6.5~(a)]{nicole1987compactification}, which adapt ideas from \cite[Lemma 12.2.2]{SV}. We omit the details for brevity.
	
	\smallskip
	(iii), (iv)
	Recall that \(\theta_t \circ \gamma_t = \on{id}\) and notice that \(P\)-a.s. \(\Gamma_0^{f}\, \circ\, \theta_t = \Gamma_t^{f}\) for all \(P \in \cK (t, x)\). Using Lemma~\ref{lem: shift}, we get that 
	\[
	\inf_{P \in \cK (t, x)} E^P \big[ \Gamma_t^{f} \big] = \inf_{P \in \cK (x)} E^P \big[ \Gamma_0^{f} \big], 
	\]
	which yields the claim of (iii). Since also \(\Gamma^f_t \circ \gamma_t = \Gamma^f_0\) and \(P \circ \gamma_t^{-1} \in \cK (t, x)\) for \(P \in \cK (x)\) by Lemma~\ref{lem: shift}~(ii), the same identity yields (iv). 
\end{proof}

For a function \(\ell \in C_b (\bR^d)\) and a constant \(\alpha > 0\), we define
\begin{align*}
	\Xi_t (\ell, \alpha) := \int_{t}^{\tau_D^t} e^{- \alpha s} \ell (X_{s}) \, ds, \quad t \in \bR_+.
\end{align*} 
The following lemma can be proved similarly to Lemma~\ref{lem: K initial}. We omit the details for brevity. 

\begin{lemma} \label{lem: K indu}
	Let \(\bR_+ \times \oD \ni (t, x) \mapsto \cU (t, x)\) be a measurable set-valued map with non-empty compact values that is stable under conditioning and pasting and satisfies \(\cU (t, x) \subset \cK (t, x)\) for all \((t, x) \in \bR_+ \times \oD\).
	Furthermore, assume the following:
	\begin{enumerate}
		\item[\textup{(a)}] \(P \in \cU (t, x)\) implies that \(P \circ \theta^{-1}_t \in \cU(0, x)\).
		\item[\textup{(b)}] \(P \in \cU (0, x)\) implies that \(P \circ \gamma^{-1}_t \in \cU (t, x)\). 
	\end{enumerate} 
	Define the set-valued map
	\[
	\bR_+ \times \oD \ni	(t, x) \mapsto \mathcal{T}^* (t, x) := \Big\{ P \in \cU (t, x) \colon E^P \big[ \Xi_t (\ell, \alpha) \big] = \inf_{Q \in \cU (t, x)} E^Q \big[ \Xi_t (\ell, \alpha) \big] \Big\}.
	\]
	Then, the following hold:
	\begin{enumerate}
		\item[\textup{(i)}] \(\mathcal{T}^*\) is measurable with non-empty compact values.
		\item[\textup{(ii)}] \(\mathcal{T}^*\) is stable under conditioning and pasting.
		\item[\textup{(iii)}] For all \((t, x) \in \bR_+ \times \oD\), \(P \in \mathcal{T}^* (t, x)\) implies that \(P \circ \theta^{-1}_t \in \mathcal{T}^* (0, x)\).
		\item[\textup{(iv)}] For all \((t, x) \in \bR_+ \times \oD\), \(P \in \mathcal{T}^* (0, x)\) implies that \(P \circ \gamma^{-1}_t \in \mathcal{T}^* (t, x)\).
	\end{enumerate}
\end{lemma}

\subsubsection{Proof of Theorem~\ref{theo: SMS}}
The basic idea has been established in \cite{krylov_selection}, cf. also the proofs of \cite[Theorems~6.2.3, 12.2.3]{SV}, \cite[Proposition 6.6]{nicole1987compactification}, \cite[Proposition 3.2]{Haus86} and \cite[Propositions~5.9, 5.14]{HausLep90}. As the presence of the stopping times \(\tau^t_D\) requires a careful treatment, we provide the details of the proof. 

Let \(\{\sigma_n \colon n \in \mathbb{N}\}\) be a dense subset of \((0, \infty)\) and let \(\{\phi_n \colon n \in \mathbb{N}\}\) be a dense (for the uniform norm) subset of \(C_c (\bR^d)\). Furthermore, let \((\alpha_N, \ell_N)_{N = 1}^\infty\) be an enumeration of \(\{(\sigma_m, \phi_n) \colon n, m \in \mathbb{N}\}\). 
For \((t, x) \in \bR_+ \times \oD\), define inductively  
\[
\cK^*_0 (t, x) := \cK^* (t, x),
\]
and, for \(N \in \mathbb{Z}_+\), 
\begin{align*}
	\cK^*_{N + 1} (t, x) := \Big\{ P \in \cK^*_N (t, x) \colon E^P \big[ \Xi_t &(\ell_{N + 1}, \alpha_{N + 1}) \big] 
	\\&= \inf_{Q \in \cK^*_N (t, x)} E^Q \big[ \Xi_t (\ell_{N + 1}, \alpha_{N + 1}) \big] \Big\}.
\end{align*} 
Moreover, we set \[\cK^*_\infty (t, x) := \bigcap_{N = 0}^\infty \cK^*_N (t, x).\]
By Lemmata~\ref{lem: K initial} and \ref{lem: K indu} and \cite[Lemma~18.4]{charalambos2013infinite}, the set-valued map \(\cK^*_\infty\) is measurable with non-empty compact and convex values, and it is stable under conditioning. The convexity follows from the convexity of the sets \(\cK (t, x)\), as the set of minimizers of a linear functional over a convex set is convex. Here, the fact that \(\cK^*_\infty (t, x)\) is non-empty follows from Cantor's intersection theorem, which states that the intersection of a nested sequence of non-empty compact sets is non-empty. 

As a consequence of the measurability, \cite[Corollary~18.15]{charalambos2013infinite} (or \cite[Theorem 12.1.10]{SV})
yields the existence of a measurable map \(\oD\ni x \mapsto Q_x\) with \(Q_x \in \cR^*_\infty (x) := \cK^*_\infty(0, x)\). Since \(\cR^*_\infty (x) \subset \cR^* (x)\), it remains to prove that \(\{Q_x \circ X^{-1}_{\cdot \wedge \tau_D} \colon x \in \oD\}\) is a strong Markov family. 

\smallskip 
For \((t, x) \in \bR_+ \times \oD\), we set 
\begin{align*}
	\cI^*_\infty (t, x) := \big\{ Q \circ X^{-1}_{\cdot \wedge \tau^t_D} \colon Q \in \cK^*_\infty (t, x) \big\} \subset \cP (\Omega). 
\end{align*}
We now show that each \(\cI^*_\infty (t, x)\) is a singleton, i.e., \(\cI^*_\infty (t, x) = \{P_{(t, x)}\}\). 

Take \((t, x) \in \bR_+ \times \oD\) and \(P, Q \in \cK^*_\infty(t, x)\).
For all \(t \leq t_1 < \dots < t_n\) and \(g_1, \dots, g_n \in C_c (\bR^d)\), we will prove that 
\begin{align*}
	E^P \Big[ \prod_{i = 1}^n  g_i (X_{t_i }) \1_{\{t_n \leq \tau_D^t\}} \Big] = E^Q \Big[ \prod_{i = 1}^n g_i (X_{t_i}) \1_{\{t_n \leq \tau^t_D\}}\Big].
\end{align*}
By a monotone class argument, this proves that \(P \circ X^{-1} = Q \circ X^{-1}\) on \(\cF_{\tau^t_D - }\), cf. \cite[Definition~IV.54.2]{DM}. 
As \(X\) is a predictable process (because it has continuous paths), it follows from \cite[Proposition~I.2.4]{JS} that all random variables of the form 
\[
g (X_{s_1 \wedge \tau^t_D}, \dots, X_{s_m \wedge \tau^t_D}), \quad g \in C_b (\bR^{dm}; \bR), \, s_1, \dots, s_m \in \bR_+, \, m \in \mathbb{N}, 
\] 
are \(\cF_{\tau_D^t-}\)-measurable.\footnote{In view of \cite[Lemma~1.3.3]{SV}, we also get that \(\cF_{\tau^t_D-} = \cF_{\tau^t_D} = \sigma (X_{s \wedge \tau^t_D}, s \geq 0)\). This precise identity is not needed for our argument.}  Hence, again by a monotone class argument, \(P \circ X^{-1} = Q \circ X^{-1}\) on \(\cF_{\tau^t_D - }\) already implies \[P \circ X_{\cdot \wedge \tau_D^t}^{-1} = Q \circ X_{\cdot \wedge \tau_D^t}^{-1},\] which shows that \(\mathcal{I}^*_\infty (t, x)\) is a singleton.

\smallskip
We proceed by induction. 
For all \(N \geq 1\), by definition of \(\cK^*_\infty\), 
\begin{align*}
	E^{P} \big[ \Xi_t (\ell_{N}, \alpha_N) \big] 
	= \inf_{R\, \in\, \cK^*_{N - 1} (t, x)} E^R \big[ \Xi_t (\ell_{N}, \alpha_N) \big] 
	= E^{Q} \big[ \Xi_t (\ell_N, \alpha_N) \big].
\end{align*}
This means that, for all \(n, m \geq 1\), 
\begin{align*}
	\int_t^\infty e^{- \sigma_m s}E^P \big[ \phi_n (X_{s}) \1_{\{s \le \tau^t_D\}}\big] ds = \int_t^\infty e^{- \sigma_m s}E^Q \big[ \phi_n (X_{s}) \1_{\{s \leq \tau^t_D\}} \big] ds.
\end{align*}
By the left-continuity of \(s \mapsto E^P [\phi_n (X_{s}) \1_{\{s \leq \tau^t_D\}} ]\) and \(s \mapsto E^Q [ \phi_n (X_{s}) \1_{\{s \leq \tau^t_D\}} ]\), the uniqueness of the Laplace transform shows that 
\begin{align*}
	E^P \big[\phi_n (X_{s}) \1_{\{s \leq \tau^t_D\}}\big] = E^Q \big [\phi_n (X_{s}) \1_{\{s \leq \tau^t_D\}}\big], \quad n \geq 1, s \geq t.
\end{align*}
This implies the induction base. 

Next, we establish the induction step. 
Take \(t \leq t_1 < \dots < t_{n + 1}\), \(g_1, \dots, g_{n + 1} \in C_c (\bR^d)\) and assume that the claim holds for \(n\).
Let \(\mathcal{V}_n\) be the \(\sigma\)-field generated by 
\[
A \cap \{t_n \leq \tau^{t}_D\}, \ \ A \in \sigma (X_{t_1}, \dots, X_{t_n}). 
\] 
Since 
\begin{align*}
	E^{P} \Big[ \prod_{k = 1}^{n + 1} g_k & (X_{t_k}) \1_{\{t_{n +1} \leq \tau^t_D\}} \Big] 
	\\&= E^{P} \Big[ E^{P} \big[ g_{n + 1} (X_{t_{n + 1}}) \1_{\{t_{n + 1} \leq \tau^{t_n}_D\}} \, |\, \mathcal{V}_n \big] \prod_{k = 1}^n g_k (X_{t_k}) \1_{\{t_n \leq \tau^t_D\}} \Big], 
\end{align*} 
it suffices to show that \(P\)-a.s. on \(\{t_n \leq \tau^t_D\}\)
\[
E^{P} \big[ g_{n + 1} (X_{t_{n + 1}}) \1_{\{t_{n + 1} \leq \tau^{t_n}_D\}} \, |\, \mathcal{V}_n \big] = E^{Q} \big[ g_{n + 1} (X_{t_{n + 1}}) \1_{\{t_{n + 1} \leq \tau^{t_n}_D\}} \, |\, \mathcal{V}_n \big].
\]
To ease our notation, set \(\tau := t_n \wedge \tau^t_D\) and let \(\Pi_{\overline{\omega}} \equiv \Pi^{\tau (\overline{\omega})}_{\overline{\omega}}\) be as in \eqref{eq: Pi kernel}. 
As \(\cK^*_\infty\) is stable under conditioning, there exists a \(P\)-null set \(N_1 \in \cG_{\tau}\) such that \(\Pi_{\overline{\omega}} \otimes_{\tau (\overline{\omega})} P (\, \cdot\mid \cG_{\tau}) (\overline{\omega}) \in \cK^*_\infty (\tau (\overline{\omega}), (X_{\tau})(\overline{\omega}))\) for all \(\overline{\omega} \not \in N_1\). 
By the tower rule, there exists a \(P\)-null set \(N_2 \in \mathcal{V}_n\) such that, for all \(\overline{\omega} \in N_2^c \cap \{t_n \leq \tau^t_D\}\) and \(A \in \cF\), 
\[
P (\{ \widetilde{\omega} \in \Omega \times \m \colon X_{t_n} (\widetilde{\omega}) = X_{t_n} (\overline{\omega}), t_n \leq \tau^t_D (\widetilde{\omega})\} \mid \mathcal{V}_n ) (\overline{\omega}) = 1\]
and 
\begin{equation} \label{eq: condi}
	\begin{split}
		\int \, (\Pi_{\overline{\omega}'} \otimes_{\tau (\overline{\omega}')} P (\, \cdot \mid \cG_{\tau}) (\overline{\omega}')) & (A \times \m) P (d \overline{\omega}' \mid \mathcal{V}_n) (\overline{\omega}) 
		\\&= (\Pi_{\overline{\omega}} \otimes_{t_ n} P (\, \cdot\mid \mathcal{V}_n) (\overline{\omega})) (A \times \m).
	\end{split}
\end{equation}
Let \(N_3 := \{P (N_1 \mid \mathcal{V}_n) > 0\} \in \mathcal{V}_n\). Clearly, \(E^P [ P (N_1 \mid \mathcal{V}_n)] = P(N_1) = 0\), which implies that \(P (N_3) = 0\). Take \(\overline{\omega} \in N_2^c \cap N_3^c \cap \{t_n \leq \tau^t_D\}\).
Then, for \(P (\, \cdot \mid \mathcal{V}_n) (\overline{\omega})\)-a.e. \(\overline{\omega}' \in \Omega \times \m\),
\[
\Pi_{\overline{\omega}'} \otimes_{\tau (\overline{\omega}')} P (\, \cdot\mid \cG_{\tau}) (\overline{\omega}') \in \cK^*_\infty (\tau (\overline{\omega}'), (X_\tau) (\overline{\omega}')) = \cK^*_\infty (t_n, X_{t_n} (\overline{\omega})).
\]
As \(\cK^*_\infty\) has convex and compact values, still for \(\overline{\omega} \in N_2^c \cap N_3^c \cap \{t_n \leq \tau^t_D\}\), we get that   
\begin{align*}
	\int (\Pi_{\overline{\omega}'} \otimes_{\tau (\overline{\omega}')} P (\, \cdot \mid \cG_{\tau}) (\overline{\omega}')) &(\ \cdot \, \times \m) P (d \overline{\omega}' \mid \mathcal{V}_n) (\overline{\omega})  
	\\&\in \big\{ P^* \circ X^{-1} \colon P^* \in \cK^*_\infty (t_n,  X_{t_n} (\overline{\omega})) \big\}, 
\end{align*}
and, by virtue of \eqref{eq: condi},
\[
(\Pi_{\overline{\omega}} \otimes_{t_n} P (\, \cdot\mid  \mathcal{V}_n) (\overline{\omega})) (\, \cdot \, \times \m) \in \big\{ P^* \circ X^{-1} \colon P^* \in \cK^*_\infty (t_n,  X_{t_n} (\overline{\omega})) \big\}.
\]
Similarly, there exists a \(Q\)-null set \(N_4 \in \mathcal{V}_n\) such that
\[
(\Pi_{\overline{\omega}} \otimes_{t_n} Q (\, \cdot\mid \mathcal{V}_n) (\overline{\omega})) (\, \cdot \, \times \m) \in \big\{ P^* \circ X^{-1} \colon P^* \in \cK^*_\infty (t_n,  X_{t_n} (\overline{\omega})) \big\}
\]
for all \(\overline{\omega} \in N^c_4 \cap \{t_n \leq \tau^t_D\}\). 
Set \(N := N_2 \cup N_3 \cup N_4\). As \(P = Q\) on \(\mathcal{V}_n\) by the induction hypothesis, we get that \(P (N) = 0\). For all \(\overline{\omega} \in N^c \cap \{t_n \leq \tau^t_D\}\), the induction base implies that 
\begin{align*}
	E^P\big[ g_{n + 1} (X_{t_{n + 1}}) \1_{\{t_{n + 1} \leq \tau^{t_n}_D\}} \mid \mathcal{V}_n\big] (\overline{\omega}) &= E^{\Pi_{\overline{\omega}} \, \otimes_{t_n} P (\, \cdot\, | \mathcal{V}_n) (\overline{\omega})} \big[ g_{n + 1} (X_{t_{n + 1}}) \1_{\{t_{n + 1} \leq \tau^{t_n}_D\}} \big] 
	\\&= E^{\Pi_{\overline{\omega}} \, \otimes_{t_n} Q (\, \cdot\, | \mathcal{V}_n) (\overline{\omega})} \big[ g_{n + 1} (X_{t_{n + 1}}) \1_{\{t_{n + 1} \leq \tau^{t_n}_D\}}\big]
	\\&= E^Q\big[ g_{n + 1} (X_{t_{n + 1}}) \1_{\{t_{n + 1} \leq \tau^{t_n}_D\}} \mid \mathcal{V}_n\big] (\overline{\omega}).
\end{align*}
The induction step is complete and hence, \(\mathcal{I}^*_\infty (t, x)\) is a singleton.

\smallskip
As above, we write \(\cI^*_\infty (t, x) = \{P_{(t, x)}\}\). Clearly, \(Q_x \circ X^{-1}_{\cdot \wedge \tau_D} = P_{(0,x)} =: P_x\). It remains to prove that \(\{P_x \colon x \in \oD\}\) is a strong Markov family. Take an initial value \(x \in \oD\), a finite stopping time \(\rho\), and a set \(G \in \cF\).
Equation~(6) on p.~103 in \cite{ito} gives that
\begin{align*}
	X_{\cdot + \rho} \circ X_{\cdot \wedge \tau_D} = X_{\cdot \wedge \tau_D} \circ \theta_{\tau'}, \quad \tau' := \rho \wedge \tau_D.
\end{align*}
Let \(\Pi^{\tau'}\) be as in \eqref{eq: Pi kernel}.
Then, for every \(A \in \cF_{\rho}\), we have \(\{ X_{\cdot \wedge \tau_D} \in A\} \in \cG_{\tau'}\) and 
\begin{align*}
	E^{P_x} \big[  \1_{A} \1_{\{X_{\cdot + \rho} \, \in \, G\}} \big] 
	&= E^{Q_x} \big[ \1_{\{X_{\cdot \wedge \tau_D} \in A\}} \1_{\{X_{\cdot \wedge \tau_D} \, \circ\, \theta_{\tau'}\, \in\, G\}} \big]
	\\&= E^{Q_x} \big[ \1_{\{X_{\cdot \wedge \tau_D} \in A\}} E^{Q_x} \big[ \1_{\{X_{\cdot \wedge \tau_D}\, \circ \, \theta_{\tau'}\, \in\, G\}} \mid \cG_{\tau'} \big] \big]
	\\&= E^{Q_x} \big[ \1_{\{X_{\cdot \wedge \tau_D} \in A\}} (\Pi^{\tau'} \otimes_{\tau'} Q_x (\, \cdot \mid \cG_{\tau'}) \, \circ \theta_{\tau'}^{-1} \circ X^{-1}_{\cdot \wedge \tau_D} ) (G) \big].
\end{align*}
Now, \(Q_x\)-a.s. \(\Pi^{\tau'} \otimes_{\tau'} Q_x (\, \cdot \mid \cG_{\tau'} ) \in \cK^*_\infty (\tau', X_{\tau'})\), as \(\cK^*_\infty\) is stable under conditioning. Notice that \(R \in \cK^*_\infty (s, y)\) implies \(R \circ \theta_s^{-1} \in \cR^*_\infty (y)\), as all \(\cK^*_N, \, N = 0, 1, 2, \dots,\) have this property by Lemmata~\ref{lem: K initial} and \ref{lem: K indu}. Hence, \(Q_x\)-a.s. \(\Pi^{\tau'} \otimes_{\tau'} Q_x (\, \cdot \mid\cG_{\tau'}) \, \circ \theta_{\tau'}^{-1} \in \cR^*_\infty (X_{\tau'})\) and in turn, 
\[ 
\Pi^{\tau'} \otimes_{\tau'}  Q_x (\, \cdot \mid \cG_{\tau'}) \, \circ \theta_{\tau'}^{-1} \, \circ X^{-1}_{\cdot \wedge \tau_D} \in \cI^*_\infty (0, X_{\tau'}) = \{P_{X_{\tau'}}\}.
\]
Thus, we proved that 
\begin{align*}
	E^{P_x} \big[ \1_A \1_{\{ X_{\cdot + \rho} \, \in \, G \}} \big] = E^{Q_x} \big[ \1_{\{X_{\cdot \wedge \tau_D} \, \in \, A\}} \, P_{X_{\tau'}} (G) \big].
\end{align*}
Finally, using that 
\(
X_{\tau'} = X_{\rho} \circ X_{\cdot \wedge \tau_D},
\) 
by \cite[Problem~6, p. 88]{ito}, we conclude that 
\begin{align*} 
	E^{P_x} \big[ \1_A \1_{\{ X_{\cdot + \rho} \, \in \, G \}} \big] &= E^{Q_x} \big[ \1_{\{X_{\cdot \wedge \tau_D} \, \in \, A\}} \, P_{X_\rho \, \circ \, X_{\cdot \wedge \tau_D}} (G) \big]
	= E^{P_x} \big[ \1_A P_{X_\rho} (G) \big].
\end{align*} 
This proves the strong Markov property and thus finishes the proof.
\qed

\subsection{Regularity estimates}
Next, we provide interior and boundary regularity estimates that we use to prove Theorems~\ref{theo: main continuity} and \ref{theo: main Holder continuity}. 
In the following, let \((P_x)_{x \in \oD}\) be the strong Markov selection from Theorem~\ref{theo: SMS}. Furthermore, take a bounded Borel function \(g^* \colon \oD\to \bR\) and define 
\begin{align} \label{eq: u in proof}
	u (x) := E^{P_x} \Big[ f (X_{\tau_D}) + \int_0^{\tau_D} g^* (X_s) \, ds \Big]
\end{align} 
for \(x \in \oD\). This function is well-defined by \eqref{eq: bound stopping time}.

\begin{theorem} \label{theo: krylov interior} 
	Assume that the Standing Assumptions~\ref{SA: main1} and \ref{SA: main2} are in force.
	Let \(R > 0\) and \(x_0 \in D\) be such that \(B_{2R} (x_0) \subset \subset D\). There exist two constants \(C = C (d, \delta, \|\b\|_{L_d (\bR^d)}, \on{diam}\, (D)) > 0\) and \(\alpha = \alpha (d, \delta, \|\b\|_{L_d (\bR^d)}) \in (0, 1)\) such that 
	\begin{align} \label{eq: holder interior}
		| u (x) - u (y) | \leq C R^{- \alpha} \, \| x - y \|^\alpha \Big( \sup_{z \in \oD} |u (z)| + \|g^*\|_{\infty} \Big)
	\end{align} 
	for all \(x, y \in B_R (x_0)\). 
\end{theorem} 
\begin{proof}
	Our strategy is as follows: First, we show that under each \(P_x\) the coordinate map \(X\) has an SDE representation till time \(\tau_D\). Second, we explain how to adapt results from \cite{krylov_PTRF_21} to conclude the claimed interior regularity estimate. 
	
	\smallskip
	\noindent	
	{\em Step 1:}
	The following uses some ideas from~\cite{Haus86,HausLep90}. 
	It follows from \cite[Corollary~1.2]{krylov_21} that
	\begin{align} \label{eq: second moment bound} 
		E^{P_x} \Big[ \sup_{s \in [0, T]} \|X_s\|^2 \Big] < \infty, \quad T > 0, \ x \in \oD.
	\end{align} 
	For \(y \in \oD\), set 
	\begin{align*}
		b^* (y) := \begin{cases} \liminf_{n \to \infty} n E^{P_y} \big[ X_{\tau_D \wedge 1/n} - X_0 \big], & \text{if the norm of the \(\liminf\)}\\&\text{is at most \(\b (y)\)}, \\ 0, &\text{otherwise},\end{cases} 
	\end{align*}
	and, for some \(A_0 \in \mathbb{S}_\delta\), 
	\begin{align*}
		a^*_{ij} (y) &:= \begin{cases}
			\liminf_{n \to \infty} n E^{P_y} \big[ (X^{(i)}_{\tau_D \wedge 1/n} - X^{(i)}_0) (X^{(j)}_{\tau_D \wedge 1/n} - X^{(j)}_0) \big],
			& \text{if the \(\liminf\)}\\&\text{(as matrix)}\\ &\text{is in \(\mathbb{S}_\delta\)}, \\ (A_0)_{ij}, &\text{otherwise}.
		\end{cases} 
	\end{align*}
	Notice that \(b^*\) and \(a^*\) are well-defined by~\eqref{eq: second moment bound}.
	We obtain from \cite[Theorem~1.1]{krylov_21} that 
	\begin{equation} \label{eq: bound for fubini} 
		\begin{split}
			E^{P_x} \Big[ \int_0^{\tau_D} \int \| b (\lambda, X_s) \| \, \mathsf{m}_s (d \lambda) \, ds \Big] &\leq E^{P_x} \Big[ \int_0^{\tau_D} \b (X_s) \, ds \Big] 
			\\&\leq \on{const} (d, \delta, \| \b \|_{L_d (\bR^d)}, \on{diam} (D)) < \infty. 
		\end{split} 
	\end{equation} 
	For \(x \in \oD\), using the Markov property of \((P_z \circ X_{\cdot \wedge \tau_D}^{-1})_{z \in \oD}\) and Fubini's theorem (which we can apply thanks to \eqref{eq: bound for fubini}), we obtain, \(P_x\)-a.s. for Lebesgue a.e. \(t < \tau_D\),
	\begin{equation*} \begin{split}
			n E^{P_{X_t}} \big[ X_{\tau_D \wedge 1/n} - X_0 \big] &= n E^{P_{x}} \big[ X_{\tau_D \wedge (t + 1/n)} - X_t \mid X^{-1} (\cF_t) \big] \phantom \int
			\\&= n E^{P_{x}} \Big[ \int_t^{\tau_D \wedge (t + 1/n)}  \int b (\lambda, X_s) \, \mathsf{m}_s (d \lambda) \, ds \mid X^{-1} (\cF_t) \Big] 
			\\&= n \int_t^{t + 1/n} E^{P_{x}} \Big[ \int b (\lambda, X_s) \, \mathsf{m}_s (d \lambda) \1_{\{s \leq \tau_D\}} \mid X^{-1} (\cF_t) \Big] \, ds
			\\&\xrightarrow{\quad} E^{P_x} \Big[ \int b (\lambda, X_t) \, \mathsf{m}_t (d \lambda) \mid X^{-1} (\cF_t) \Big], \quad n \to \infty,
	\end{split} \end{equation*} 
which means that, \(P_x\)-a.s. for Lebesgue a.e. \(t < \tau_D\),
	\begin{align} \label{eq: b indetity imp}
		b^* (X_t) = E^{P_x} \Big[ \int b (\lambda, X_t) \, \mathsf{m}_t (d \lambda) \mid X^{-1} (\cF_t) \Big].
	\end{align} 
	Similarly, we obtain that, \(P_x\)-a.s. for Lebesgue a.e. \(t < \tau_D\), 
	\[
	a^* (X_t) = E^{P_x} \Big[ \int a (\lambda, X_t) \, \mathsf{m}_t (d \lambda) \mid X^{-1} (\cF_t) \Big]. 
	\] 
	By the definition of \(\cR (x)\), \(x \in \oD\), and \cite[Theorem~II.2.42]{JS}, the process \(X_{\cdot \wedge \tau_D}\) is a \(P_x\)-\((\cG_t)_{t \geq 0}\)-semimartingale whose semimartingale characteristics \((B, C)\) are given by 
	\[
	B_t = \int_0^{t \wedge \tau_D} \int b (\lambda, X_s) \, \mathsf{m}_s (d \lambda) \, ds, \quad C_t = \int_0^{t \wedge \tau_D} \int a (\lambda, X_s) \, \mathsf{m}_s (d \lambda) \, ds , \quad t \in \bR_+.
	\]
	Thanks to Stricker's lemma (\cite[Theorem~9.19]{jacod79}), the stopped process \(X_{\cdot \wedge \tau_D}\) is a \(P_x\)-\(X^{-1}((\cF_t)_{t \geq 0})\)-semimartingale and its semimartingale characteristics (for the new filtration \(X^{-1} ((\cF_t)_{t \geq 0})\)) are given by the dual predictable projection (\cite[Theorem~9.22]{jacod79}) of \((B, C)\) to the filtration \(X^{-1} ((\cF_t)_{t \geq 0})\), cf.  \cite[Proposition~9.24]{jacod79}.
	Thanks to \cite[(9.23)]{jacod79}, we have \(P_x\)-a.s.
	\begin{align*}
		\Big( \int_0^{\cdot \wedge \tau_D}\int b (\lambda, X_s) \, \mathsf{m}_s (d \lambda) \, ds \Big)^p = \int_0^{\cdot \wedge \tau_D}  \hspace{-0.2cm} {\phantom{\Big(}}^p\Big( \int b (\lambda, X_\cdot) \, \mathsf{m}_\cdot (d \lambda) \Big)_s \, ds, 
		\\ 
		\Big( \int_0^{\cdot \wedge \tau_D} \int a (\lambda, X_s) \, \mathsf{m}_s (d \lambda) \,ds \Big)^p = \int_0^{\cdot \wedge \tau_D} \hspace{-0.2cm} {\phantom{\Big(}}^p\Big( \int a (\lambda, X_\cdot) \, \mathsf{m}_\cdot (d \lambda) \Big)_s \, ds, 
	\end{align*} 
	where \((\, \cdot \, )^p\) denotes the dual predictable projection and \(^p (\, \cdot \,)\) denotes the predictable projection (\cite[Theorem~1.23]{jacod79}), both for the filtration \(X^{-1} ((\cF_t)_{t \geq 0})\). We deduce from \eqref{eq: b indetity imp}, and the definition of the predictable projection, that 
	\begin{align*}
		E^{P_x} \Big[ &\int_0^{\tau_D}  \1 \Big\{  \hspace{-0.15cm} {\phantom{\Big(}}^p\Big( \int b (\lambda, X_\cdot) \,  \mathsf{m}_\cdot (d \lambda) \Big)_s \not = b^* (X_s) \Big\} \, ds\, \Big] 
		\\&= \int_0^\infty P_x \Big( \hspace{-0.15cm} {\phantom{\Big(}}^p\Big( \int b (\lambda, X_\cdot) \,  \mathsf{m}_\cdot (d \lambda) \Big)_s \not = b^* (X_s), s < \tau_D \Big) \, ds
		\\&= \int_0^\infty P_x \Big( E^{P_x} \Big[ \int b (\lambda, X_s) \, \mathsf{m}_s (d \lambda) \mid X^{-1} (\cF_s) \Big] \not = b^* (X_s), s < \tau_D \Big) \, ds
		= 0. \phantom \int 
	\end{align*}
	Here, we also use that \(X^{-1} (\cF_{s-}) = X^{-1} (\cF_s)\), which is due to the continuous paths of \(X\), cf. \cite[Problem~2.7.1]{KaraShre}.
	This implies that \(P_x\)-a.s. 
	\[
	\Big( \int_0^{\cdot \wedge \tau_D} \hspace{-0.15cm}\int b (\lambda, X_s) \, \mathsf{m}_s (d \lambda) \, ds \Big)^p = \int_0^{\cdot \wedge \tau_D} b^* (X_s) \, ds.
	\] 
	Similarly, we get \(P_x\)-a.s. 
	\[
	\Big( \int_0^{\cdot \wedge \tau_D} \hspace{-0.15cm}\int a (\lambda, X_s) \, \mathsf{m}_s (d \lambda) \, ds \Big)^p = \int_0^{\cdot \wedge \tau_D} a^* (X_s) \, ds.
	\] 
	Consequently, from \cite[Theorem~II.2.42]{JS} and \cite[Lemma~2.9]{jacod80}, we get that the stopped processes
	\[
	f (X_{\cdot \wedge \tau_D}) - \int_0^{\cdot \wedge \tau_D} \Big( \tfrac{1}{2} \on{tr} \, \big[ a^* (X_s) \nabla^2 f (X_s) \big] + \langle b^* (X_s), \nabla f (X_s) \rangle \Big) \, ds, \quad f \in C^2_b (\bR^d),  
	\] 
	are \(P_x \circ X^{-1}\)-\((\cF_t)_{t \geq 0}\)-martingales. 
	Now, by \cite[Proposition~5.4.6]{KaraShre}, possibly on an extension of the probability space \((\Omega, \cF, (\cF_t)_{t \geq 0}, P_x \circ X^{-1})\), there exists a Brownian motion \((W_t)_{t \geq 0}\) such that a.s. 
	\begin{align} \label{eq: SDE reprentation}
		d X_{t \wedge \tau_D} = b^* (X_t) \1_{\{t \leq \tau_D\}} \, dt + \sqrt{a^* (X_t)}\1_{\{t \leq \tau_D\}} \, d W_t, \quad X_0 = x. 
	\end{align} 
	
	\noindent
	{\em Step 2:}
	Given \eqref{eq: SDE reprentation}, we know that \(X_{\cdot \wedge \tau_D}\) is a strong Markov process with SDE dynamics whose drift and diffusion coefficients satisfy \(\|b^*\| \leq \b\) and \(a^* \in \mathbb{S}_\delta\). All this brings us close to the setting of \cite{krylov_PTRF_21}. 
	Indeed, we would like to use \cite[Theorem~6.8]{krylov_PTRF_21}, which provides our statement. However, we cannot use it without a discussion, because \cite{krylov_PTRF_21} works with a {\em global} strong Markov process (instead of a stopped one). 
	To overcome this, set \(Y_t := X_{t \wedge \tau_D} + W'_t - W'_{t \wedge \tau_D}\), where \(W'\) is a \(d\)-dimensional Brownian motion independent of \((X, W)\), possibly defined on a further extension of the probability space. The process \(Y - X_0\) satisfies \cite[Assumption~1.1]{krylov_AOP_21} and \cite[Assumption~1.1]{krylov_21} with \(\b\) replaced by \(\b (X_0 + \cdot \,)\) and \(2 \b (X_0 + \cdot \,) / \delta\), respectively. Moreover, \(Y = X\) on \([0, \tau_D]\). A close inspection of the proof of \cite[Theorem~6.8]{krylov_PTRF_21} shows that it relies on results from \cite{krylov_AOP_21, krylov_21} that do not require the strong Markov property and that are applied to the process \(Y\), and on local arguments that use the strong Markov property of the stopped process \(X_{\cdot \wedge \tau_D}\). The only exception is the use of \cite[Corollary~4.9]{krylov_PTRF_21} in the proof of \cite[Lemma~5.2]{krylov_PTRF_21} and in the indicated proof of \cite[Corollary~5.5]{krylov_PTRF_21}. The proof of \cite[Corollary~4.9]{krylov_PTRF_21} uses the semigroup of a global Markov process (see \cite[Theorem~4.8]{krylov_PTRF_21}) and does not apply to \(X_{\cdot \wedge \tau_D}\). In both cases, the application of \cite[Corollary~4.9]{krylov_PTRF_21} can be replaced by the parabolic estimate \cite[Theorem~4.1]{krylov_21} applied to \(Y\). For the proof of \cite[Lemma~5.2]{krylov_PTRF_21}, \cite[Theorem~4.1]{krylov_21} yields that, for every ball \(B_r (z) \subset B_{2R} (x_0)\), every Borel set \(\Gamma \subset C_r (z) := [0, r^2) \times B_r (z)\) and every \(y \in B_r (z)\),
	\[
	E^{P_y} \Big[ \int_0^{\tau_r (z) \wedge r^2} \1_{C_r (z) \setminus \Gamma} (t, X_t) \, dt \Big] \leq N r^{d/(d + 1)} \big( \mu_L (C_r (z)) - \mu_L(\Gamma) \big)^{1/(d + 1)},
	\]
	where \(\tau_r (z) := \inf \{t \geq 0 \colon X_t \not \in B_r (z)\}\), \(\mu_L\) denotes the Lebesgue measure on \(\bR^{d + 1}\) and \(N = N (d, \delta, \|\b\|_{L_d (\bR^d)})\). Hence, the proof of \cite[Lemma~5.2]{krylov_PTRF_21} remains valid with the exponent \(1/d_0\) replaced by \(1/(d + 1)\). The proof of \cite[Corollary~5.5]{krylov_PTRF_21} is similar to that of \cite[Corollary~4.3]{krylov_AOP_21}: the expected time that \((t, X_t)\) spends in the target set before leaving the cylinder is bounded from below by \cite[Theorem~5.4]{krylov_PTRF_21} and, by the strong Markov property at the hitting time of the target set, from above by \cite[Theorem~4.1]{krylov_21} times the hitting probability. This gives a version of \cite[Corollary~5.5]{krylov_PTRF_21} with the exponent \(\mu - 1 / d_0\) replaced by \(\mu - 1 / (d + 1)\), where \(\mu\) and \(d_0\) are as in \cite{krylov_PTRF_21}. This does not affect the proof of \cite[Theorem~5.9]{krylov_PTRF_21}, which uses \cite[Corollary~5.5]{krylov_PTRF_21} only through the fact that its lower bound is bounded away from zero by a constant depending only on \(d, \delta, \|\b\|_{L_d (\bR^d)}\), a parameter \(\kappa\) and a lower bound for the relative measure of the target set in the cylinder. The modified bound has this property, as its constants \(N\) and \(\mu\) depend only on \(d, \delta, \|\b\|_{L_d (\bR^d)}\) and \(\kappa\). Given these observations, we can use \cite[Theorem~6.8]{krylov_PTRF_21} in our setting and conclude the claim. For brevity we will not repeat the proof of \cite[Theorem~6.8]{krylov_PTRF_21} and leave the details to the reader.
\end{proof}

In case \(f \colon \partial D \to \bR\) is a continuous map on the compact set \(\partial D\), there exists an increasing concave modulus of continuity, i.e., an increasing concave function \(w \colon \bR_+ \to \bR_+\) with \(w (r) \searrow w (0) = 0\) as \(r \searrow 0\) and \(| f (x) - f (y) | \leq w (\|x - y\|)\) for all \(x, y \in \partial D\), see \cite[Lemma~12~(ii)]{JKZ_24}. 
As \(D\) satisfies the exterior ball condition, for every \(x \in \partial D\), there exists a radius \(\rho (x) > 0\) and a center \(z \not \in \oD\) such that \(\{ y \in \bR^d \colon \| y - z \| \leq \rho (x)\} \cap \oD = \{x\}\). 

\begin{theorem} \label{theo: krylov boundary}
	Assume that the Standing Assumptions~\ref{SA: main1} and \ref{SA: main2} are in force, and that \(f\) is continuous with increasing concave modulus of continuity \(w\).
	For \(x_0 \in \partial D\)
	there exist two constants \[
	C = C(d, \delta, \|\b\|_{L_d (\bR^d)}, \on{diam}\, (D), \|g^*\|_{\infty}, \rho (x_0)) > 0, \quad \beta = \beta (d, \delta, \|\b\|_{L_d (\bR^d)}) \in (0, 1),\] such that 
	\[
	|u (x) - u (x_0)| \leq C  \| x - x_0 \|^\beta + w \big( C  \|x - x_0\|^{\beta/2} \big)
	\] 
	for all \(x \in D\). 
\end{theorem}

\begin{proof} 
	We adapt the idea behind \cite[Corollary~4.12]{krylov_AOP_21}.
	Throughout the proof, take \(x \in D\). 
	Further, \(C\) and \(\beta\) are constants as in the assertion of the theorem. We use the convention that the constants might change from line to line. 
	Using that \(w\) is concave and increasing, we get from Jensen's inequality that 
	\begin{align*}
		| u (x) - u (x_0) | &= \Big| E^{P_x} \Big[ f (X_{\tau_D}) - f (x_0) + \int_0^{\tau_D} g^* (X_s) \, ds \Big] \Big|
		\\&\leq E^{P_x} \big[ | f (X_{\tau_D}) - f (x_0)| \big] + E^{P_x} \Big[ \int_0^{\tau_D} |g^* (X_s) | \, ds \Big] 
		\\&\leq E^{P_x} \big[ w ( \|X_{\tau_D} - x_0\| ) \big] + E^{P_x} \Big[ \int_0^{\tau_D} |g^* (X_s) | \, ds \Big] 
		\\&\leq w \Big( \sqrt{ E^{P_x} \big[ \|X_{\tau_D} - x_0\|^2 \big] }\, \Big) + E^{P_x} \Big[ \int_0^{\tau_D} |g^* (X_s) | \, ds \Big].
	\end{align*} 
	It is easy to see that, for all \(r \in (0, \rho (x_0))\), \(\mu_L (B_r (x_0) \cap D^c) \geq \mu_L (B_{r/2} (x_0)) = 2^{-d}\mu_L (B_r (x_0))\), where \(\mu_L\) denotes the Lebesgue measure. By this fact, \cite[Theorem~4.11]{krylov_AOP_21} is applicable and we deduce that 
	\[
	E^{P_x} \Big[ \int_0^{\tau_D} |g^* (X_s) | \, ds \Big] \leq C \|x - x_0\|^\beta.
	\] 
	By It\^o's formula, we obtain that 
	\begin{align*}
		E^{P_x}& \big[ \| X_{\tau_D} - x_0 \|^2 \big] 
		\\&= \|x - x_0\|^2 + E^{P_x} \Big[ \int_0^{\tau_D} \int \big(2 \langle X_s - x_0, b (\lambda, X_s)\rangle + \on{tr} \big[ a (\lambda, X_s) \big] \big) \, M (ds, d \lambda) \Big]
		\\&\leq \|x - x_0\|^2 + C E^{P_x} \Big[ \int_0^{\tau_D} \big(\b (X_s) +  1 \big) \, ds \Big].
	\end{align*} 
	Using again \cite[Theorem~4.11]{krylov_AOP_21}, we conclude that 
	\begin{align*}
		E^{P_x} \big[ \| X_{\tau_D} - x_0 \|^2 \big] &\leq \|x - x_0\|^2 + C \|x - x_0\|^\beta 
		\leq C \| x - x_0\|^\beta.
	\end{align*} 
	Putting the pieces together, Theorem~\ref{theo: krylov boundary} is proved. 
\end{proof} 

We are now in a position to prove Theorem~\ref{theo: main Holder continuity}. 

\begin{proof}[Proof of Theorem~\ref{theo: main Holder continuity}]
	Recall that 
	\[
	v (x) = E^{P_x} \Big[ f (X_{\tau_D}) + \int_0^{\tau_D} \int g (\lambda, X_t) \, M (dt, d \lambda) \, \Big],
	\] 
	where \((P_x)_{x \in \oD}\) is as in Theorem~\ref{theo: SMS}. The first step is to show that \(v\) has the same form as the function \(u\) from \eqref{eq: u in proof}. Then, we deduce the claimed continuity properties from the Theorems~\ref{theo: krylov interior} and \ref{theo: krylov boundary}.
	
	\smallskip
	
	\noindent
	{\em Step 1:} 
	The following uses some ideas from~\cite{HausLep90}. 
	Define 
	\begin{align*}
		\Gamma_t &:= v (X_t) + \int_0^t \int g (\lambda, X_s) \, M (ds, d \lambda), \quad t \leq \tau_D, 
		\\
		\Gamma^f &:= f (X_{\tau_D}) + \int_0^{\tau_D} \int g (\lambda, X_t) \, M (dt, d \lambda).
	\end{align*}
	We now prove that \(\Gamma_{\cdot \wedge \tau_D}\) is a \(P_x\)-\((\cG_t)_{t \geq 0}\)-martingale for every \(x \in \oD\). At this point, we use the shift operator \(\theta\) from \eqref{eq: shift}.
	The claim is trivial for \(x \in \partial D\), so take \(x \in D\), \(r \in \bR_+\), and set \(\xi := r \wedge \tau_D\), which is an \((\cF_t)_{t \geq 0}\)-stopping time. 
	Using Lemma~\ref{lem: K initial}, we observe that \(P_x\)-a.s. \(P_x (\theta_{\xi}^{-1} (\, \cdot \,) \mid \cG_{\xi}) \in \cK^* (X_{\xi})\) and hence, \(P_x\)-a.s.
	\[
	v (X_\xi) = E^{P_x (\theta_{\xi}^{-1} ( \, \cdot \,) \, \mid\, \cG_\xi)} \big[ \Gamma^f \big] =  E^{P_x} \big[ \Gamma^{f} \circ \theta_\xi \mid \cG_\xi \big]. 
	\]
	Notice that \(\tau_D \circ \theta_\xi = \tau_D - \xi\) and 
	\begin{equation*} \begin{split}
			\Gamma^f\circ \theta_\xi 
			&= f (X_{\tau_D \, \circ\, \theta_\xi \, +\, \xi}) + \int_{0}^{\tau_D\, \circ\, \theta_\xi} \int g (\lambda, X_{s + \xi}) \, \mathsf{m}_{s + \xi} (d \lambda) \, ds
			\\&= f (X_{\tau_D}) + \int_{\xi}^{\tau_D} \int g (\lambda, X_s) \, \mathsf{m}_s (d \lambda) \, ds.
		\end{split}
	\end{equation*} 
Thus, we get that \(P_x\)-a.s.
\[
\Gamma_\xi = E^{P_{x}} \big[ \Gamma^f \circ \theta_\xi \mid \cG_\xi \big] + \int_{0}^\xi \int g (\lambda, X_s) \, \mathsf{m}_s (d \lambda) \,ds = E^{P_x} \big[ \Gamma^f \mid \cG_{\xi} \big].
\]
As \(\Gamma^f\) is \(P_x\)-integrable by \eqref{eq: bound stopping time}, the tower rule shows that \(\Gamma_{\cdot \wedge \tau_D}\) is a \(P_x\)-\((\cG_{t \wedge \tau_D})_{t \geq 0}\)-martingale and, by elementary considerations, also a \(P_x\)-\((\cG_t)_{t \geq 0}\)-martingale.
	For \(y \in D\), we set 
	\begin{align*}
		g^* (y) := \begin{cases} \liminf_{n \to \infty} n \big( v (y) - E^{P_y} \big[ v (X_{\tau_D \wedge 1/n}) \big] \big), & \text{if the \(\liminf\) is in the}
			\\&\text{interval \([- \|g\|_\infty, \|g\|_\infty]\)},
			\\ 0, &\text{otherwise},\end{cases} 
	\end{align*}
	and \(g^* (z) := 0\) for \(z \in \partial D\).
	Clearly, this \(g^*\) is a bounded Borel function.
	Using the Markov property of \((P_z \circ X^{-1}_{\cdot \wedge \tau_D})_{z \in \oD}\), the \(P_x\)-\((\cG_t)_{t \geq 0}\)-martingale property of \(\Gamma_{\cdot \wedge \tau_D}\), and Fubini's theorem (which we may apply because \(g\) is bounded), we obtain, \(P_x\)-a.s. for Lebesgue a.e. \(t < \tau_D\), 
	\begin{align*}
		n \big( v (X_t) &- E^{P_{X_t}} \big[ v (X_{\tau_D \wedge 1/n})\big] \big) 
		\\&= n \big( v (X_t) - E^{P_x} \big[ v (X_{\tau_D \wedge (t + 1/n)}) \mid X^{-1} (\cF_t) \big] \big) \phantom \int
		\\&= n \Big( v (X_t) - E^{P_x} \big[ \Gamma_{\tau_D \wedge (t + 1/n)} \mid X^{-1} (\cF_t) \big]
		\\&\hspace{1cm}+  E^{P_x} \Big[ \int_0^{\tau_D \wedge (t + 1/n)}\int g (\lambda, X_s) \, \mathsf{m}_s (d \lambda) \, ds \mid X^{-1} (\cF_t) \Big] \Big)
		\\&= n \Big( v (X_t) - E^{P_x} \big[ E^{P_x} \big[ \Gamma_{\tau_D \wedge (t + 1/n)} \mid \cG_t \big] \mid X^{-1} (\cF_t) \big] \phantom \int 
		\\&\hspace{1cm}+ E^{P_x} \Big[ \int_0^{\tau_D \wedge (t + 1/n)} \int g (\lambda, X_s) \, \mathsf{m}_s (d \lambda) \, ds \mid X^{-1} (\cF_t) \Big] \Big)
		\\&= n \Big( v (X_t) - E^{P_x} \big[ \Gamma_{t} \mid X^{-1} (\cF_t) \big] \phantom \int 
		\\&\hspace{1cm}+ E^{P_x} \Big[ \int_0^{\tau_D \wedge (t + 1/n)} \int g (\lambda, X_s) \,  \mathsf{m}_s (d \lambda) \, ds \mid X^{-1} (\cF_t) \Big] \Big)
		\\&=  E^{P_x} \Big[ \, n \int_t^{\tau_D \wedge (t + 1/n)} \int g (\lambda, X_s) \, \mathsf{m}_s (d \lambda) \, ds \mid X^{-1} (\cF_t) \Big]
		\\&= n \int_t^{t + 1/n} E^{P_x} \Big[ \int g (\lambda, X_s) \, \mathsf{m}_s (d \lambda) \1_{\{s \leq \tau_D\}} \mid X^{-1} (\cF_t) \Big] \, ds
		\\&\xrightarrow{\quad} E^{P_x} \Big[ \int g (\lambda, X_t) \, \mathsf{m}_t (d \lambda)\mid X^{-1} (\cF_t) \Big], \quad n \to \infty.
	\end{align*}
	We conclude that, \(P_x\)-a.s. for Lebesgue a.e. \(t < \tau_D\), 
	\[
	g^* (X_t) = E^{P_x} \Big[ \int g (\lambda, X_t) \, \mathsf{m}_t (d \lambda) \mid X^{-1} (\cF_t) \Big].
	\] 
	As in Step~1 of the proof of Theorem~\ref{theo: krylov interior}, we obtain that \(P_x\)-a.s. 
	\begin{align*}
		\Big( \int_0^{\cdot \wedge \tau_D} \int g (\lambda, X_s) \, \mathsf{m}_s (d \lambda) \, ds \Big)^p = \int_0^{\cdot \wedge \tau_D} g^* (X_s) \, ds, 
	\end{align*}
	and, using \cite[Theorem~9.22]{jacod79}, 
	\begin{align*}
		E^{P_x} \Big[ \int_0^{\tau_D}\int g (\lambda, X_s) \, \mathsf{m}_s (d \lambda) \, ds \Big] 
		&= E^{P_x} \Big[ \int_0^{\tau_D} g^* (X_s) \, ds \Big].
	\end{align*}
	{\em Step 2:} Summarizing Step 1, we have 
	\[
	v (x) = E^{P_x} \Big[ f (X_{\tau_D}) + \int_0^{\tau_D} g^* (X_s) \, ds \Big]
	\]
	for a bounded Borel function \(g^*\). In other words, \(v\) has the form \eqref{eq: u in proof}. Consequently, Theorem~\ref{theo: krylov interior} yields that \(v\) is locally \(\alpha\)-H\"older continuous in the interior \(D\) with exponent \(\alpha = \alpha (d, \delta, \|\b\|_{L_d(\bR^d)}) \in (0, 1)\). This completes the proof of (a).
	If \(f\) is even continuous (and not only lower semicontinuous), \(v\) is continuous on \(\oD\) by Theorems~\ref{theo: krylov interior} and \ref{theo: krylov boundary}. 
	
	\smallskip
	{\em Step 3:} Finally, we assume that \(D\) satisfies the uniform exterior ball condition and that \(f\) is continuous with concave, increasing modulus of continuity \(w_f\). 
	We adapt an argument from \cite[Theorem~9.7.1]{krylov_18}. Throughout this step, we replace \(\alpha\) by the minimum of \(\alpha\) and the exponent \(\beta\) from Theorem~\ref{theo: krylov boundary} (not to be confused with the H\"older exponent \(\beta\) of \(f\) in part (c)). This is possible, as \eqref{eq: holder interior} remains true for smaller exponents (with a modified constant), because \(\|x - y\| < 2R\) for \(x, y \in B_R (x_0)\), and the estimate of Theorem~\ref{theo: krylov boundary} remains true for smaller exponents (with a modified constant), because \(\|x - x_0\| \leq \on{diam} (D)\). In particular, as local \(\alpha\)-H\"older continuity implies local \(\alpha'\)-H\"older continuity for \(\alpha' \leq \alpha\), we can take the exponent \(\alpha\) in all parts (a)--(c) of Theorem~\ref{theo: main Holder continuity} to be this minimum, which still depends only on \(d, \delta\) and \(\|\b\|_{L_d (\bR^d)}\).
	Take \(x_i \in \oD\), set \(\rho := \| x_1 - x_2\|\) and \(\rho_i := \on{dist} (x_i, \partial D)\). We assume that \(\rho < 1/4\); for \(\rho \geq 1/4\), the claim follows from \(| v (x_1) - v (x_2) | \leq 2 \sup_{z \in \oD} | v (z) | \leq 2 \sup_{z \in \oD} | v (z) | \, (4 \rho)^{\alpha / 2}\). 
	Let \(y_i\) be a closest point to \(x_i\) on \(\partial D\). 
	We now distinguish two cases: 
	\begin{enumerate}
		\item[(i)] \(\rho \geq \rho^2_1 \wedge \rho^2_2\); 
		\item[(ii)] \(\rho < \rho^2_1 \wedge \rho^2_2\).
	\end{enumerate}  
	For (i), notice that \(| \rho_1 - \rho_2 | \leq \| x_1 - x_2 \| = \rho < \sqrt{\rho}\) and \(\rho_i \leq \rho_1 \wedge \rho_2 + | \rho_1 - \rho_2 | < 2 \sqrt{\rho}\). 
	Using Theorem~\ref{theo: krylov boundary}, we obtain that
	\begin{align*}
		| v (x_1) - v (x_2) | &\leq | v (x_1) - v (y_1) | + | v (y_1) - v (y_2) | + | v (y_2) - v (x_2) | \phantom \int
		\\&= | v (x_1) - v (y_1) | + | f (y_1) - f (y_2) | + | v (y_2) - v (x_2) | \phantom \int 
		\\&\leq \sum_{i = 1}^2 \, ( C \| x_i - y_i \|^\alpha  + w_f ( C \|x_i - y_i\|^{\alpha/2} ) ) + w_f ( \| y_1 - y_2 \| ) 
		\\&\leq  \sum_{i = 1}^2 \, ( C \rho_i^\alpha + w_f (C \rho_i^{\alpha / 2}) ) +  w_f ( \rho_1 + \rho_2 + \rho) 
		\\&\leq  C \rho^{\alpha/2} + 2 w_f (C \rho^{\alpha / 4} )  + w_f ( 5 \rho^{1/2} ) \leq C \rho^{\alpha/2} + 3 w_f (C \rho^{\alpha / 4} ), \phantom \int 
	\end{align*}
	where \(C\) only depends on \(d, \delta, \|\b\|_{L_d (\bR^d)}, \on{diam} (D), \| g \|_\infty\), and the radius from the uniform exterior ball condition (for the last inequality, we take \(C \geq 5\) and use that \(\rho^{1/2} \leq \rho^{\alpha/4}\), as \(\rho < 1\)). 
	
	\noindent
	For (ii), set  \(R = \sqrt{\rho} / 2\) and \(x_0 = x_1\). Then, \(B_{2R} (x_0) \subset \subset D\), since \(2 R < \rho_1\), and 
	\[
	\| x_1 - x_2 \| = \rho <\frac{\sqrt{\rho}}{2} = R, 
	\] 
	since \(\rho <1 / 4\). Hence, we may apply  Theorem~\ref{theo: krylov interior} and obtain that 
	\begin{align*}
		| v (x_1) - v (x_2) | &\leq C R^{- \alpha} \rho^{\alpha} \Big( \sup_{z \in \oD} | v (z) | + \|g\|_\infty \Big)
		= C \rho^{\alpha / 2} \Big( \sup_{z \in \oD} | v(z) | + \| g \|_\infty \Big),  
	\end{align*}
	where the constant \(C\) depends on \(d, \delta, \|\b\|_{L_d(\bR^d)}\) and \(\on{diam}(D)\). This proves the claim from (b). The claim of (c) follows by taking \(w_f (x) = C x^\beta\). 
\end{proof}

We now take a step towards the proof of Theorem~\ref{theo: main continuity}.
As usual, the key ingredient for the viscosity property is the dynamic programming principle, which we provide for our setting in the following lemma.

\begin{lemma}[Dynamic programming principle] \label{lem: DPP} 
	Assume that the Standing Assumptions~\ref{SA: main1} and \ref{SA: main2} are in force.
	For all \((\cF_t)_{t \geq 0}\)-stopping times \(\rho \leq \tau_D\), 
	\begin{align} \label{eq: DPP}
		v (x) = \inf_{P \in \cK (x)} E^P \Big[ v (X_\rho) + \int_0^\rho \int g (\lambda, X_s) \, M (ds, d \lambda) \Big], \quad x \in \oD.
	\end{align}
\end{lemma} 
\begin{proof}
	Recall that \(\theta\) denotes the shift from \eqref{eq: shift} and that \(\gamma\) denotes the right-inverse shift from~\eqref{eq: inverse shift}.
	
	First, by the stability under conditioning property (see Definition~\ref{def: stability}~(i), Lemma~\ref{lem: stability} and Lemma~\ref{lem: shift}~(i)), for every \(P \in \cK (x)\), we have \(P\)-a.s. \(P ( \theta^{-1}_\rho (\, \cdot \,) \mid \cG_\rho) \in \cK (X_\rho)\). Hence, 
	\begin{align*}
		\inf_{P \in \cK (x)} E^P \Big[ &v (X_\rho) + \int_0^\rho \int g (\lambda, X_s) \, M (ds, d \lambda) \Big] 
		\\&\leq \inf_{P \in \cK(x)} E^P \Big[ E^P \Big[ f (X_{\tau_D}) + \int_\rho^{\tau_D} \int g (\lambda, X_s) \, M (ds, d\lambda) \mid \cG_\rho \Big] 
		\\&\hspace{4.05cm}+ \int_0^\rho \int g (\lambda, X_s) \, M(ds, d\lambda) \, \Big] 
		\\&= v (x).
	\end{align*}
	For the converse inequality, by Lemma~\ref{lem: K initial}~(i) and a measurable selection theorem (see \cite[Theorem~18.13]{charalambos2013infinite} or \cite[Theorem~12.1.10]{SV}), there exists a measurable map \(y \mapsto P_y\) with \(P_y \in \cK (y)\) and
	\[
	E^{P_y} \big[ \Gamma^f \big] = v (y), \quad y \in \oD.
	\] 
	Setting
	\[
	Q_{\overline{\omega}} := P_{X_{\rho} (\overline{\omega})} \circ \gamma^{-1}_{\rho (\overline{\omega})}, \qquad \overline{\omega} \in \Omega \times \m,
	\]
	we deduce from Lemmata~\ref{lem: stability} and~\ref{lem: shift} that, for every \(P \in \cK(x)\), \(P \otimes_{\rho} Q \in \cR (x)\), where the pasting notation is taken from \eqref{eq: pasting measure}.
	Using that \(\theta_\rho \circ \gamma_\rho = \on{id}\), 
	we obtain that 
	\begin{align*}
		\inf_{P \in \cR (x)} &E^P \Big[ v (X_{\rho}) + \int_0^{\rho} \int g (\lambda, X_s) \, M (ds, d\lambda) \Big]
		\\&\quad = \inf_{P \in \cR (x)} E^P \Big[ E^{P_{X_{\rho}}} \big[ \Gamma^f\big]  + \int_0^\rho \int g (\lambda, X_s) \, M (ds, d \lambda) \Big]
		\\&\quad = \inf_{P \in \cR (x)} E^P \Big[ E^{Q} \big[ \Gamma^f \circ \theta_{\rho (\, \cdot \,)} \big ] + \int_0^\rho \int g (\lambda, X_s) \, M (ds, d \lambda) \Big]
		\\&\quad = \inf_{P \in \cR (x)} E^P \Big[ E^Q \Big[ f (X_{\tau_D}) + \int_{\rho (\, \cdot\,)}^{\tau_D \, \circ\, \theta_{\rho (\, \cdot\,)} + \rho (\, \cdot \,)} \int g (\lambda, X_s) \, M (ds, d \lambda) \Big] \\&\hspace{4.4cm}+ \int_0^\rho \int g (\lambda, X_s) \, M (ds, d \lambda) \Big]
		\\&\quad= \inf_{P \in \cR (x)} E^{P \otimes_{\rho} Q} \big[ \Gamma^f \big] \geq v (x). \phantom \int
	\end{align*}
	The proof is complete.
\end{proof}

\begin{proof}[Proof of Theorem~\ref{theo: main continuity}]
	As \(f\) is assumed to be continuous, we showed in the proof of Theorem~\ref{theo: main Holder continuity} that \(v\) is continuous. It remains to prove that \(v\) satisfies the viscosity properties. By \cite[Theorem~1.3]{krylov_AOP_21}, there exists a constant \(d_0 = d_0 (d, \delta, \|\b\|_{L_d (\bR^d)}) \in (d/2, d)\) such that an It\^o formula holds for all functions from \(W^{2}_{p, \textup{loc}} (D)\) with \(p \geq d_0\), up to the exit time from subdomains \(U \subset\subset D\), by localization. Increasing \(d_0\) if necessary, but keeping \(d_0 < d\), we may and will assume that \(d_0 > \bar{p}\), where \(\bar{p} = \bar{p} (d, \delta, \|\b\|_{L_d (\bR^d)}) < d\) is the constant from Lemma~\ref{lem: strict maximum} (for the ellipticity constants \(\delta / 2\) and \(1 / (2 \delta)\), see Remark~\ref{rem: HJB appendix form}). It suffices to prove that \(v\) is an \(L_{d_0}\)-viscosity solution: as \(W^2_{p, \textup{loc}} (D) \subset W^2_{d_0, \textup{loc}} (D)\) for \(p \geq d_0\), every \(L_{d_0}\)-viscosity sub- (resp. super-)solution is also an \(L_p\)-viscosity sub- (resp. super-)solution. From now on, take \(p = d_0\).
	
	\smallskip
	{\em Step 1:} We start with the \(L_p\)-viscosity subsolution property. Take \(x_0 \in D\), let \(R > 0\) be such that \(U := B_R (x_0) \subset\subset D\), and let \(\varphi \in W^{2}_{p, \textup{loc}} (D)\) be such that \(v (x_0) = \varphi (x_0)\) and \(v <\varphi\) on \(\overline{U} \setminus \{x_0\}\). 
	By part (a) of our Standing Assumption~\ref{SA: main1} and Standing Assumption~\ref{SA: main2},
	\[
	(\lambda, x) \mapsto L (\lambda, x, \nabla \varphi (x), \nabla^2 \varphi (x)) + g (\lambda, x)
	\] 
	is continuous in the first and measurable in the second variable. Hence, as \(\Lambda\) is compact, the measurable maximum theorem \cite[Theorem~18.19]{charalambos2013infinite} yields the existence of a measurable map \(x \mapsto \lambda (x)\), from \(D\) into \(\Lambda\), such that 
	\[ 
	L (\lambda (x), x, \nabla \varphi (x), \nabla^2 \varphi (x)) + g (\lambda (x), x) = H (x, \nabla \varphi (x), \nabla^2 \varphi (x)), \quad x \in D. 
	\] 
	By \cite[Theorem~1.1]{krylov_AOP_21}, there exists an \(\bR^d \)-valued continuous It\^o process \((Y_t)_{t \geq 0}\) and a \(d\)-dimensional Brownian motion \((W_t)_{t \geq 0}\) such that \(Y_0 = x_0\) and 
	\begin{align*}
		d Y_t = ( b (\lambda (Y_t), Y_t) \1_{\{Y_t \in D\}} &+ b (\lambda_0, Y_t) \1_{\{Y_t \not \in D\}}) \, dt 
		\\&+ ( \sqrt{a (\lambda (Y_t), Y_t)} \1_{\{Y_t \in D\}} + \sqrt{a (\lambda_0, Y_t)} \1_{\{Y_t \not \in D\}} ) \, d W_t,
	\end{align*} 
	where \(\lambda_0 \in \Lambda\) is fixed but arbitrary. Set \(\lambda' (y) := \lambda (y)\) for \(y \in D\) and \(\lambda' (y) := \lambda_0\) for \(y \not \in D\), and let \(P'\) be the law of \((Y, \delta_{\lambda' (Y_t)} (d \lambda) \, dt)\), seen as a probability measure on the canonical product space \((\Omega \times \m, \cF\otimes \mathcal{M})\). Clearly, we have \(P' \in \cK(x_0)\).
	By the dynamic programming principle (Lemma~\ref{lem: DPP}), the It\^o formula for Sobolev functions \cite[Theorem~1.3]{krylov_AOP_21}, and \cite[Theorem~2.8]{krylov_AOP_21}, 
	\begin{equation} \label{eq: positive shown} \begin{split}
			0 &= \inf_{Q \in \cK (x_0)} E^Q \Big[ v (X_{\tau_U}) + \int_0^{\tau_U} \int g (\lambda, X_s) \, M (ds, d \lambda) \Big] - v (x_0)
			\\&< E^{P'} \Big[ \varphi (X_{\tau_U}) - \varphi (x_0) + \int_0^{\tau_U} \int g (\lambda, X_s) \, M (ds, d \lambda) \Big]
			\\&= E^{P'} \Big[ \int_0^{\tau_U} \big( L (\lambda (X_s), X_s, \nabla \varphi (X_s), \nabla^2 \varphi (X_s)) + g (\lambda (X_s), X_s) \big) \, ds \Big] 
			\\&= E^{P'} \Big[ \int_0^{\tau_U} H (X_s, \nabla \varphi (X_s), \nabla^2 \varphi (X_s)) \, ds \Big]
			\\&\leq C\, \| H (\, \cdot \,, \nabla \varphi, \nabla^2 \varphi)_+ \|_{L_{d_0} (U)}, \quad y_+ := \max \{y, 0\}, \phantom \int 
			\\&\leq C\, \underset{z \in U}{\on{ess \, sup}} \, H (z , \nabla \varphi (z), \nabla^2 \varphi (z))_+, \phantom \int 
		\end{split}
	\end{equation} 
	where the constant \(C > 0\) depends on \(d, \delta, \|\b\|_{L_d (\bR^d)}\) and \(\on{diam}\, (D)\), and it changed from the second to last line.
	We note the following fact: For a measurable function \(\ell \colon D \to \bR\), \(\on{ess \, sup}_{z \in U} \, \ell_+ (z) > 0\) implies that \(\on{ess \, sup}_{z \in U}\, \ell (z) > 0\). Indeed, assume for contradiction that \(\alpha := \on{ess \, sup}_{z \in U} \, \ell (z) \leq 0\). Then, by definition of the essential supremum, the set \(\{ \ell > 0 \}\) has (Lebesgue) measure zero. Consequently, as \(\{ \ell_+ > 0 \} = \{ \ell > 0 \}\), we arrive at the contradiction \(\on{ess \, sup}_{z \in U} \, \ell_+ (z) \leq 0\), showing that we must have \(\alpha > 0\). Combining this observation with \eqref{eq: positive shown} shows that 
		\[
		\underset{z \in U}{\on{ess \, sup}} \, H (z , \nabla \varphi (z), \nabla^2 \varphi (z)) > 0. 
		\] 
		Finally, using this for \(B_r (x_0)\) with \(r < R\) instead of \(U = B_R (x_0)\), we obtain that 
		\[
		\lim_{r \searrow 0} \underset{z \in B_r (x_0)}{\on{ess \, sup}} \, H (z , \nabla \varphi (z), \nabla^2 \varphi (z)) \geq 0.
		\] 
		This proves that \(v\) is an \(L_{p}\)-subsolution (see Lemma~\ref{lem: strict maximum} in Appendix~\ref{app: recap elliptic}).
	
	\smallskip
	{\em Step 2:} Next, we prove that \(v\) is also an \(L_p\)-supersolution. 
	Take \(x_0 \in D\), let \(R > 0\) such that \(U := B_R (x_0) \subset\subset D\), and \(\varphi \in W^{2}_{p, \textup{loc}} (D)\) such that \(v (x_0) = \varphi (x_0)\) and \(v > \varphi\) on \(\overline{U} \setminus \{x_0\}\). 
	Recall from Lemma~\ref{lem: upper hemi} that \(\cK(x_0)\) is compact. Further, using Lemma~\ref{lem: lower expectation} with \(U\) instead of \(D\) and \(v\) instead of \(f\), we obtain that 
	\[
	P \mapsto E^P \Big[ v (X_{\tau_U}) + \int_0^{\tau_U} \int g (\lambda, X_s) \, M(ds, d \lambda) \Big] 
	\] 
	is lower semicontinuous. Here, we use that \(v\) is continuous on \(\overline{U}\). As real-valued lower semicontinuous functions attain their infima on compact sets, there exists a measure \(P' \in \cK(x_0)\) such that 
	\begin{equation*} \begin{split}
			\inf_{Q \in \cK (x_0)} E^Q \Big[ &v (X_{\tau_U}) + \int_0^{\tau_U} \int g (\lambda, X_s) \, M (ds, d \lambda) \Big] 
			\\&=
			E^{P'} \Big[ v(X_{\tau_U}) + \int_0^{\tau_U} \int g (\lambda, X_s) \, M (ds, d \lambda) \Big].
		\end{split}
	\end{equation*} 
	Now, as in Step 1, we obtain 
	\begin{align*}
		0 &= v (x_0) - \inf_{Q \in \cK (x_0)} E^Q \Big[ v (X_{\tau_U}) + \int_0^{\tau_U} \int g (\lambda, X_s) \, M (ds, d \lambda) \Big]
		\\&= \varphi (x_0) - E^{P'} \Big[ v (X_{\tau_U}) + \int_0^{\tau_U} \int g (\lambda, X_s) \, M (ds, d \lambda) \Big]
		\\&< - E^{P'} \Big[ \int_0^{\tau_U} \int \big( L (\lambda , X_s, \nabla \varphi (X_s), \nabla^2 \varphi (X_s)) + g (\lambda, X_s) \big) \, M (ds, d \lambda) \Big] 
		\\&\leq - E^{P'} \Big[ \int_0^{\tau_U} H (X_s, \nabla \varphi (X_s), \nabla^2 \varphi (X_s)) \, ds \Big]
		\\&\leq E^{P'} \Big[ \int_0^{\tau_U} H (X_s, \nabla \varphi (X_s), \nabla^2 \varphi (X_s))_- \, ds \Big], \quad y_- := \max \{- y, 0\}, 
		\\&\leq C\, \| H (\, \cdot \, , \nabla \varphi, \nabla^2 \varphi )_- \|_{L_{d_0} (U)} \phantom \int 
		\\ &\leq C\, \underset{z \in U}{\on{ess \, sup}} \, H (z , \nabla \varphi (z), \nabla^2 \varphi (z))_-, \phantom \int 
	\end{align*}
	where the constant has changed in the second to last line. 
	As in the final part of Step~1, this entails that
	\[
	\underset{z \in U}{\on{ess \, sup}} \, ( - H (z , \nabla \varphi (z), \nabla^2 \varphi (z))) > 0, 
	\] 
	and hence, 
	\[
	\underset{z \in U}{\on{ess \, inf}} \, H (z , \nabla \varphi (z), \nabla^2 \varphi (z)) = -	\, \underset{z \in U}{\on{ess \, sup}} \, ( - H (z , \nabla \varphi (z), \nabla^2 \varphi (z))) < 0. 
	\] 
	 Finally, applying this observation with \(U = B_r (x_0)\) for \(r < R\) instead of \(U = B_R (x_0)\), we conclude that 
		\[
		\lim_{r \searrow 0} \underset{z \in B_r (x_0)}{\on{ess \, inf}} \, H (z , \nabla \varphi (z), \nabla^2 \varphi(z)) \leq 0,
		\]
		which
	shows that \(v\) is an \(L_p\)-supersolution (see Lemma~\ref{lem: strict maximum} in Appendix~\ref{app: recap elliptic}). The proof is complete. 
\end{proof} 

\section{Proofs of Theorems~\ref{theo: semigroup}, \ref{theo: strong Feller} and \ref{theo: joint regu semigroup}} \label{sec: pf main semigroup}
\subsection{Proofs of Theorems~\ref{theo: semigroup} and \ref{theo: strong Feller}}
In the following we prove Theorems~\ref{theo: semigroup} and \ref{theo: strong Feller}, leaving out some details that follow as in Section~\ref{sec: pf SMS}. 

The semigroup property from Theorem~\ref{theo: semigroup} is a consequence of the dynamic programming principle for \(u (t, x) := S_t (f) (x)\), i.e., for all \((\cF_t)_{t \geq 0}\)-stopping times \(\rho \leq t\),  
\begin{align} \label{eq: para DPP}
	u(t, x) = \inf_{P \in \cR(x)} E^P \Big[ u (t - \rho, X_{\rho}) + \int_0^\rho \int g (\lambda, X_s) \, M (ds, d \lambda) \Big]. 
\end{align} 
This identity can be proved in the same way as Lemma~\ref{lem: DPP} and we omit the proof for brevity. 
As a consequence, we obtain that 
\begin{align*}
	S_{t + s} (f) (x) =  \inf_{P \in \cR(x)} E^P \Big[ S_t (f) (X_{s}) + \int_0^s \int g (\lambda, X_v) \, M (dv,d \lambda) \Big] = S_s (S_t (f)) (x). 
\end{align*}
To conclude Theorem~\ref{theo: semigroup}, it remains to establish the \(C_b\)-Feller property, i.e., \(S_t (f) \in C_b (\bR^d)\) for all \(f \in C_b (\bR^d)\) and \(t \geq 0\). Of course, \(S_0 = \on{id}\), so that we only have to consider \(t > 0\). 
For these cases the property follows by the same arguments used to prove the regularity of \eqref{eq: value function D} in Theorem~\ref{theo: main Holder continuity}. 
The key tools are again a strong Markov selection and analytic estimates.

To give the idea, take a deterministic time \(\xi > 0\) and a bounded lower semicontinuous function \(f \colon \bR^d \to \bR\).
In the spirit of Theorem~\ref{theo: SMS}, there exists a time-inhomogeneous family \((P_{t, x})_{(t, x) \in \bR_+ \times \bR^d}\) such that \(P_{t, x} \in \cK (t, x)\),  the family \((P_{t, x} \circ X^{-1})_{(t, x) \in \bR_+ \times \bR^d}\) is strongly Markov, and
\begin{equation} \label{eq: SMSP parabolic} \begin{split}
		v (t, x) := \inf_{P \in \cK (t, x)} E^P \Big[ f & (X_\xi) + \int_t^\xi \int g (\lambda, X_s) \, M(ds, d\lambda) \Big] 
		\\&= E^{P_{t, x}} \Big[ f (X_\xi) + \int_t^\xi \int g (\lambda, X_s) \, M(ds, d\lambda) \Big].
	\end{split}
\end{equation} 
This strong Markov selection principle can be proved similarly to Theorem~\ref{theo: SMS}. In fact, the proof is slightly less technical, as the cost function does not depend on the stopping time \(\tau_D\). A result in the same spirit is given in \cite[Theorem~7.4]{nicole1987compactification} for a bounded drift coefficient \(b\). We omit a detailed proof for brevity. 
Given \eqref{eq: SMSP parabolic}, we may conclude \(S_\xi (f) \in C_b (\bR^d)\) once we prove that
\[
x \mapsto E^{P_{0, x}} \Big[ f (X_\xi) + \int_0^\xi \int g (\lambda, X_s) \, M(ds, d\lambda) \Big]
\]
is continuous. As in the proof of Theorem~\ref{theo: krylov interior}, one can show that there are Borel functions \(b^* \colon \bR_+ \times \bR^d \to \bR^d\) and \(a^* \colon \bR_+ \times \bR^d \to \mathbb{S}_\delta\) with \(\| b^* (t, x) \| \leq \b (x)\) and, under the measure \(P_{t, x}\), 
\[
X_{t + s} = x + \int_0^s b^* (t + r, X_{t + r}) \, dr + \int_0^s \sqrt{a^* (t + r, X_{t + r})} \, d W_r,
\] 
for a \(d\)-dimensional Brownian motion \((W_t)_{t \geq 0}\). Moreover, as in the proof of Theorem~\ref{theo: main Holder continuity}, there exists a bounded Borel function \(g^* \colon [0, \xi] \times \bR^d \to \bR\) such that 
\begin{align} \label{eq: g^' time}
	E^{P_{t, x}} \Big[ f (X_\xi) + \int_t^\xi \int g (\lambda, X_s) \, M(ds, d\lambda) \Big] = E^{P_{t, x}} \Big[ f (X_\xi) + \int_t^\xi g^* (s, X_s) \, ds \Big].
\end{align} 
Let us sketch the idea of this step in some detail. For each \((t, x) \in [0, \xi) \times \bR^d\), the process
\[
\Gamma_s := v (s, X_s) + \int_t^s \int g (\lambda, X_v) \, M (dv, d \lambda), \quad s \in [t, \xi], 
\] 
is a \(P_{t, x}\)-martingale. To see this, take \(r \in [t, \xi]\) and notice that, by an analog of Lemma~\ref{lem: K initial}~(ii), \(P_{t, x}\)-a.s. 
\[
v (r, X_r) = E^{P_{t, x}} \Big[ f (X_{\xi}) + \int_{r}^\xi \int g (\lambda, X_v) \, M (dv, d \lambda) \mid \cG_r \Big],
\] 
and consequently,
\begin{align*}
\Gamma_r &= E^{P_{t, x}} \Big[ f (X_{\xi}) + \int_{r}^\xi \int g (\lambda, X_v) \, M (dv, d \lambda) \mid \cG_r \Big] + \int_t^r \int g (\lambda, X_v) \, M (dv, d \lambda)
\\&= E^{P_{t, x}} \Big[ f (X_{\xi}) + \int_{t}^\xi \int g (\lambda, X_v) \, M (dv, d \lambda) \mid \cG_r \Big], 
\end{align*} 
which proves the claimed martingale property. Using this martingale property, the Markov property of the family \((P_{s, y} \circ X^{-1})_{(s, y) \in \bR_+ \times \bR^d}\) and Fubini's theorem, we obtain, for all \(t, n\) such that \(t + 1/n \leq \xi\) and all \(t_0 \in [0, t]\), \(P_{t_0, x}\)-a.s. 
\begin{align*}
	E^{P_{t, X_t}} \big[ v &(t + 1/n, X_{t + 1/n}) \big] 
	\\&= E^{P_{t_0, x}} \big[ v (t + 1/n, X_{t + 1/n}) \mid X^{-1} (\cF_t) \big] 
	\\&= v (t, X_t) - \int_t^{t + 1/n} E^{P_{t_0,x}} \Big[ \int g (\lambda, X_s) \, \mathsf{m}_s (d \lambda) \mid X^{-1} (\cF_t) \Big] \, ds.
\end{align*} 
Now, as in the proof of Theorem~\ref{theo: main Holder continuity}, it follows that equation \eqref{eq: g^' time} holds with
\[
g^* (t, x) := \begin{cases}\liminf_{n \to \infty} n \big( v (t, x) - E^{P_{t, x}} \big[ v (t + 1/n, X_{t + 1/n}) \big] \big), & \text{if the \(\liminf\)}\\ &\text{lies in } [- \|g\|_\infty, \|g\|_\infty], \\ 0, & \text{otherwise}. \end{cases} 
\] 
We are ready to finish the proofs of Theorems~\ref{theo: semigroup} and \ref{theo: strong Feller}.
Fix \(\widehat{x} \in \bR^d\) and \(R < \sqrt{\xi}/2\).
For \(\varrho > 0\), let \(C_{\varrho} (\widehat{x}) := [0, \varrho^2) \times B_{\varrho} (\widehat{x})\) and \(\overline{C}_{\varrho} (\widehat{x}) := [0, \varrho^2] \times \overline{B}_{\varrho} (\widehat{x})\), and for \((t, x) \in [0, \xi] \times \bR^d\), set 
\begin{align*}
	\widehat{f} (t, x) &:= E^{P_{t, x}} \Big[ f (X_\xi) + \int_t^\xi g^* (s, X_s) \, ds \Big], \\
	\widehat{\tau}_{t, x} &:= \inf \{ s \geq 0 \colon (t + s, X_{t + s}) \not \in C_{2R} (\widehat{x}) \},
\end{align*}
and 
\[
\widehat{v} (t, x) := E^{P_{t, x}} \Big[ \widehat{f} (t + \widehat{\tau}_{t, x}, X_{t + \widehat{\tau}_{t, x}} ) + \int_0^{\widehat{\tau}_{t, x}} g^* (t + s, X_{t + s}) \, ds \Big]. 
\] 
By the strong Markov property of \((P_{s, y} \circ X^{-1})_{(s, y) \in \bR_+ \times \bR^d}\) and the fact that \(t + \widehat{\tau}_{t, x} \leq \xi\), we obtain that
\begin{align*}
	\widehat{v} (t, x) & = E^{P_{t, x}} \Big[ E^{P_{t + \widehat{\tau}_{t, x}, X (t + {\widehat{\tau}}_{t, x})}} \Big[ f (X_\xi) + \int_{t + \widehat{\tau}_{t, x} (\, \cdot \,)}^\xi g^* (s, X_s) \, ds \Big] + \int_0^{\widehat{\tau}_{t, x}} g^* (t + s, X_{t + s}) \, ds  \Big] 
	\\& \ = E^{P_{t, x}} \Big[ E^{P_{t, x}} \Big[ f (X_\xi) + \int_{t + \widehat{\tau}_{t, x}}^\xi g^* (s, X_s) \, ds \mid X^{-1}(\cF_{t + \widehat{\tau}_{t, x}}) \Big] + \int_t^{t + \widehat{\tau}_{t, x}} g^* (s, X_s) \, ds\Big] 
	\\&\ = E^{P_{t, x}} \Big[ f (X_\xi) + \int_t^\xi g^* (s, X_s) \, ds \Big], 
\end{align*} 
where \(\widehat{\tau}_{t, x} (\, \cdot \,)\) means that the integration is taken with respect to the outer expectation. Hence, we have \(\widehat{v} (0, x) = S_\xi (f) (x)\), meaning that continuity properties of \(\widehat{v}\) propagate to \(S_\xi (f)\).
The proof of the following lemma is delegated to the end of this section.

\begin{lemma} \label{lem: parabolic Holder}
Assume that the Standing Assumptions~\ref{SA: main1} and \ref{SA: main 3} are in force.
	Then, there are three constants \(\alpha = \alpha (d, \delta) \in (0, 1)\), \(C =\) \(C (d, \delta) > 0\) and \(\underline{R} = \underline{R} (d, \delta, \| \b\|_{L_d (\bR^d)}, \| \b \|_{L_{d + 1}(\bR^d)}) \in (0, 1)\) such that, for every \(R \in (0, \underline{R}]\) with \(R < \sqrt{\xi} / 2\), 
	\begin{equation} \label{eq: holder bound}\begin{split}
			|\widehat{v} (t, x) - \widehat{v} (s, y) | \leq C \, R^{- \alpha} \, \Big( |t - s|^{1/2} &+ \|x - y\| \Big)^\alpha \\&\cdot \Big( \sup_{z \in \overline{C}_{2R} (\widehat{x})} | \widehat{v} (z) | + R^2 \sup_{z \in C_{2R} (\widehat{x})} | g^* (z) | \Big),
	\end{split}\end{equation}
	for all \((t, x), (s, y) \in C_R (\widehat{x})\). 
\end{lemma}

Lemma~\ref{lem: parabolic Holder} yields that the function \(\widehat{v}\) is H\"older continuous on \(C_{R} (\widehat{x})\) for all \(R \in (0, \underline{R}]\) with \(R < \sqrt{\xi} / 2\).
As \(\widehat{x}\) was arbitrary, we conclude that \(S_\xi (f) \in C_b (\bR^d)\), finishing the proof of Theorem~\ref{theo: semigroup}.
As we have only assumed that \(f\) is lower semicontinuous, the argument also proves Theorem~\ref{theo: strong Feller}.
\qed

\medskip 
	\begin{proof}[Proof of Lemma~\ref{lem: parabolic Holder}]
		Using the notation from \cite{krylov_book_arxiv}, let \(\mathbb{C}_R\) be the collection of all cylinders \((s, y) + C_R (0)\). 
		By \cite[Theorem~4.1]{krylov_21}, we obtain that there exists a constant \(N = N (d, \delta, \|\b\|_{L_d (\bR^d)}) \in (0, \infty)\) such that
		\begin{align*}
			E^{P_{t, x}} \Big[ \int_t^{\rho_{C} + t} \| b^* (s, X_s) \| \, ds \Big] \leq N R^{1 + 1/ (d + 1)} \| \b \|_{L_{d + 1} (\bR^d)}, 
		\end{align*}
		where \(\rho_{C} := \inf \{ s \geq 0 \colon (t + s, X_{t + s}) \not \in C \}\) for any \(C \in \mathbb{C}_R\). Consequently, we obtain that 
		\begin{align*}
			R^{-1} \sup_{\substack{C \in \mathbb{C}_R \\ (t, x) \in \bR_+ \times \bR^d}} E^{P_{t, x}} \Big[ \int_t^{\rho_{C} + t} \| b^* (s, X_s) \| \, ds \Big] \leq N R^{1 / (d + 1)} \| \b \|_{L_{d + 1} (\bR^d)}.
		\end{align*}
		Let \(\underline{R} \in (0, 1)\) be small enough such that the right-hand side, with \(R\) replaced by \(2 \underline{R}\), is smaller than the positive constant \(\mathsf{b}_0 = \mathsf{b}_0 (d, \delta)\) from \cite[Theorem~3, p.~37]{krylov_book_arxiv}. Then, \((P_{t, x})_{(t, x) \in \bR_+ \times \bR^d}\) is a regular diffusion process corresponding to \(a^*, b^*\) with \(\rho_b = 2 \underline{R}\) as defined in \cite[Definition~1, p.~75]{krylov_book_arxiv} and \cite[Theorem~2, p.~81]{krylov_book_arxiv} yields the claim of Lemma~\ref{lem: parabolic Holder}.
\end{proof}

\subsection{Proof of Theorem~\ref{theo: joint regu semigroup}}
(a) 
This result follows from good estimates in time uniformly in space, adapting the proof of \cite[Theorem~2.36]{CN22b}. 
By the continuity in space that is provided by Theorem~\ref{theo: semigroup}, it suffices to prove that, for every compact set \(K \subset \bR^d\), 
\[
\lim_{s \to t} \sup_{x \in K}| S_s (f) (x) - S_t (f) (x) | = 0.
\] 
We have 
\begin{align*}
	\sup_{x \in K}| S_s (f) (x) - S_t (f) (x) | &\leq \sup_{x \in K} \sup_{P \in \cK (x)} E^P \big[ | f (X_s) - f (X_t) | \big] + \|g\|_\infty |s - t|.
\end{align*}
Let us investigate the first term. Take a large time horizon \(T > 0\) such that \(s, t \leq T\).
In the following, \(C\) denotes a generic constant that only depends on \(d, \delta, \|\b\|_{L_{d} (\bR^d)}\), \(T\) and \(\sup_{x \in K} \|x\|\). As usual, the constant may change from line to line.
 By \cite[Corollary~1.2]{krylov_21} and Chebyshev's inequality, for all \(R > 0\), 
\[
\sup_{x \in K} \sup_{P \in \cK(x)} P \Big( \sup_{v \in [0, T]} \| X_v \| > R \Big) \leq \frac{C}{R}.
\] 
For every \(\varepsilon > 0\), using once again \cite[Corollary~1.2]{krylov_21} and Chebyshev's inequality, we obtain that 
\begin{align*}
	\sup_{x \in K} &\sup_{P \in \cK (x)} E^P \big[ | f (X_s) - f (X_t) | \big] 
	\\&\leq \frac{2 C \|f\|_\infty}{R} + \sup_{x \in K} \sup_{P \in \cK (x)} E^P \big[ | f (X_s) - f (X_t) | \1_{\{\sup_{v \in [0, T]}\|X_v\| \leq R\}} \big]
	\\&\leq \frac{2 C \|f\|_\infty}{R} + \sup\, \{ | f (x) - f(y)| \colon x, y \in B_R, \|x -y\| \leq \varepsilon \} \\ &\hspace{6cm}+ 2 \|f\|_\infty \sup_{x \in K} \sup_{P \in \cR(x)} P ( \|X_s - X_t\| > \varepsilon )
	\\&\leq \frac{2 C \|f\|_\infty}{R} + \sup\, \{ | f (x) - f(y)| \colon x, y \in B_R, \|x -y\| \leq \varepsilon \} + \frac{2 \|f\|_\infty C |t - s|^{1/2}}{\varepsilon}.
\end{align*}
Now, choosing first \(R > 0\) large and then \(\varepsilon > 0\) small yields that 
\[
\sup_{x \in K} \sup_{P \in \cK (x)} E^P \big[ | f (X_s) - f (X_t) | \big] 
\]
can be made arbitrarily small as \(s \to t\). This completes the proof of (a).

\smallskip 
(b), (c) Suppose now that \(f\) is bounded and uniformly continuous with concave, increasing modulus of continuity \(w_f\). 
We deduce from \cite[Corollary~1.2]{krylov_21} and Jensen's inequality that 
\begin{equation} \label{eq: time holder}
	\begin{split} 
		\sup_{x \in \bR^d}| S_s (f) (x) - S_t (f) (x) | &\leq \sup_{x \in \bR^d} \sup_{P \in \cK (x)} E^P \big[ | f (X_s) - f (X_t) | \big] + \|g\|_\infty |s - t|
		\\&\leq \sup_{x \in \bR^d} \sup_{P \in \cK (x)} w_f \Big( E^P \big[  \| X_s - X_t \| \big] \Big) + \|g\|_\infty |s - t|
		\\&\leq w_f ( C |s - t|^{1/2}) + \|g\|_\infty |s - t|, 
	\end{split}
\end{equation} 
where \(C\) only depends on \(d, \delta\) and \(\|\b\|_{L_{d} (\bR^d)}\). 
This gives us good control over the time domain. 
To complete the proof, we also need good control over the space domain. 
Take a time horizon \(T > 0\), \(0 \leq t_1 \leq t_2 \leq T\), \(x_1, x_2 \in \bR^d\) and set \(\rho := |t_1 - t_2| + \| x_1 - x_2 \|\) and \(R := \sqrt{\rho}\). 
Let \(\underline{R} = \underline{R} (d, \delta, \| \b \|_{L_d (\bR^d)}, \| \b \|_{L_{d + 1} (\bR^d)}) \in (0, 1)\) be the constant from Lemma~\ref{lem: parabolic Holder} and assume that \(\rho \in (0,  \underline{R}^2]\); for \(\rho > \underline{R}^2\), the claim follows from \(|S_{t_1} (f) (x_1) - S_{t_2} (f) (x_2)| \leq 2 (\|f\|_\infty + T \|g\|_\infty) \leq 2 (\|f\|_\infty + T \|g\|_\infty) \, \underline{R}^{- \alpha} \rho^{\alpha / 2}\). We distinguish two cases:
\begin{enumerate}
	\item[(i)] \(t_2 > 4 R^2\); 
	\item[(ii)] \(t_2 \leq 4 R^2\).
\end{enumerate}
For (i), we have \(R < \sqrt{t_2} / 2\) and 
\(
\| x_1 - x_2 \| \leq \rho < \sqrt{\rho} = R.
\) 
Hence, \eqref{eq: time holder} and Lemma~\ref{lem: parabolic Holder} yield that
\begin{align*}
	| S_{t_1} (f) (x_1) - S_{t_2} (f) (x_2) | &\leq | S_{t_1} (f) (x_1) - S_{t_2} (f) (x_1) | + | S_{t_2} (f) (x_1) - S_{t_2} (f) (x_2) | 
	\\&\leq w_f (C \rho^{1/2}) + C \rho + C R^{- \alpha} \rho^{\alpha} \leq w_f (C \rho^{1 /2 }) + C \rho^{\alpha / 2}, 
\end{align*}
where \(C\) depends on \(d, \delta, \|\b\|_{L_{d}(\bR^d)}, \|f\|_\infty, \|g\|_\infty\) and the time horizon \(T\).

\noindent
For (ii), \eqref{eq: time holder} yields that 
\begin{align*}
	| S_{t_1} (f) (x_1) &- S_{t_2} (f) (x_2) | 
	\\&\leq | S_{t_1} (f) (x_1) - f (x_1) | + | f (x_1) - f (x_2) | + | f (x_2) - S_{t_2} (f) (x_2) | 
	\\&\leq w_f (C t_1^{1 / 2}) + C t_1 + w_f (\rho ) + w_f (Ct_2^{1 / 2}) + C t_2
	\\&\leq C w_f (C \rho^{1 / 2}) + C \rho^{\alpha / 2}, 
\end{align*}
where the constant \(C\) depends on \(d, \delta\), \(\|g\|_\infty\) and \(\|\b\|_{L_{d}(\bR^d)}\).

Putting these estimates together yields the claim of (b). Further, applying (b) with \(w_f (x) = C x^\beta\) yields the claim of (c) for \(\rho \leq 1\); for \(\rho > 1\), it follows from the trivial bound
\[
|S_{t_1} (f) (x_1) - S_{t_2} (f) (x_2)| \leq 2 (\|f\|_\infty + T \|g\|_\infty) \leq 2 (\|f\|_\infty + T \|g\|_\infty) \, \rho^{(\alpha \wedge \beta) / 2}. 
\]
This completes the proof. \qed

\section{Proof of Theorem~\ref{theo: parabolic PDE}} \label{sec: pf parabolic viscosity}
As in the proof of Theorem~\ref{theo: main continuity}, the main tools for the proof are an It\^o formula for Sobolev functions and a weak existence theorem for certain stochastic differential equations with time-inhomogeneous coefficients.

	\begin{theorem} \label{theo: ito formula}
		Let \(Y\) be a \(d\)-dimensional It\^o process with dynamics	
		\begin{align*} 
			Y = \int_0^\cdot b_s \, ds + \int_0^\cdot \sigma_s \, d W_s, \quad \|b_s\| \leq \b (Y_s), \quad \sigma_s \sigma^*_s \in \mathbb{S}_\delta,
		\end{align*} 
		for a \(d\)-dimensional Brownian motion \(W\), where \(\b \in L_{d + 2, \on{loc}} (\bR^d) \cap L_d (\bR^d)\).
		Then, for every \(u \in W^{1, 2}_{d + 1} (C_R (0))\) with \(C_R (0) = [0, R^2) \times B_R (0)\), we have a.s., for all \(t \in \bR_+\), 
		\begin{align*}
			u (t &\wedge \rho_R, Y_{t \wedge \rho_R}) \\&=  u (0, 0) + \int_0^{t \wedge \rho_R} \sum_{i, k = 1}^d \sigma^{ik}_s \partial_{x_i}  u (s,  Y_s) \, d W^k_s
			\\&\qquad + \int_0^{t \wedge \rho_R} \Big( \partial_t u (s, Y_s) + \tfrac{1}{2} \on{tr} \big[ \sigma_s \sigma^*_s \nabla^2 u (s, Y_s) \big] + \langle b_s, \nabla u (s, Y_s) \rangle \Big) \, ds,
		\end{align*} 
		where the stochastic It\^o integral is a square-integrable martingale and \(\rho_R\) denotes the exit time of \((t, Y_t)\) from the cylinder \(C_R (0)\), i.e., \(\rho_R := \inf \{t \geq 0 \colon (t, Y_t) \not \in C_R (0) \}\).
	\end{theorem} 
	
	\begin{proof}
		As on p.~4 in \cite{KS_26}, by H\"older's inequality and parabolic Sobolev embeddings, 
		\[
		\| \b \, \| \nabla u \| \|_{L_{d + 1} (C_R (0))} \leq C R^{2/ (d + 2)} \| \b \|_{L_{d + 2} (B_R (0))} \| u \|_{W^{1, 2}_{d + 1} (C_R (0))}.
		\] 
		In light of this inequality, and the parabolic Krylov estimate from \cite[Theorem~4.1]{krylov_21}, the formula follows precisely as in the proof of \cite[Theorem~3, p.~8]{krylov_book_arxiv}.
		
		Alternatively, the formula can be deduced directly from \cite[Theorem~3, p.~8]{krylov_book_arxiv}, applied to the process 
		\[
		Z = \int_0^\cdot b_s \1_{\{s \leq \rho_R\}} \, ds + \int_0^\cdot \sigma_s \, d W_s. 
		\] 
		As \(Z\) and \(Y\) coincide up to time \(\rho_R\), the corresponding It\^o formulas also coincide up to time \(\rho_R\). Notice that 
		\[
		\| b_s \1_{\{s \leq \rho_R\}} \| \leq \bar{\b} (s, Z_s), 
		\] 
		where \(\bar{\b}(s, y) := \b (y) \1_{C_{R + 1} (0)} (s, y)\). It follows from the hypothesis \(\b \in L_{d + 2, \on{loc}} (\bR^d)\) that \(\bar{\b} \in L_{d + 2} (\bR^{d + 1})\) and consequently, \cite[Assumption~(i), p.~7]{krylov_book_arxiv} is satisfied in view of \cite[Remark~1~(i), p.~7]{krylov_book_arxiv}. Finally, as \cite[Assumption~(ii), p.~7]{krylov_book_arxiv} holds by \cite[Theorem~4.1]{krylov_21}, the claimed It\^o formula for the process \(Z\) follows from \cite[Theorem~3, p.~8]{krylov_book_arxiv}.
	\end{proof} 
	
	\begin{theorem} \label{theo: weak existence} 
		Let \(\mu \colon \bR_+ \times \bR^d \to \bR^d\) and \(\sigma \colon \bR_+ \times \bR^d \to \bR^{d \times d}\) be Borel functions such that 
		\[
		\| \mu (t, x) \| \leq \b (x), \quad \sigma (t, x) \sigma^* (t, x) \in \mathbb{S}_\delta, 
		\]
		for all \((t, x) \in \bR_+ \times \bR^d\), and \(\b \in L_{d + 1} (\bR^d)\).
		Then, for every \(x \in \bR^d\), there exists a weak solution to the SDE
		\begin{align} \label{eq: SDE}
			d Y_t = \mu (t, Y_t) \, dt + \sigma (t, Y_t) \, d W_t, \quad Y_0 = x.
		\end{align} 
	\end{theorem} 
	
	\begin{proof}
		For every \(T > 0\), \cite[Theorem~6.1]{krylov_UMJ_21} yields the existence of a measurable family \(\{ P_{t, x} \colon\) \((t, x) \in [0, T] \times \bR^d \}\), where \(P_{t, x}\) is the law of a solution to the SDE \eqref{eq: SDE} up to the time horizon \(T\), started at time \(t\) at the space point \(x\). Now, one can use \cite[Theorem~6.1.2]{SV} to paste these solutions together, obtaining global existence.
\end{proof} 

\subsection{Conclusion of the proof of Theorem~\ref{theo: parabolic PDE}}
Continuity is implied by Theorem~\ref{theo: joint regu semigroup}. Thus, it suffices to prove the \(L_{d + 1}\)-viscosity sub- and supersolution properties. We restrict our attention to the subsolution property; the supersolution property can be established as in Step 2 from the proof of Theorem~\ref{theo: main continuity}. 
Fix \((t_0, x_0) \in [0, T) \times \bR^d, r \in (0, 1 \wedge (T - t_0))\) and take \(\varphi \in W^{1, 2}_{d + 1, \textup{loc}} ([0, T) \times \bR^{d})\) such that \(u (t_0, x_0) = \varphi (t_0, x_0)\) and \(u < \varphi\) on \((t_0, t_0 + r^2] \times \overline{B}_{r} (x_0)\). 
By part (a) of our Standing Assumption~\ref{SA: main1} and Standing~Assumption~\ref{SA: main 3},
\[
(\lambda, t, x) \mapsto L (\lambda, x, \nabla \varphi (t_0 + t, x), \nabla^2 \varphi (t_0 + t, x)) + g (\lambda, x)
\] 
is continuous in the first and measurable in the last two variables. Hence, as \(\Lambda\) is compact, the measurable maximum theorem \cite[Theorem~18.19]{charalambos2013infinite} yields the existence of a measurable map \((t, x) \mapsto \lambda (t, x)\in \Lambda\) such that 
\begin{align*} 
	L (\lambda (t, x), x, \nabla \varphi (t_0 + t, x)&, \nabla^2 \varphi (t_0 + t, x)) + g (\lambda (t, x), x) 
	\\&= H (x, \nabla \varphi (t_0 + t, x), \nabla^2 \varphi (t_0 + t, x))
\end{align*} 
for all \((t, x) \in [0, r] \times \bR^d\). 
In view of Theorem~\ref{theo: weak existence}, let \(Y\) be a solution to the SDE  
\begin{align*}
	d Y_t = b (\lambda (t \wedge r, Y_t), Y_t) \, dt + \sqrt{a (\lambda (t \wedge r, Y_t), Y_t)}  \, d W_t, \quad Y_0 = x_0, 
\end{align*} 
and let \(P'\) be the law of 
\((Y, \delta_{\lambda(t \wedge r, Y_{t})} (d \lambda) \, dt),\) seen as a probability measure on the canonical product space \((\Omega \times \m, \cF\otimes \mathcal{M})\), which is clearly an element of \(\cK(x_0)\).
By the dynamic programming principle \eqref{eq: para DPP}, Theorem~\ref{theo: ito formula} and \cite[Theorem~4.1]{krylov_21}, with (the \(P'\)-a.s. strictly positive) stopping time \(T'_r := \inf \{ s \geq 0 \colon X_s \not \in B_r (x_0)\} \wedge r^2\),  
\begin{align*}
	0 &= \inf_{Q \in \cK (x_0)} E^Q \Big[ u (t_0 + T'_r, X_{T'_r}) + \int_0^{T'_r} \int g (\lambda, X_s) \, M (ds, d \lambda) \Big] - u (t_0, x_0)
	\\&< E^{P'} \Big[ \varphi (t_0 + T'_r, X_{T'_r}) - \varphi (t_0, x_0) + \int_0^{T'_r} \int g (\lambda, X_s) \, M (ds, d \lambda) \Big] 
	\\&= E^{P'} \Big[ \int_0^{T'_r}  \big(L (\lambda (s, X_s), X_s, \nabla \varphi (t_0 + s, X_s), \nabla^2 \varphi (t_0 + s, X_s)) \\&\hspace{5cm} + \partial_t \varphi (t_0 + s, X_s) + g (\lambda (s, X_s), X_s) \big) \, ds \Big]
	\\&= E^{P'} \Big[ \int_0^{T'_r} \big( H (X_s, \nabla \varphi (t_0 + s, X_s), \nabla^2 \varphi (t_0+s, X_s)) + \partial_t \varphi (t_0 + s, X_s) \big) \, ds \Big]
	\\&\leq C\, \| (H (\, \cdot \,, \nabla \varphi, \nabla^2 \varphi) + \partial_t \varphi)_+ \|_{L_{d + 1} ([t_0, t_0 + r^2) \times B_r (x_0))}, \quad y_+ := \max \{y, 0\}, \phantom \int 
	\\&\leq C\, \underset{(s, z) \in [t_0, t_0 + r^2) \times B_r (x_0)}{\on{ess \, sup}} \, (H (z, \nabla \varphi (s, z), \nabla^2 \varphi (s, z)) + \partial_t \varphi (s, z))_+, \phantom \int 
\end{align*}
where the constant \(C > 0\) depends on \(d, \delta, \|\b\|_{L_{d} (\bR^d)}\) and \(T\), and it changed from the second to last line.
As in Step~1 of the proof of Theorem~\ref{theo: main continuity}, this entails
\[
\underset{(s, z) \in [t_0, t_0 + r^2) \times B_r (x_0)}{\on{ess \, sup}} \, (H (z, \nabla \varphi (s, z), \nabla^2 \varphi (s, z)) + \partial_t \varphi (s, z)) > 0, 
\] 
and consequently, 
	\[
	\lim_{r \searrow 0} \underset{(s, z) \in [t_0, t_0 + r^2) \times B_r (x_0)}{\on{ess \, sup}} \, (H (z, \nabla \varphi (s, z), \nabla^2 \varphi (s, z)) + \partial_t \varphi (s, z)) \geq 0.
	\]
This proves that \(u\) is an \(L_{d + 1}\)-viscosity subsolution (see Lemma~\ref{lem: strict maximum parabolic} in Appendix~\ref{app: recap parabolic}).

\smallskip
As mentioned above, the \(L_{d + 1}\)-viscosity supersolution property follows by the same arguments as Step 2 from the proof of Theorem~\ref{theo: main continuity}. We omit the details here for brevity. \qed 

\section*{Acknowledgments}
The author is very grateful to two anonymous referees for numerous remarks and suggestions that helped to improve the presentation of the paper. Moreover, he is very thankful to Andrzej {\'S}wi{\k{e}}ch, whose comments on a previous version of the paper revealed a missing assumption in the parabolic setting. Finally, the author acknowledges the use of Claude (Opus 5.5, Anthropic) for a final search for typographical and mathematical errors in v5 of the paper right before its upload to the arXiv. Besides typographical errors and some minor gaps, this check revealed a mistake in the original proof of Lemma~\ref{lem: st cont}~(i), and Claude suggested the barrier argument given in the present version. The author has verified and revised the suggestion.

\appendix
	\section{A short recap on \texorpdfstring{\(L_p\)}{Lp}-viscosity theory} \label{app: recap}
	To keep this paper as self-contained as possible, we recall the basic definitions and a few important results concerning \(L_p\)-viscosity solutions to certain nonlinear PDEs with measurable coefficients.
	The concept was introduced in the seminal paper~\cite{CCKS_96}. Here, we use results from the recent works~\cite{KS_22,KS_26}, and we distinguish between the elliptic and the parabolic case. 
	
	\subsection{The elliptic setting} \label{app: recap elliptic}
	Let \(D \subset \bR^d\) be a bounded domain. 
	In this section we consider elliptic nonlinear PDEs of the form 
	\begin{align} \label{eq: PDE appendix}
		F (x, \nabla u (x), \nabla^2 u (x)) = f (x), 
	\end{align}
	where \(F \colon D \times \bR^d \times \mathbb{S}^d \to \bR\), with \(\mathbb{S}^d\) denoting the space of all real symmetric \(d \times d\) matrices, satisfies the following assumptions:
	\begin{enumerate}
		\item[(a)] \(F\) is measurable.
		\item[(b)] There exists a Lebesgue null set \(N \subset D\) such that, for all \(x \in D \setminus N, p, q \in \mathbb{R}^d\) and \(A, B \in \mathbb{S}^d\), 
		\begin{align*}
			\mathcal{P}^- (A - B) &- \b (x) \| p - q \| \\&\leq F (x, p, A) - F (x, q, B) \\&\qquad\qquad\quad\ \, \, \leq \mathcal{P}^+ (A - B) + \b (x) \| p - q \|, 
		\end{align*}
		where \(\b \in L_d (\bR^d)\) is non-negative and, for \(0 < \lambda \leq \Lambda < \infty\), 
		\[
		\mathcal{P}^- (A) := - \Lambda \on{tr} [ A^+ ] + \lambda \on{tr} [ A^- ], \quad \mathcal{P}^+ (A) := - \lambda \on{tr} [ A^+ ] + \Lambda \on{tr} [ A^- ], 
		\] 
		with \(A^+\) and \(A^-\) being the positive and negative parts of \(A\), respectively.
		\item[(c)] \(F (x, 0, 0) = 0\) for all \(x \in D\).
	\end{enumerate}
	
	\medskip
	For this elliptic setting, \(L_p\)-viscosity solutions are defined as follows.
	
	\begin{definition} \label{def: VS elliptic}
		A continuous function \(u \colon \oD \to \bR\) is called an {\em \(L_p\)-viscosity subsolution} (resp. {\em \(L_p\)-viscosity supersolution}) to \eqref{eq: PDE appendix}
		if for every \(\varphi \in W^{2}_{p, \textup{loc}} (D)\) such that \(u - \varphi\) has a local maximum (resp., a local minimum) at \(x_0 \in D\), we have 
		\begin{align*}
			&\lim\nolimits_{r \searrow 0} \on{ess \, inf}_{x \in B_r (x_0)} \, (F (x, \nabla \varphi (x), \nabla^2 \varphi (x)) - f(x)) \leq 0 \\
			(\text{resp. }\hspace{-0.05cm}&\lim\nolimits_{r \searrow 0} \on{ess \, sup}_{x \in B_r (x_0)} \,
			(F (x, \nabla \varphi (x), \nabla^2 \varphi (x)) - f(x)) \geq 0).
		\end{align*} 
		The function \(u\) is an {\em \(L_p\)-viscosity solution} to \eqref{eq: PDE appendix} if it is both an \(L_p\)-viscosity sub- and supersolution.
	\end{definition}
	
	\begin{remark}
		Definition~\ref{def: VS elliptic} follows the usual convention that subsolutions (resp., supersolutions) are connected to a \(\liminf\) (resp., \(\limsup\)) condition. This convention is switched for our main equation \eqref{eq: PDE} due to the negative sign. To be more specific, for \(u\) to be a subsolution of \eqref{eq: PDE}, one has to show that 
		\[
		\lim\nolimits_{r \searrow 0} \on{ess \, sup}_{x \in B_r (x_0)} \, H (x, \nabla \varphi (x), \nabla^2 \varphi (x)) \geq 0.
		\] 
		A similar statement holds for the supersolution property.
	\end{remark}

\begin{remark} \label{rem: HJB appendix form}
The HJB equation \eqref{eq: PDE} is of the form \eqref{eq: PDE appendix} with
\[
F (x, p, A) := H (x, 0, 0) - H (x, p, A), \quad f (x) := H (x, 0, 0) = \inf_{\lambda \in \Lambda} g (\lambda, x).
\]
Indeed, as \(H (x, \cdot\,, \cdot\,)\) is an infimum of affine functions, \(F (x, p, A) - F (x, q, B) = H (x, q, B) - H (x, p, A)\) lies between the infimum and the supremum over \(\lambda \in \Lambda\) of \(- \tfrac{1}{2} \on{tr} [ a (\lambda, x) (A - B) ] + \langle b (\lambda, x), q - p \rangle\). As the eigenvalues of \(\tfrac{1}{2} a (\lambda, x)\) lie in \([\delta / 2, 1 / (2 \delta)]\), the structure conditions (a)--(c) hold with the ellipticity constants \(\delta/2\) and \(1/(2\delta)\) and with \(\b\) from Standing Assumption~\ref{SA: main1}~(c). The same applies to the parabolic equation \eqref{eq: parabolic PDE} and Appendix~\ref{app: recap parabolic} with \(\gamma = \b\).
\end{remark}
	
	Throughout, \(p > d/2\) and \(f \in L_p (D)\). The requirement \(p > d / 2\) ensures that the test functions from \(W^{2}_{p, \textup{loc}} (D)\) are continuous and that their derivatives can be defined in a pointwise (and a distributional) sense (see p.~371 in \cite{CCKS_96}).
	
	The classical definition of viscosity solutions (also often called \(C\)-viscosity solutions) uses the smaller class \(C^2 (D)\) as test functions. It is therefore clear that any \(L_p\)-viscosity solution is also a \(C\)-viscosity solution. In case the coefficients are continuous, any \(C\)-viscosity solution is also an \(L_p\)-viscosity solution under the additional assumptions of \cite[Proposition~2.9]{CCKS_96}. 
	
	We highlight the regularity of the solution: \(L_p\)-sub- and \(L_p\)-supersolutions are both assumed to be continuous, while \(C\)-viscosity subsolutions (resp., supersolutions) are typically only assumed to be upper (resp., lower) semicontinuous. 
	
	The following result shows that it suffices to consider strict extrema. It is a restatement of \cite[Proposition~4.1]{KS_22}.
	\begin{lemma}  \label{lem: strict maximum}
		Let \(\bar{p} = \bar{p} (d, \lambda, \Lambda, \| \b \|_{L_d (\bR^d)}) \in (0, d)\) be the constant from \cite[Theorem~3.1]{KS_22}. Assume that \(f \in L_p (D)\) for \(p \in (\bar{p}, d)\). Then, Definition~\ref{def: VS elliptic} is equivalent to the same definition where maxima and minima are strict.
	\end{lemma} 
	
	The examples in \cite[Remark~1.3]{krylov_CPDE_20} show that structural assumptions on \(F\), such as the Lipschitz condition in the gradient variable from (b), are crucial for the duality between existence of Sobolev solutions and uniqueness of \(L_p\)-viscosity solutions.
	The next result is a restatement of \cite[Proposition~4.2]{KS_22}.
	
	\begin{lemma}
		With \(\bar{p}\) as in Lemma~\ref{lem: strict maximum}, let \(\bar{p} < p_1 < p_2 < d\) and \(f \in L_{p_2} (D)\). If \(u\) is an \(L_{p_2}\)-viscosity subsolution (resp., supersolution) to \eqref{eq: PDE appendix}, then it is also an \(L_{p_1}\)-viscosity subsolution (resp., supersolution).
	\end{lemma} 
	
	The next result shows that the existence of Sobolev solutions entails uniqueness among viscosity solutions. It is a restatement of \cite[Proposition~4.5]{KS_22}.
	
	\begin{lemma} \label{lem: uniqueness from Sobolev solution}
		With \(\bar{p}\) as in Lemma~\ref{lem: strict maximum}, let \(f \in L_p (D)\) for \(p > \bar{p}\). If \(u \in W^2_{p, \on{loc}} (D) \cap C (\overline{D})\) is a Sobolev solution to \eqref{eq: PDE appendix}, i.e., the equation holds Lebesgue a.e., then \(u\) is the unique \(L_p\)-viscosity solution \(v\) to \eqref{eq: PDE appendix} with \(v = u\) on \(\partial D\).
	\end{lemma} 
	
	Finally, we recall the following stability result from \cite[Proposition~4.6]{KS_22}.
	
	\begin{lemma} \label{lem: stability appendix}
		Let \(F_k \colon D \times \bR^d \times \mathbb{S}^d \to \bR, k = 1, 2, \dots,\) be Hamiltonians that satisfy the same assumptions as \(F\) for the same drift function \(\b\) and the same ellipticity constants \(\lambda, \Lambda\). Let \(p \in (\bar{p}, d)\) and \(f, f_k \in L_p (D)\). Suppose that \(u_k\) are \(L_p\)-viscosity solutions to 
		\[
		F_k (x, \nabla u_k (x), \nabla^2 u_k (x)) = f_k (x).
		\] 
		Assume that \(u \in C (D)\) is such that, for every \(B_r (x) \subset \subset D\), \(u_k \to u\) uniformly on \(B_r (x)\), and for \(\phi \in W^2_p (B_r (x))\)
		\[
		\| G (\phi) - G_k (\phi) \|_{L_p (B_r (x))} \to 0, \quad k \to \infty, 
		\] 
		where 
		\begin{align*}
			G_k (\phi) (x) &:= F_k (x, \nabla \phi (x), \nabla^2 \phi (x)) - f_k (x),
			\\
			G (\phi) (x) &:= F (x, \nabla \phi (x), \nabla^2 \phi (x)) - f (x).
		\end{align*} 
		Then, \(u\) is an \(L_p\)-viscosity solution to \eqref{eq: PDE appendix}.
	\end{lemma}

	\subsection{The parabolic setting} \label{app: recap parabolic}
	In this section we comment on the parabolic setting, considering nonlinear PDEs of the form 
	\begin{align} \label{eq: PDE parabolic appendix}
		- \partial_t u (t, x) + F (t, x, \nabla u (t, x), \nabla^2 u (t, x)) = f (t, x), 
	\end{align}
	where \(F \colon [0, T) \times \bR^d \times \bR^d \times \mathbb{S}^d \to \bR\) satisfies the following assumptions:
	\begin{enumerate}
		\item[(a)] \(F\) is measurable.
		\item[(b)] There exists a Lebesgue null set \(N \subset [0, T) \times \bR^d\) such that, for all \((t, x) \in ([0, T) \times \bR^d) \setminus N\), \(p, q \in \bR^d\) and \(A, B \in \mathbb{S}^d\), 
		\begin{align*}
			\mathcal{P}^- (A - B) &- \gamma (t, x) \| p - q \| \\&\leq F (t, x, p, A) - F (t, x, q, B) \\&\qquad\qquad\quad\ \, \, \leq \mathcal{P}^+ (A - B) + \gamma (t, x) \| p - q \|, 
		\end{align*}
		where \(\gamma \in L_{d + 2, \on{loc}} ([0, T] \times \bR^d)\) is non-negative and \(\mathcal{P}^-, \mathcal{P}^+\) are as in Appendix~\ref{app: recap elliptic}.
		\item[(c)] \(F (t, x, 0, 0) = 0\) for all \((t, x) \in [0, T) \times \bR^d\).
	\end{enumerate}
	
	\medskip 
	
	The parabolic definition of an \(L_{d + 1}\)-viscosity solution is given by the following.
	
	\begin{definition} \label{def: parabolic viscosity solution} 
		A continuous function \(u \colon [0, T] \times \bR^d \to \bR\) is called an {\em \(L_{d + 1}\)-viscosity subsolution} (resp. {\em \(L_{d + 1}\)-viscosity supersolution}) to \eqref{eq: PDE parabolic appendix}
		if for every \(\varphi \in W^{1, 2}_{d + 1, \textup{loc}} ([0, T) \times \bR^d)\) such that \(u - \varphi\) has a local maximum (resp., a local minimum) at \((t_0, x_0) \in [0, T) \times \bR^d\), we have 
		\begin{align*}
			&\lim\nolimits_{r \searrow 0} \on{ess \, inf}\nolimits_{(t, x) \in [t_0, t_0 + r^2) \times B_r (x_0)} (- \partial_t \varphi (t, x) \\&\hspace{4.5cm}+ F (t, x, \nabla \varphi (t, x), \nabla^2 \varphi (t, x)) - f (t, x)) \leq 0 \\
			(\text{resp. }\hspace{-0.05cm}&\lim\nolimits_{r \searrow 0}\on{ess \, sup}\nolimits_{(t, x) \in [t_0, t_0 + r^2) \times B_r (x_0)} (- \partial_t \varphi (t, x) \\&\hspace{4.5cm}+ F (t, x, \nabla \varphi (t, x), \nabla^2 \varphi (t, x)) - f(t, x)) \geq 0).
		\end{align*} 
		The function \(u\) is an {\em \(L_{d + 1}\)-viscosity solution} to \eqref{eq: PDE parabolic appendix} if it is both an \(L_{d + 1}\)-viscosity sub- and supersolution.
	\end{definition}
	
	We recall an analog of Lemma~\ref{lem: strict maximum}, which is used in the main body of the paper. For more results comparable to those in Appendix~\ref{app: recap elliptic}, see the recent paper~\cite{KS_26}. The following is due to \cite[Proposition~3.5]{KS_26}. 
	
	\begin{lemma} \label{lem: strict maximum parabolic}
		Let \(\bar{p} = \bar{p} (d, \lambda, \Lambda) \in (d/ 2+ 1, d + 1)\) be the constant from \cite{KS_26}, denoted by \(p_1\) on p.~3 there, and take \(f \in L_{d + 1, \on{loc}} ([0, T] \times \bR^d)\), noting that \(d + 1 \in (\bar p, d + 2)\). Then, Definition~\ref{def: parabolic viscosity solution} is equivalent to the same definition where the maxima and minima are strict.
	\end{lemma}

\bibliographystyle{abbrv}
\bibliography{references}

\begin{thebibliography}{10}

\bibitem{charalambos2013infinite}
C.~D. Aliprantis and K.~C. Border.
\newblock {\em Infinite dimensional analysis. {A} hitchhiker's guide.}
\newblock Berlin: Springer, 3rd edition, 2006.

\bibitem{BDLM_25}
J.~Blessing, R.~Denk, M.~Kupper, and M.~Nendel.
\newblock Convex monotone semigroups and their generators with respect to
  {{\(\Gamma\)}}-convergence.
\newblock {\em J. Funct. Anal.}, 288(8):73, 2025.
\newblock Id/No 110841.

\bibitem{bogachev}
V.~I. Bogachev.
\newblock {\em Measure theory. {Vol}. {I} and {II}}.
\newblock Berlin: Springer, 2007.

\bibitem{CCKS_96}
L.~Caffarelli, M.~G. Crandall, M.~Kocan, and A.~{\'S}wi{\k{e}}ch.
\newblock On viscosity solutions of fully nonlinear equations with measurable
  ingredients.
\newblock {\em Commun. Pure Appl. Math.}, 49(4):365--397, 1996.

\bibitem{CKLS99}
M.~G. Crandall, M.~Kocan, P.~L. Lions, and A.~{\'S}wi{\k{e}}ch.
\newblock Existence results for boundary problems for uniformly elliptic and
  parabolic fully nonlinear equations.
\newblock {\em Electron. J. Differ. Equ.}, 1999:22, 1999.
\newblock Id/No 24.

\bibitem{CKS_00}
M.~G. Crandall, M.~Kocan, and A.~{\'S}wi{\k{e}}ch.
\newblock {{\(L^p\)}}-theory for fully nonlinear uniformly parabolic equations.
\newblock {\em Commun. Partial Differ. Equations}, 25(11-12):1997--2053, 2000.

\bibitem{C_24_JMAA}
D.~Criens.
\newblock Stochastic processes under parameter uncertainty.
\newblock {\em J. Math. Anal. Appl.}, 538(2):45, 2024.
\newblock Id/No 128388.

\bibitem{CN22b}
D.~Criens and L.~Niemann.
\newblock Markov selections and {Feller} properties of nonlinear diffusions.
\newblock {\em Stochastic Processes Appl.}, 173:30, 2024.
\newblock Id/No 104354.

\bibitem{CN_25_JEEQ}
D.~Criens and L.~Niemann.
\newblock Nonlinear semimartingales and {Markov} processes with jumps.
\newblock {\em J. Evol. Equ.}, 25(1):39, 2025.
\newblock Id/No 21.

\bibitem{CN_25_EJP}
D.~Criens and L.~Niemann.
\newblock A stochastic representation theorem for sublinear semigroups with
  non-local generators.
\newblock {\em Electron. J. Probab.}, 30:36, 2025.
\newblock Id/No 77.

\bibitem{F_25}
F.~de~Feo.
\newblock Stochastic optimal control with measurable coefficients and
  applications.
\newblock {\em SIAM J. Control Optim.}, 64(4):2671--2688, 2026.

\bibitem{DM}
C.~Dellacherie and P.-A. Meyer.
\newblock {\em Probabilities and potential. {Transl}. from the {French}},
  volume~29 of {\em North-Holland Math. Stud.}
\newblock Elsevier, Amsterdam, 1978.

\bibitem{DKN_20}
R.~Denk, M.~Kupper, and M.~Nendel.
\newblock A semigroup approach to nonlinear {L{\'e}vy} processes.
\newblock {\em Stochastic Processes Appl.}, 130(3):1616--1642, 2020.

\bibitem{EKNJ88}
N.~El~Karoui, D.~Nguyen, and M.~Jeanblanc-Picqu{\'e}.
\newblock Existence of an optimal {Markovian} filter for the control under
  partial observations.
\newblock {\em SIAM J. Control Optim.}, 26(5):1025--1061, 1988.

\bibitem{nicole1987compactification}
N.~El~Karoui, D.~H. Nguyen, and M.~Jeanblanc-Picqu{\'e}.
\newblock Compactification methods in the control of degenerate diffusions:
  {Existence} of an optimal control.
\newblock {\em Stochastics}, 20:169--219, 1987.

\bibitem{ElKa15}
N.~El~Karoui and X.~Tan.
\newblock Capacities, {measurable} {selection} and {dynamic} {programming}
  {part} {II}: {application} in {stochastic} {control} {problems}.
\newblock {arXiv}:1310.3364 [math.{OC}], 2013.

\bibitem{FS_06}
W.~H. Fleming and H.~M. Soner.
\newblock {\em Controlled {Markov} processes and viscosity solutions},
  volume~25 of {\em Stoch. Model. Appl. Probab.}
\newblock New York, NY: Springer, 2nd edition, 2006.

\bibitem{FN_26}
F.~Fuchs and M.~Nendel.
\newblock Existence of viscosity solutions to abstract {Cauchy} problems via
  nonlinear semigroups.
\newblock {\em Bull. Lond. Math. Soc.}, 58(5):22, 2026.
\newblock Id/No e70386.

\bibitem{GT_01}
D.~Gilbarg and N.~S. Trudinger.
\newblock {\em Elliptic partial differential equations of second order}.
\newblock Class. Math. Berlin: Springer, reprint of the 1998 edition, 2001.

\bibitem{GNR_24}
B.~Goldys, M.~Nendel, and M.~R{\"o}ckner.
\newblock Operator semigroups in the mixed topology and the infinitesimal
  description of {Markov} processes.
\newblock {\em J. Differ. Equations}, 412:23--86, 2024.

\bibitem{Haus86}
U.~G. Haussmann.
\newblock Existence of optimal {Markovian} controls for degenerate diffusions.
\newblock Stochastic differential systems, {Proc}. 3rd {Bad} {Honnef} {Conf}.
  1985, {Lect}. {Notes} {Control} {Inf}. {Sci}. 78, 171-186, 1986.

\bibitem{HausLep90}
U.~G. Haussmann and J.~P. Lepeltier.
\newblock On the existence of optimal controls.
\newblock {\em SIAM J. Control Optim.}, 28(4):851--902, 1990.

\bibitem{himmelberg}
C.~J. Himmelberg.
\newblock Measurable relations.
\newblock {\em Fundam. Math.}, 87:53--72, 1975.

\bibitem{hol16}
J.~Hollender.
\newblock {\em {\em L{\'e}vy-Type Processes under Uncertainty and Related
  Nonlocal Equations}}.
\newblock PhD thesis, TU Dresden, 2016.

\bibitem{ito}
K.~Ito and H.~P. McKean.
\newblock {\em Diffusion processes and their sample paths. 2nd printing,
  corrected}, volume 125 of {\em Grundlehren Math. Wiss.}
\newblock Springer-Verlag, Berlin, 1974.

\bibitem{jacod79}
J.~Jacod.
\newblock {\em Calcul stochastique et probl{\`e}mes de martingales}, volume 714
  of {\em Lect. Notes Math.}
\newblock Springer-Verlag, Berlin, 1979.

\bibitem{jacod80}
J.~Jacod.
\newblock Weak and strong solutions of stochastic differential equations.
\newblock {\em Stochastics}, 3:171--191, 1980.

\bibitem{JS}
J.~Jacod and A.~N. Shiryaev.
\newblock {\em Limit theorems for stochastic processes.}, volume 288 of {\em
  Grundlehren Math. Wiss.}
\newblock Berlin: Springer, 2nd edition, 2003.

\bibitem{JKS_02}
R.~Jensen, M.~Kocan, and A.~{\'S}wi{\k{e}}ch.
\newblock Good and viscosity solutions of fully nonlinear elliptic equations.
\newblock {\em Proc. Am. Math. Soc.}, 130(2):533--542, 2002.

\bibitem{Jen_Sw_05}
R.~Jensen and A.~{\'S}wi{\k{e}}ch.
\newblock Uniqueness and existence of maximal and minimal solutions of fully
  nonlinear elliptic {PDE}.
\newblock {\em Commun. Pure Appl. Anal.}, 4(1):199--207, 2005.

\bibitem{J_96}
R.~R. Jensen.
\newblock Uniformly elliptic {PDEs} with bounded, measurable coefficients.
\newblock {\em J. Fourier Anal. Appl.}, 2(3):237--259, 1996.

\bibitem{JKZ_24}
M.~Johanis, V.~Kry{\v{s}}tof, and L.~Zaj{\'{\i}}{\v{c}}ek.
\newblock On {Whitney}-type extension theorems on {Banach} spaces for {{\(C^{1,
  \omega}, C^{1,+}, C_{\text{loc}}^{1, +}\)}}, and {{\(C_{\text{B}}^{1,
  +}\)}}-smooth functions.
\newblock {\em J. Math. Anal. Appl.}, 532(1):27, 2024.
\newblock Id/No 127976.

\bibitem{Kallenberg}
O.~Kallenberg.
\newblock {\em Foundations of modern probability. {In} 2 volumes}, volume~99 of
  {\em Probab. Theory Stoch. Model.}
\newblock Cham: Springer, 3rd revised and expanded edition, 2021.

\bibitem{KaraShre}
I.~Karatzas and S.~E. Shreve.
\newblock {\em Brownian motion and stochastic calculus.}, volume 113 of {\em
  Grad. Texts Math.}
\newblock New York etc.: Springer-Verlag, 2nd edition, 1991.

\bibitem{KS_09}
S.~Koike and A.~{\'S}wi{\k{e}}ch.
\newblock Weak {Harnack} inequality for fully nonlinear uniformly elliptic
  {PDE} with unbounded ingredients.
\newblock {\em J. Math. Soc. Japan}, 61(3):723--755, 2009.

\bibitem{KS_12}
S.~Koike and A.~{\'S}wi{\k{e}}ch.
\newblock Local maximum principle for {{\(L^p\)}}-viscosity solutions of fully
  nonlinear elliptic {PDEs} with unbounded coefficients.
\newblock {\em Commun. Pure Appl. Anal.}, 11(5):1897--1910, 2012.

\bibitem{KS_22}
S.~Koike and A.~{\'S}wi{\k{e}}ch.
\newblock {Aleksandrov-Bakelman-Pucci} maximum principle for
  {{\(L^p\)}}-viscosity solutions of equations with unbounded terms.
\newblock {\em J. Math. Pures Appl. (9)}, 168:192--212, 2022.

\bibitem{KS_26}
S.~Koike and A.~{\'S}wi{\k{e}}ch.
\newblock Aleksandrov-{Bakelman}-{Pucci}-{Krylov}-{Tso} maximum principle for
  {{\(L^p\)}}-viscosity solutions of parabolic equations with drift term in
  {{\(L^q\)}}.
\newblock {\em Math. Ann.}, 395(2):Paper No.~35, 33~pp., 2026.

\bibitem{KST_19}
S.~Koike, A.~{\'S}wi{\k{e}}ch, and S.~Tateyama.
\newblock Weak {Harnack} inequality for fully nonlinear uniformly parabolic
  equations with unbounded ingredients and applications.
\newblock {\em Nonlinear Anal., Theory Methods Appl., Ser. A, Theory Methods},
  185:264--289, 2019.

\bibitem{krylov_selection}
N.~V. Krylov.
\newblock On the selection of a {Markov} process from a system of processes and
  the construction of quasi-diffusion processes.
\newblock {\em Math. USSR, Izv.}, 7:691--709, 1974.

\bibitem{krylov_80}
N.~V. Krylov.
\newblock {\em Controlled diffusion processes. {Transl}. by {A}. {B}. {Aries}},
  volume~14 of {\em Appl. Math. (N. Y.)}.
\newblock Springer, New York, 1980.

\bibitem{krylov_18}
N.~V. Krylov.
\newblock {\em Sobolev and viscosity solutions for fully nonlinear elliptic and
  parabolic equations}, volume 233 of {\em Math. Surv. Monogr.}
\newblock Providence, RI: American Mathematical Society (AMS), 2018.

\bibitem{krylov_CPDE_20}
N.~V. Krylov.
\newblock Linear and fully nonlinear elliptic equations with {{\(L_d\)}}-drift.
\newblock {\em Commun. Partial Differ. Equations}, 45(12):1778--1798, 2020.

\bibitem{krylov_PTRF_21}
N.~V. Krylov.
\newblock On diffusion processes with drift in {{\(L_d\)}}.
\newblock {\em Probab. Theory Relat. Fields}, 179(1-2):165--199, 2021.

\bibitem{krylov_AOP_21}
N.~V. Krylov.
\newblock On stochastic equations with drift in {{\({L_d}\)}}.
\newblock {\em Ann. Probab.}, 49(5):2371--2398, 2021.

\bibitem{krylov_21}
N.~V. Krylov.
\newblock On stochastic {It{\^o}} processes with drift in {{\(L_d\)}}.
\newblock {\em Stochastic Processes Appl.}, 138:1--25, 2021.

\bibitem{krylov_UMJ_21}
N.~V. Krylov.
\newblock On time inhomogeneous stochastic {It{\^o}} equations with drift in
  {{\(L_{d +1}\)}}.
\newblock {\em Ukr. Math. J.}, 72(9):1420--1444, 2021.

\bibitem{krylov_PA_23}
N.~V. Krylov.
\newblock On diffusion processes with drift in {{\(L_{d+1}\)}}.
\newblock {\em Potential Anal.}, 59(3):1013--1037, 2023.

\bibitem{krylov_book_arxiv}
N.~V. Krylov.
\newblock Stochastic {It{\^o}} {Equations} and {Parabolic} {Second}-{Order}
  {Equations} with singular {Drift}.
\newblock Preprint, {arXiv}:2605.03236v2 [math.{PR}], 2026.

\bibitem{LakSPA15}
D.~Lacker.
\newblock Mean field games via controlled martingale problems: existence of
  {Markovian} equilibria.
\newblock {\em Stochastic Processes Appl.}, 125(7):2856--2894, 2015.

\bibitem{Lions85_2}
P.-L. Lions.
\newblock Optimal control of diffusion processes and
  {Hamilton}-{Jacobi}-{Bellman} equations. {II}: {Viscosity} solutions and
  uniqueness.
\newblock {\em Commun. Partial Differ. Equations}, 8:1229--1276, 1983.

\bibitem{LM80}
P.-L. Lions and J.-L. Menaldi.
\newblock Optimal control of stochastic integrals and
  {Hamilton}-{Jacobi}-{Bellman} equations. {I}.
\newblock {\em SIAM J. Control Optim.}, 20:58--81, 1982.

\bibitem{NL_82}
P.-L. Lions and M.~Nisio.
\newblock A uniqueness result for the semigroup associated with the {Hamilton}-
  {Jacobi}-{Bellman} operator.
\newblock {\em Proc. Japan Acad., Ser. A}, 58:273--276, 1982.

\bibitem{MPT_23}
O.~Menoukeu-Pamen and L.~Tangpi.
\newblock Maximum principle for stochastic control of {SDEs} with measurable
  drifts.
\newblock {\em J. Optim. Theory Appl.}, 197(3):1195--1228, 2023.

\bibitem{zbMATH01140123}
N.~Nadirashvili.
\newblock Nonuniqueness in the martingale problem and the {Dirichlet} problem
  for uniformly elliptic operators.
\newblock {\em Ann. Sc. Norm. Super. Pisa, Cl. Sci., IV. Ser.}, 24(3):537--550,
  1997.

\bibitem{N_09}
K.~Nakagawa.
\newblock Maximum principle for {{\(L^p\)}}-viscosity solutions of fully
  nonlinear equations with unbounded ingredients and superlinear growth terms.
\newblock {\em Adv. Math. Sci. Appl.}, 19(1):89--107, 2009.

\bibitem{NR_21}
M.~Nendel and M.~R{\"o}ckner.
\newblock Upper envelopes of families of {Feller} semigroups and viscosity
  solutions to a class of nonlinear {Cauchy} problems.
\newblock {\em SIAM J. Control Optim.}, 59(6):4400--4428, 2021.

\bibitem{N_76a}
M.~Nisio.
\newblock On a non-linear semi-group attached to stochastic optimal control.
\newblock {\em Publ. Res. Inst. Math. Sci.}, 12:513--537, 1976.

\bibitem{N_76b}
M.~Nisio.
\newblock Some remarks on stochastic optimal controls.
\newblock Proc. 3rd {Japan}-{USSR} {Symp}. {Probab}. {Theory}, {Tashkent} 1975,
  {Lect}. {Notes} {Math}. 550, 446-460, 1976.

\bibitem{N_15}
M.~Nisio.
\newblock {\em Stochastic control theory. {Dynamic} programming principle},
  volume~72 of {\em Probab. Theory Stoch. Model.}
\newblock Tokyo: Springer, 2nd edition, 2015.

\bibitem{N_19}
G.~Nornberg.
\newblock {{\(C^{1,\alpha}\)}} regularity for fully nonlinear elliptic
  equations with superlinear growth in the gradient.
\newblock {\em J. Math. Pures Appl. (9)}, 128:297--329, 2019.

\bibitem{pinsky}
R.~G. Pinsky.
\newblock {\em Positive harmonic functions and diffusion: an integrated
  analytic and probabilistic approach}, volume~45 of {\em Camb. Stud. Adv.
  Math.}
\newblock Cambridge: Cambridge University Press, 1995.

\bibitem{SV}
D.~W. Stroock and S.~R.~S. Varadhan.
\newblock {\em Multidimensional diffusion processes.}
\newblock Class. Math. Berlin: Springer, reprint of the 2nd corrected printing,
  1st edition, 2006.

\bibitem{W_09}
N.~Winter.
\newblock {{\(W^{2,p}\)}}- and {{\(W^{1,p}\)}}-estimates at the boundary for
  solutions of fully nonlinear, uniformly elliptic equations.
\newblock {\em Z. Anal. Anwend.}, 28(2):129--164, 2009.

\end{thebibliography}
\end{document}